\documentclass{article} 
\usepackage[utf8]{inputenc}
\usepackage[T1]{fontenc}
\usepackage{lmodern}
\usepackage[a4paper]{geometry}
\usepackage{babel}
\usepackage{enumitem}
\usepackage{fullpage}
\usepackage{pifont}

\usepackage[usenames, dvipsnames]{xcolor}
\usepackage{tikz}
\usetikzlibrary{patterns}
\usepackage[justification=justified]{caption}
\usepackage{pgfplots}
\pgfplotsset{compat=1.15}
\usepackage{mathrsfs}
\usetikzlibrary{arrows}
\usepackage{amssymb,amsmath,mathtools,amsthm,dsfont}
\usepackage{stmaryrd}

\usepackage[all]{xy}

\newcounter{test}
\newtheorem{thm}{Theorem}[test]

\newtheorem{theoreme}{Théorème}[section]
\newtheorem{theorem}[theoreme]{Theorem}
\newtheorem{prop-f}[theoreme]{Proposition}
\newtheorem{prop}[theoreme]{Proposition}

\newtheorem{lemma}[theoreme]{Lemma}

\newtheorem{definition}[theoreme]{Definition}

\newtheorem{remk}[theoreme]{Remark}

\newcommand{\E}{\mathbb{E}}

\newcommand{\N}{\mathbb{N}}

\renewcommand{\P}{\mathbb{P}}

\newcommand{\R}{\mathbb{R}}

\newcommand{\Z}{\mathbb{Z}}
\renewcommand{\l}{\ell}
\renewcommand{\lll}{\ell^\Lambda}

\renewcommand{\L}{\mathcal{L}}
\newcommand{\Lt}{\tilde{\mathcal{L}}}

\newcommand{\cA}{\mathcal{A}}
\newcommand{\cB}{\mathcal{B}}
\newcommand{\cC}{\mathcal{C}}

\newcommand{\cE}{\mathcal{E}}
\newcommand{\cF}{\mathcal{F}}
\newcommand{\cG}{\mathcal{G}}

\newcommand{\cI}{\mathcal{I}}

\newcommand{\cM}{\mathcal{M}}
\newcommand{\cN}{\mathcal{N}}

\newcommand{\cP}{\mathcal{P}}

\newcommand{\cS}{\mathcal{S}}
\newcommand{\cT}{\mathcal{T}}

\newcommand{\cV}{\mathcal{V}}

\newcommand{\OG}{\overline{\gamma}}

\newcommand{\gu}{\gamma^{*}}

\newcommand{\Tu}{T^{*}}

\newcommand{\ogu}{\overline{\gamma}^{*}}

\newcommand{\fpm}{\vec{\pi}}

\newcommand{\Emid}{E^{*}_{\text{mid}}}
\newcommand{\Einf}{E^{*}_\infty}

\newcommand{\Emod}{E^*_{\text{modif}}}

\newcommand{\Ep}{E^*_+}
\newcommand{\Em}{E^*_-}

\newcommand{\Bus}{B_{1,s,N}}
\newcommand{\Bds}{B_{2,s,N}}

\newcommand{\Bts}{B_{3,s,N}}

\newcommand{\Buz}{B_{1,0,N}}
\newcommand{\Bdz}{B_{2,0,N}}

\newcommand{\Btz}{B_{3,0,N}}

\newcommand{\ulm}{u^\Lambda}
\newcommand{\vlm}{v^\Lambda}
\renewcommand{\r}{t_{\min}}

\newcommand{\Kpar}{K}
\newcommand{\good}{\cG^k}
\newcommand{\tpi}{\tilde{\pi}}

\renewcommand{\1}{\mathds{1}}

\renewcommand{\1}{\mathds{1}}

\renewcommand{\epsilon}{\varepsilon}
\renewcommand{\phi}{\varphi}

\definecolor{qqwuqq}{rgb}{0,0.39215686274509803,0}
\definecolor{ccqqqq}{rgb}{0.8,0,0}

\numberwithin{equation}{section}
\newcounter{numeroexo}

\makeindex
\begin{document}
	
	\title{Geodesics cross any pattern in first-passage percolation without any moment assumption and with possibly infinite passage times}
	\author{Antonin Jacquet\footnote{Institut Denis Poisson, UMR-CNRS 7013, Université de Tours, antonin.jacquet@univ-tours.fr}}
	\date{}
	\maketitle
	\begin{abstract}
		In first-passage percolation, one places nonnegative i.i.d.\ random variables ($T (e)$) on the edges of $\Z^d$. A geodesic is an optimal path for the passage times $T(e)$. 
		Consider a local property of the time environment. We call it a pattern. We investigate the number of times a geodesic crosses a translate of this pattern.
		When we assume that the common distribution of the passage times satisfies a suitable moment assumption, it is shown in [Antonin Jacquet. Geodesics in first-passage percolation cross any pattern, arXiv:2204.02021, 2023] that, apart from an event with exponentially small probability, this number is linear in the distance between the extremities of the geodesic. 
		This paper completes this study by showing that this result remains true when we consider distributions with an unbounded support without any moment assumption or distributions with possibly infinite passage times.
		The techniques of proof differ from the preceding article and rely on a notion of penalized geodesic.
	\end{abstract}
	
	

	\section{Introduction}
	
	\subsection{Settings}\label{Sous-section Settings.}
	
	Fix an integer $d \ge 2$. In this paper, we consider first passage percolation on the hypercubic lattice $\Z^d$. We denote by $0$ the origin of $\Z^d$ and by $\cE$ the set of edges in this lattice. The edges in $\cE$ are those connecting two vertices $x$ and $y$ such that $\|x-y\|_1=1$. The basic random object consists of a family $T=\{T(e) \, : \, e \in \cE\}$ of i.i.d.\ random variables taking values in $[0,\infty]$ and defined on a probability space $(\Omega,\cF,\P)$. The random variable $T(e)$ represents the passage time of the edge $e$. Their common distribution is denoted by $\L$.
	
	A finite path $\pi=(x_0,\dots,x_k)$ is a sequence of adjacent vertices of $\Z^d$, i.e.\ for all $i=0,\dots,k-1$, $\|x_{i+1}-x_i\|_1=1$. We say that $\pi$ goes from $x_0$ to $x_k$. Sometimes we identify a path with the sequence of edges it visits, writing $\pi=(e_1,...,e_k)$ where for $i=1,\dots,k$, $e_i=\{x_{i-1},x_i\}$. We say that $k$ is the length of $\pi$ and we denote $|\pi|=k$. The passage time $T(\pi)$ of a path $\pi=(e_1,\dots,e_k)$ is the sum of the variables $T(e_i)$ for $i=1,\dots,k$. 
	
	We do not exclude the case $\L(\infty) > 0$.
	In this case, there are vertices between which all paths have an infinite passage time.
	Thus, we define the following random set:
	\[\mathfrak{C} = \{(x,y) \in \Z^d \times \Z^d \, : \, \exists \text{ a path $\pi$ from $x$ to $y$ such that $T(\pi)<\infty$}\}.\]
	Throughout the article, we assume that 
	\begin{equation}
		\L([0,\infty))> p_c, \label{eq: hypothèse percolation surcritique des temps finis.}
	\end{equation}
	where $p_c$ denotes the critical probability for Bernoulli bond percolation model on $\Z^d$.
	We refer to \cite{Grimmett} for background on percolation. Say that an edge $e$ is open if its passage time $T(e)$ is finite and closed otherwise. Thanks to \eqref{eq: hypothèse percolation surcritique des temps finis.}, this percolation model is supercritical. Therefore there exists a unique infinite component which we denote by $\cC_\infty$.
	When $\L(\infty)=0$, note that every couple of vertices belongs to $\mathfrak{C}$ and that $\cC_\infty$ is equal to the whole graph.
	
	Now, for two vertices $x$ and $y$, we define the geodesic time
	\begin{equation}
		t(x,y)= \inf \{T(\pi) \, : \, \pi \mbox{ is a path from $x$ to $y$} \}. \label{Définition geodesic time.}
	\end{equation}
	Note that $\mathfrak{C} = \{(x,y) \in \Z^d \times \Z^d \, : \, \text{$t(x,y)$ is finite}\}$.
	A self-avoiding path $\gamma$ from $x$ to $y$ such that $T(\gamma)=t(x,y)$ is called a geodesic between $x$ and $y$.
	
	For the following and for the existence of geodesics, we need some assumptions on $\L$. Let $\r$ denote the minimum of the support of $\L$. We extend a definition introduced in \cite{VdBK}. A distribution $\L$ with support in $[0,\infty]$ is called useful if the following holds: 
	\begin{equation}
		\begin{split}
		\L(\r) < p_c & \mbox{ when } \r=0, \\
		\L(\r) < \overrightarrow{p_c} & \mbox{ when } \r>0,
		\end{split}\label{eq: loi useful.}
	\end{equation}
	where $p_c$ has been introduced above and where $\overrightarrow{p_c}$ is the critical probability for oriented Bernoulli bond percolation on $\Z^d$ (see Section 12.8 in \cite{Grimmett}). Throughout the article, we also assume that $\L$ is useful. 
	
	Geodesics between any pair of vertices belonging to $\mathfrak{C}$ exist with probability one. This is Proposition 4.4 in \cite{50years} when $\L(\infty)=0$ and Proposition \ref{prop: Proposition annexe existence des géodésiques.} in Appendix \ref{Annexe sur l'existence des géodésiques avec des arêtes infinies.} when $\L(\infty)>0$. Thus, geodesics between any pair of vertices belonging to $\cC_\infty$ exist with probability one.
	
	\subsection{Patterns}
	
	For a set $B$ of vertices, we denote by $\partial B$ its boundary, this is the set of vertices which are in $B$ and which are linked by an edge to a vertex which is not in $B$. We make an abuse of notation by saying that an edge $e=\{u,v\}$ belongs to a set of vertices if $u$ and $v$ are in this set.
	
	Let $L_1,\dots,L_d$ be non-negative integers. To avoid trivialities we assume that at least one of them is positive. We fix $\displaystyle \Lambda=\prod_{i=1}^d \{0,\dots,L_i\}$ and two distinct vertices $\ulm$ and $\vlm$ on the boundary of $\Lambda$. These points $\ulm$ and $\vlm$ are called endpoints. Then we fix an event $\cA^\Lambda$, with positive probability, only depending on the passage time of the edges of $\Lambda$. We say that $\mathfrak{P}=(\Lambda,u^\Lambda,v^\Lambda,\cA^\Lambda)$ is a pattern.
	Let $x \in \Z^d$. Define:
	\begin{itemize}
		\item for $y \in \Z^d$, $\theta_x y = y-x$,
		\item for $e=\{u,v\}$ an edge connecting two vertices $u$ and $v$, $\theta_x e = \{\theta_x u, \theta_x v\}$.
	\end{itemize}
	
	Similarly, if $\pi=(x_0,\dots,x_k)$ is a path, we define $\theta_x \pi=(\theta_x x_0,\dots, \theta_x x_k)$. Then $\theta_x T$ denotes the environment $T$ translated by $-x$, i.e. the family of random variables indexed by the edges of $\Z^d$ defined for all $e \in \cE$ by 
	\[
	\left(\theta_x T\right) (e) = T \left( \theta_{-x} e \right).
	\]
	
	Let $\pi$ be a self-avoiding path and $x \in \Z^d$. We say that $x$ satisfies the condition $(\pi ; \mathfrak{P})$ if these two conditions are satisfied:
	\begin{enumerate}
		\item $\theta_x \pi$ visits $u^\Lambda$ and $v^\Lambda$, and the subpath of $\theta_x \pi$ between $u^\Lambda$ and $v^\Lambda$ is entirely contained in $\Lambda$,
		\item $\theta_x T \in \cA^\Lambda$.
	\end{enumerate}
	Note that, if $x$ satisfies the condition $(\pi ; \mathfrak{P})$ when $\pi$ is a geodesic, then the subpath of $\theta_x \pi$ between $u^\Lambda$ and $v^\Lambda$ is one of the optimal paths from $u^\Lambda$ to $v^\Lambda$ entirely contained in $\Lambda$ in the environment $\theta_x T$.	
	When the pattern is given, we also say "$\pi$ takes the pattern in $\theta_{-x}\Lambda$" for "$x$ satisfies the condition $(\pi;\mathfrak{P})$".
	We denote:
	\begin{equation}
		N^\mathfrak{P}(\pi)=\sum_{x \in \Z^d} \1_{\{x \text{ satisfies the condition }(\pi ; \mathfrak{P})\}}. \label{Compteur nombre de motifs empruntés introduction.}	
	\end{equation}
	Note that the number of terms in this sum is actually bounded from above by the number of vertices in $\pi$. If $N^\mathfrak{P}(\pi) \ge 1$, we say that $\pi$ takes the pattern. The aim of the article is to investigate, under reasonable conditions on $\mathfrak{P}$, the behavior of $N^\mathfrak{P}(\gamma)$ for all geodesics $\gamma$ from $0$ to $x$ with $\| x \|_1$ large. The first step is to determine these reasonable conditions, that is why we define the notion of valid patterns. 
	
	\begin{definition}
		Denote by $\{\epsilon_1,\dots,\epsilon_d\}$ the vectors of the canonical basis. An external normal unit vector associated to a vertex $z$ of the boundary of $\Lambda$ is an element $\alpha$ of the set $\{\pm\epsilon_1,\dots,\pm\epsilon_d\}$ such that $z+\alpha$ does not belong to $\Lambda$.
	\end{definition}
	
	\begin{definition}\label{Définition motif valable.}
		We say that a pattern is valid if the following three conditions hold:
		\begin{itemize}
			\item $\cA^\Lambda$ has a positive probability,
			\item when $\cA^\Lambda$ occurs, there exists a path between the two endpoints, entirely contained in $\Lambda$, whose passage time is finite,
			\item one of the following two conditions holds:
			\begin{itemize}
				\item the support of $\L$ is unbounded,
				\item there exist two distinct external normal unit vectors, one associated with $\ulm$ and one associated with $\vlm$.
			\end{itemize}
		\end{itemize}
	\end{definition}

	\begin{remk}
		\,
		\begin{itemize}
			\item The second condition is always satisfied when $\L(\infty)=0$.
			\item The existence of the two distinct vectors in the third condition of Definition \ref{Définition motif valable.} is equivalent to the fact that the endpoints of the pattern belong to two different faces. As explained in Remark 1.3 in \cite{Jacq}, a real obstruction can appear when the support of $\L$ is bounded and this third condition is not satisfied.
		\end{itemize}
	\end{remk}
	
	\subsection{Main result and applications}\label{Sous-section résultat principal et applications.}

	Here is our main result. We assume that one of the following two conditions is satisfied:
	\begin{equation}
		\begin{split}
			& \L(\infty) > 0 \text{ and } \L([0,\infty))>p_c, \\
			\text{or } & \L(\infty) = 0 \text{ and the support of $\L$ is unbounded.}
		\end{split}\label{eq: Hypothèses sur la loi.}
	\end{equation}
	
	\begin{theorem}\label{Théorème à démontrer.}
		Let $\mathfrak{P}=(\Lambda,u^\Lambda,v^\Lambda,\cA^\Lambda)$ be a valid pattern, assume \eqref{eq: Hypothèses sur la loi.} and that $\L$ is useful. Then there exist $\alpha >0$, $\beta_1 >0$ and $\beta_2>0$ such that for all $x \in \Z^d$, \[ \P \left((0,x) \in \mathfrak{C} \text{ and } \exists \mbox{ a geodesic $\gamma$ from $0$ to $x$ such that } N^\mathfrak{P}(\gamma) < \alpha \|x\|_1 \right) \le \beta_1 \mathrm{e}^{-\beta_2 \|x\|_1}. \]
	\end{theorem}

	In \cite{Jacq} we proved the following result.
	
	\begin{thm}[Theorem 1.4 in \cite{Jacq}]\label{Théorème du premier article.}
		Let $\mathfrak{P}=(\Lambda,u^\Lambda,v^\Lambda,\cA^\Lambda)$ be a valid pattern, assume that $\L$ is useful, $\L(\infty)=0$ and one of the following two conditions is satisfied:
		\begin{enumerate}[label=(\Roman*)]
			\item\label{cas : Premier cas traité dans l'article 1.} $\L$ has a bounded support,
			\item\label{cas : Deuxième cas traité dans l'article 1.} 
			$\L$ has an unbounded support and we have
			\begin{equation}
				\E \min \left[T^d_1,\dots,T^d_{2d}\right] < \infty, \label{eq: hypothèse de moment du premier article.}
			\end{equation}
			where $T^d_1,\dots,T^d_{2d}$ are independent with distribution $\L$.
		\end{enumerate}
		Then there exist $\alpha >0$, $\beta_1 >0$ and $\beta_2>0$ such that for all $x \in \Z^d$, \[ \P \left(\exists \mbox{ a geodesic $\gamma$ from $0$ to $x$ such that } N^\mathfrak{P}(\gamma) < \alpha \|x\|_1 \right) \le \beta_1 \mathrm{e}^{-\beta_2 \|x\|_1}. \]
	\end{thm}

	Combining Theorems \ref{Théorème à démontrer.} and \ref{Théorème du premier article.} we immediately get:
	
	\begin{theorem}\label{Théorème général.}
		Let $\mathfrak{P}=(\Lambda,u^\Lambda,v^\Lambda,\cA^\Lambda)$ be a valid pattern, assume that $\L$ is useful and $\L([0,\infty)) > p_c$. Then there exist $\alpha >0$, $\beta_1 >0$ and $\beta_2>0$ such that for all $x \in \Z^d$, \[ \P \left((0,x) \in \mathfrak{C} \text{ and } \exists \mbox{ a geodesic $\gamma$ from $0$ to $x$ such that } N^\mathfrak{P}(\gamma) < \alpha \|x\|_1 \right) \le \beta_1 \mathrm{e}^{-\beta_2 \|x\|_1}. \]
	\end{theorem}
	
	Theorem \ref{Théorème général.} is a generalization of Theorem 2.3 in \cite{AndjelVares} (stated below as Theorem \ref{Théorème 2.3 de Andjel/Vares.}) since, to get this result, we can take the pattern (reduced to one edge) $\mathfrak{P}=(\{\ulm,\vlm\},\ulm,\vlm,\cA^\Lambda)$ where $\ulm=(0,\dots,0)$, $\vlm=(1,0,\dots,0)$ and $\cA^\Lambda$ is the event on which the passage time of the only edge of the pattern is greater than $M$. 
	
	\begin{thm}[Theorem 2.3 in \cite{AndjelVares}]\label{Théorème 2.3 de Andjel/Vares.}
		Let $\L$ be a useful distribution on $[0,+\infty)$ with unbounded support. Then, for each $M$ positive there exists $\epsilon=\epsilon(M)>0$ and $\alpha=\alpha(M)>0$ so that for all $x$, we have
		\begin{equation}
			\P \left(\exists \text{ geodesic $\pi$ from $0$ to $x$ such that } \sum_{e\in\pi} \1_{T(e) \ge M} \le \alpha \|x\|_1 \right) \le \mathrm{e}^{-\epsilon \|x\|_1}. \label{Théorème Andjel/Vares}
		\end{equation}		
	\end{thm}
	
	The proof of Theorem \ref{Théorème à démontrer.} is given in Section \ref{Section cas général.}. It is partly inspired by the proof of Theorem 2.3 in \cite{AndjelVares}.
	The proofs of Theorems \ref{Théorème à démontrer.} and \ref{Théorème du premier article.} are independent and the only intersection between these two theorems is the case \ref{cas : Deuxième cas traité dans l'article 1.} above. Theorem \ref{Théorème général.} is an extension of Theorem 1.4 in \cite{Jacq}. We refer to \cite{Jacq} for an account of the history of such results and for applications. As an example of application we prove the following result, which is a generalization of the main result of \cite{VdBK}.
	
	\paragraph*{The van den Berg-Kesten comparison principle without any moment assumption.}
	
	Let $\L$ and $\Lt$ be two distributions taking values in $[0,\infty]$ such that:
	\begin{enumerate}[label=($\mathcal{H}$\arabic*)]
		\item\label{enum: VdB-K 4.} $\L$ is useful,
		\item\label{enum: VdB-K 1.} $\L([0,\infty)) > p_c$ and $\Lt([0,\infty)) > p_c$,
		\item\label{enum: VdB-K 2.} $\L \ne \Lt$,
		\item\label{enum: VdB-K 3.} there exists a couple of random variables $\tau$, $\tilde{\tau}$ on some probability space, with marginal distributions $\L$ and $\Lt$, respectively, and satisfying 
		\begin{equation}
			\E [\tilde{\tau} | \tau] \le \tau. \label{eq: extension de VdB-K-1.}
		\end{equation}
	\end{enumerate}
	
	We consider a family $T=\{T(e) \, : \, e \in \cE\}$ of i.i.d.\ random variables with distribution $\L$ and another family $\tilde{T}=\{\tilde{T}(e) \, : \, e \in \cE\}$ of i.i.d\ random variables with distribution $\Lt$. The geodesic time defined at \eqref{Définition geodesic time.} is denoted by $t$ in the environment $T$ and by $\tilde{t}$ in the environment $\tilde{T}$.
	With these assumptions, a time constant for each distribution can be defined thanks to \cite{CerfTheret}. We refer to \cite{CerfTheret} for an extensive account. Here we recall what we need for our purpose.
	By \ref{enum: VdB-K 1.}, there exists $M \in \R$ such that 
	\begin{equation}
		\L([0,M]) > p_c \text{ and } \Lt([0,M]) > p_c.\label{eq: extension de VdB-K-2.}
	\end{equation} Fix such a $M$. Let $\cC_M$ (resp. $\tilde{\cC}_M$) be the infinite cluster for the Bernoulli percolation $(\1_{\{T(e) \le M\}}, \, e \in \cE)$ (resp. $(\1_{\{\tilde{T}(e) \le M\}}, \, e \in \cE)$) which exists and is unique a.s. To any $x \in \R^d$, we associate a random point $\phi(x)$ (resp. $\tilde{\phi}(x)$) in $\cC_M$ (resp. in $\tilde{\cC}_M$) such that $\|x-\phi(x)\|_1$ (resp.  $\|x-\tilde{\phi}(x)\|_1$) is minimal, with a deterministic rule to break ties.

	Theorem 1 in \cite{CerfTheret} gives the existence of two deterministic functions $\mu \, : \, \R^d \to [0,\infty)$ and $\tilde{\mu} \, : \, \R^d \to [0,\infty)$ such that 
	\begin{equation}
		\forall x \in \Z^d, \, \lim\limits_{n \to \infty} \frac{t(\phi(0),\phi(nx))}{n} = \mu(x) \text{ a.s. and in $L^1$, and } \lim\limits_{n \to \infty} \frac{\tilde{t}(\tilde{\phi}(0),\tilde{\phi}(nx))}{n} = \tilde{\mu}(x) \text{ a.s. and in $L^1$}.\label{eq: extension de VdB-K-3.}
	\end{equation}	
	Theorem 4 in \cite{CerfTheret} ensures that the functions $\mu$ and $\tilde{\mu}$ do not depend on the choice of the constant $M$ satisfying \eqref{eq: extension de VdB-K-2.}.
	Furthermore, when  
	\begin{equation}
		\E \min[\tau_1,\dots,\tau_{2d}] < \infty,\label{eq: extension vdbk nouvelles corrections 1.}
	\end{equation}
	where $\tau_1,\dots,\tau_{2d}$ are i.i.d.\ copies of $\tau$, Theorem 4 in \cite{CerfTheret} also ensures that for all $x \in \R^d$, 
	\begin{equation}
		\lim\limits_{n \to \infty} \frac{t(0,\lfloor nx \rfloor)}{n} = \mu(x) \text{ a.s. and in $L^1$.}\label{eq: extension vdbk nouvelles corrections 2}
	\end{equation}
	This is the usual definition of the time constant. We refer to Theorem 2.18 in \cite{SaintFlourKesten} and Section 2.1 in \cite{50years} for more details on the result \eqref{eq: extension vdbk nouvelles corrections 2}. 
	The same holds for the environment $\tilde{T}$ if \eqref{eq: extension vdbk nouvelles corrections 1.} holds for $2d$ i.i.d.\ copies of $\tilde{\tau}$. 

	\begin{remk}
		We warn the reader that notations $\tilde{T}$ and $\tilde{\mu}$ are used in \cite{CerfTheret} with a different meaning. We refer in particular to Remark 2 in \cite{CerfTheret} for explanations.
	\end{remk}

	We can now state the van den Berg-Kesten comparison principle for these time constants. 
	
	\begin{theorem}[Extension of the van den Berg-Kesten comparison principle]\label{Théorème VdB-K.}
		Let $\L$ and $\Lt$ be two distributions taking values in $[0,\infty]$ satisfying \ref{enum: VdB-K 4.}, \ref{enum: VdB-K 1.}, \ref{enum: VdB-K 2.} and \ref{enum: VdB-K 3.}. For all $x \in \Z^d$ such that $x \ne 0$, 
		\begin{equation}
			\tilde{\mu}(x) < \mu(x).\label{eq: nouvelles modifs vdbk intro}
		\end{equation}
	\end{theorem}
	
	The proof of Theorem \ref{Théorème VdB-K.} is given in Section \ref{Section preuve VdB-K.}.
	In \cite{VdBK}, van den Berg and Kesten prove the following theorem.
	
	\begin{thm}[Theorem 2.9 in \cite{VdBK}]\label{Théorème orginal de VDBK.}
		Let $\L$ and $\tilde{\L}$ be two distributions taking values in $[0,\infty)$, having a finite first moment, satisfying\footnote{It is stated in \cite{VdBK} with the definition of a distribution more variable than another, but Theorem 2.6 in \cite{VdBK} ensures that the fact that $\tilde{\L}$ is more variable than $\L$ is equivalent to \ref{enum: VdB-K 3.} when $\L$ and $\tilde{\L}$ have a finite first moment.} \ref{enum: VdB-K 4.}, \ref{enum: VdB-K 2.} and \ref{enum: VdB-K 3.}. Then, \[\tilde{\mu}(\epsilon_1) < \mu(\epsilon_1).\]
	\end{thm}

	Theorem \ref{Théorème VdB-K.} is an extension of Theorem \ref{Théorème orginal de VDBK.}. With Theorem 1.2 in \cite{MarchandStrictInqualities}, Marchand extends Theorem \ref{Théorème orginal de VDBK.} in another direction.
	
	\begin{thm}[Theorem 1.2 in \cite{MarchandStrictInqualities}]\label{Théorème Régine Marchand.}
		Assume that $d=2$ and let $\L$ and $\tilde{\L}$ be two distributions taking values in $[0,\infty)$, such that $\L(0) < p_c$, satisfying\footnote{It is also stated in \cite{MarchandStrictInqualities} with the definition of a distribution more variable than another, but Lemma 6.1 in \cite{MarchandStrictInqualities} ensures that this definition is also equivalent to \ref{enum: VdB-K 3.} when $\L$ and $\tilde{\L}$ takes value in $[0,\infty)$.} \ref{enum: VdB-K 2.} and \ref{enum: VdB-K 3.}. Then, \[\tilde{\mu}(\epsilon_1) < \mu(\epsilon_1).\]
	\end{thm}

	In dimension 2, the result of Marchand is stronger than Theorem \ref{Théorème orginal de VDBK.} on two aspects: on the one hand, there is no moment assumption and on the other hand, the condition 
	\begin{equation}
		\L(\r) < \overrightarrow{p_c} \text{ when } \r>0 \label{eq: condition enlevée par R. Marchand.}
	\end{equation} 
	is removed.
	When $\r > 0$ and $\L(\r) > \overrightarrow{p_c}$, the problem involves oriented percolation, where the open edges correspond to those with the smallest time values. In this context, the largest part of each geodesic linking the origin to a distant point within the cone of percolation is a directed path made of minimal edges, highlighting a distinct behavior. 
	
	For a point $x$ inside this cone, $\mu(x) = \r \|x\|_1$. Moreover, when $\tilde{t}_{\min} = \r$ and the \ref{enum: VdB-K 1.} condition is met, $\tilde{\mu}(x) = \r \|x\|_1$ as well. Notably, \eqref{eq: nouvelles modifs vdbk intro} does not apply to such $x$ values. Establishing \eqref{eq: nouvelles modifs vdbk intro} for $x$ outside the cone, like $\epsilon_1$, requires specific arguments, in particular large deviations for supercritical oriented percolation. We have opted not to explore this case in this article.
	We refer to \cite{MarchandStrictInqualities}, and more specifically to Theorem 1.5 in \cite{MarchandStrictInqualities}, for further explanations.

	\subsection{Sketch of the proof}\label{Sous-section sketch of the proof.}
	
	In this section, we give an informal sketch of the proof of Theorem \ref{Théorème à démontrer.}. Fix a pattern $\mathfrak{P}$ and $x \in \Z^d$ with $\|x\|$ large. Consider the event:
	\[
	\cM = \{\text{$(0,x) \in \mathfrak{C}$ and there exists a geodesic from }0\text{ to }x\text{ which does not take the pattern}\}.
	\]
	The aim is to prove that $\cM$ has a probability small enough in $\|x\|$.
	More precisely, we want to prove
	\begin{equation}
		\P(\cM) \ll \frac{1}{\|x\|^{d-1}}. \label{eq: nouveau sketch of proof 4.}
	\end{equation}
	From this result, by a standard re-normalization argument, we easily get Theorem \ref{Théorème à démontrer.} (see Proposition \ref{Gros théorème à démontrer, un seul motif.} in Section \ref{Sous-section organisation de la preuve.} for a formal statement of \eqref{eq: nouveau sketch of proof 4.}). 
	
	\paragraph*{General idea.}
	
	As in \cite{Jacq}, to get \eqref{eq: nouveau sketch of proof 4.}, the idea is to define a suitable event $\cG$ and a suitable sequence of events $\cM(\ell)$ for $0 \le \ell \le q$ such that, for some positive constant $c<1$,
	\begin{enumerate}
		\item $q \ge c\|x\|$,
		\item $\cM \subset \cM(q) \cup \cG^c$ where $\displaystyle \P(\cG^c) \ll \frac{1}{\|x\|^{d-1}}$,
		\item for all $\ell \ge 1$, 
		\begin{equation}
			\P(\cM(\ell)) \le c \P(\cM(\ell-1)).\label{eq: nouveau sketch of proof 1.}
		\end{equation}
	\end{enumerate}
	If the above holds, we get $\P(\cM) \le c^{c\|x\|_1} + \P(\cG^c)$, which allows us to conclude.
	
	\paragraph*{Penalized geodesics.}
	
	We now introduce the notion of penalized path. This is an idea which comes from the article \cite{AndjelVares} by Andjel and Vares in their proof of Theorem \ref{Théorème Andjel/Vares}. A penalized path is a path which does not take the pattern. In other words, this is a path $\pi$ such that $N^\mathfrak{P}(\pi)=0$. This allows us to define the penalized passage time for every $z \in \Z^d$:
	\[t_P(0,z) = \inf \{T(\pi) \, : \, \pi \text{ is a penalized path from $0$ to $z$}\},\] with the convention $\inf \emptyset = \infty$. Then, for every $z \in \Z^d$, if it exists, a penalized geodesic from $0$ to $z$ is a penalized path $\gamma$ from $0$ to $z$ such that $T(\gamma)=t_P(0,z)$. 
	With these definitions, we have 
	\begin{equation}
		\cM \subset \{\text{$(0,x) \in \mathfrak{C}$ and } t_P(0,x)=t(0,x)\}. \label{eq: nouveau sketch of proof 4 2.}
	\end{equation}

	\paragraph*{Shortcuts.}
	
	A good way to get that the event $\{\text{$(0,x) \in \mathfrak{C}$ and } t_P(0,x)=t(0,x)\}$ does not occur is to prove that a penalized geodesic has a shortcut. The formal definition of a shortcut is given in Section \ref{Sous-section $k$-penalized paths.}. Informally, a shortcut for a penalized geodesic $\gamma$ is a path going from a vertex $u$ of $\gamma$ to another vertex $v$ of $\gamma$ which takes the pattern and which has a passage time lower than the passage time of the subpath of $\gamma$ going from $u$ to $v$. Hence, if a penalized geodesic $\gamma$ from $0$ to $x$ has a shortcut, it implies that there exists a path from $0$ to $x$ which is not penalized and such that its passage time is strictly lower than the passage time of $\gamma$. It gives $t(0,x) < t_P(0,x)$. 

	\paragraph*{Events $\cG$ and $\cM(\l)$.}
	A successful box for a path $\pi$ is a box satisfying one of the following two conditions:
	\begin{itemize}
		\item it is a typical box,
		\item the path $\pi$ has a shortcut taking the pattern inside the box. 
	\end{itemize}
	We say that a box is shortcut-equipped for a path $\pi$ or not shortcut-equipped for $\pi$ depending on whether the second condition is satisfied or not. 
	We define $\cG$ as the event on which $(0,x) \in \mathfrak{C}$ and there exists a penalized geodesic $\pi$ whose passage time is finite and which crosses at least $q$ successful boxes for $\pi$. On this event, we define the selected penalized geodesic denoted by $\gamma$: it is the first (for an arbitrary deterministic order) of the penalized geodesics satisfying the condition which appears in the definition of $\cG$. We define the sequence of successful boxes crossed by $\gamma$ as the sequence of the first $q$ successful boxes crossed by $\gamma$ indexed in the order in which they are crossed by $\gamma$. It allows us to define, for every $\l \in \{1,\dots,q\}$ the event
	\[\cM(\l)=\cG \cap \{\text{the $\l$ first successful boxes of the sequence of $\gamma$ are not shortcut-equipped for $\gamma$}\}.\]
	This gives us $\l$ opportunities to modify the environment in each of these $\l$ typical boxes to create a shortcut for $\gamma$. The aim is now reduced to proving \eqref{eq: nouveau sketch of proof 1.}.
	
	\paragraph*{$k$-boxes.}
	
	We describe a small change of the plan above. This change does not create any complications. In particular, the entire plan describe above works with these new objects. The advantage is to avoid a number of complications, such as, for example, those related to what happens at the boundary of a box when we modify the environment in it.
	
	The idea is to only consider a family of boxes (called the $k$-boxes, see Section \ref{Sous-section $k$-penalized paths.}) that partitions $\Z^d$. We also replace "penalized paths" -the paths which does not take the pattern- by "$k$-penalized paths" -the paths which does not take any pattern contained in a $k$-box-. We similarly replace "penalized geodesics" by "$k$-penalized geodesics" : the geodesics which does not take any pattern contained in a $k$-box. We say "$k$-geodesic" instead of "$k$-penalized geodesic" for short. 

	\paragraph*{Modification and stability.}
	For all $\l$, we have $\cM(\l) \subset \cM(\l-1)$.
	Thus \eqref{eq: nouveau sketch of proof 1.} is equivalent to the existence of a constant $\eta>0$ (by taking $\eta=\frac{1}{c}-1$) such that 
	\begin{equation}
		\P (\cM(\l-1) \setminus \cM(\l)) \ge \eta \P (\cM(\l)). \label{eq: nouveau sketch of proof 2.}
	\end{equation}
	Fix $\l \in \{1,\dots,q\}$ and denote by $B_s$ the $\l$-th successful $k$-box crossed by $\gamma$. The aim is to prove \eqref{eq: nouveau sketch of proof 2.}. The idea is to resample the passage times of edges of $B_s$ in an environment in which $\cM(\l)$ occurs to get a new environment in which $\cM(\l-1) \setminus \cM(\l)$ occurs. When the resampled passage times satisfy good conditions (to be determined), the following properties are satisfied:
	\begin{enumerate}
		\item The event $\cG$ still occurs and the selected $k$-geodesic is still $\gamma$ in the new environment.
		\item The box $B_s$ is shortcut-equipped for $\gamma$ in the new environment.
		\item The sequence of successful boxes crossed by $\gamma$ is the same in the two environments. 
		\item The event $\cM(\l-1) \setminus \cM(\l)$ occurs in the new environment.
	\end{enumerate}
	By the fourth property, we get roughly
	\[\P(\cM(\l)) \P(\text{good conditions on the resampled passage times}) \le \P(\cM(\l-1) \setminus \cM(\l)).\]
	Since $\eta$ is fixed according to the probability of the good conditions on the resampled passage times, which is positively bounded from below independently of the box, we get \eqref{eq: nouveau sketch of proof 2.}. See the proof of Lemma \ref{Lemme majoration de la probabilité de l'événement par lambda puissance Q.} using Lemma \ref{Lemme pour les inégalités avec les indicatrices dans la modification.} in Section \ref{Sous-section réduction dans la preuve.}.

	The third property follows from the first two. Indeed, the box $B_s$ is typical in the first environment (since it is a successful box for $\gamma$ and it is not shortcut-equipped for $\gamma$ as $\cM(\l)$ holds) and is shortcut-equipped for $\gamma$ in the new environment. Furthermore, the other boxes have the same status (successful or not for $\gamma$) in the two environments since the passage times of the edges of the other boxes have not been modified.
	
	The fourth property follows from the first three by similar ideas. We thus see that, in order to get the fourth property, we do not only need to get the second one. We also need the first and the third ones. We call these two additional properties "stability properties". 
	The proof is thus reduced to getting the first two properties.
	
	\paragraph*{Some more details.}
	
	Recall that we assume \eqref{eq: Hypothèses sur la loi.}. There are two cases to be considered differently:
	\begin{enumerate}[label=(INF)]
		\item $\L(\infty)>0$,
	\end{enumerate}
	\begin{enumerate}[label=(FU)]
		\item $\L(\infty)=0$ and the support of $\L$ is unbounded.
	\end{enumerate}
	(INF) stands for "infinite" and (FU) stands for "finite unbounded".
	
	In what follows, when we say "after the modification" or "in the new environment", we mean "in the new environment where passage times of the edges in the box $B_s$ have been resampled and on the event where the resampled passage times satisfy some good properties that we do not explicit here". In this paragraph, we focus on the first property of the previous paragraph. To get it, it is sufficient to prove the following properties:
	\begin{enumerate}[label=(\roman*)]
		\item The path $\gamma$ still has a finite passage time in the new environment. There are no difficulties with this property.
		\item The path $\gamma$ remains a $k$-penalized path in the new environment. Since the passage times of the $k$-boxes different from $B_s$ have not been modified, it is sufficient to prove that $\gamma$ does not take the pattern in $B_s$ in the new environment. This is based on the two following ideas.
		\begin{itemize}
			\item We identify {\em forbidden zones} which are subsets of $B_s$ where $\gamma$ can not go. In the case \ref{c: Case I.} the forbidden zones are simply balls whose edges have infinite passage time. In the case \ref{c: Case III.} we refer to Lemma \ref{Lemme pour les clusters dans les typical boxes dans le cas finin non borné.}. By definition, a typical box possesses many forbidden zones (see the third item of the definition of a typical box in Section \ref{Sous-section boîtes typiques.}).
			\item We make sure that, after a successful modification, there is a {\em unique} pattern inside $B_s$. The uniqueness is ensured by replacing the original pattern by a new larger pattern, containing the original one, and by requiring that the behavior of passage times in the boundary of the new pattern is very atypical (see Lemma \ref{Lemme pour dire qu'on ne perd pas de généralité avec les restrictions sur le motifs.} and in particular its last item). In the case \ref{c: Case III.} we just	require that the passage times on the boundary of the new pattern are very high and contained in a special interval (see Remark \ref{rem: Remarque sur les hypothèses concernant le motif.}
			and (AF-4’)). These will be the unique edges with passage times in this interval after the modification. In the case \ref{c: Case I.} we require the existence of a large connected component of edges with finite passage time (see Definition \ref{Définition boundary condition.} and (AI-4)). This will be the unique such large component not touching the boundary of $B_s$ after the modification.
	\end{itemize}	
		We place the pattern in a forbidden zone. By this we mean that, after the modification, the pattern lies in what was a forbidden zone before the modification. Recall that the pattern is unique in $B_s$ and that $\gamma$ does not enter into forbidden zones. Therefore $\gamma$ does not take the pattern in $B_s$ after the modification.
		\item A $k$-penalized path $\pi$ with finite passage time in the new environment is also a $k$-penalized path in the initial environment. Once again, it is sufficient to prove that $\pi$ does not take a pattern entirely contained in $B_s$ in the initial environment. The proof differs between the case \ref{c: Case I.} and the case \ref{c: Case III.}.
		\begin{itemize}
			\item In the case \ref{c: Case III.}, this is a consequence of the fact that $B_s$ is a typical box and that there is no pattern is a typical box. Indeed, the passage times on the boundary of the pattern are bigger (in the case \ref{c: Case III.}) than they can be in a typical box. 
			\item In the case \ref{c: Case I.}, it comes from the fact that, in the new environment, a path with a finite passage time taking edges in $B_s$ is very constrained (see Figure \ref{f: Modification.} where a path with a finite passage time can only take edges of the green, red, blue and orange parts). If $\pi$ does not take the pattern in the new environment, it can only take edges of $\gamma$. Since $\gamma$ does not take a pattern entirely contained in $B_s$ in the initial environment, neither does $\pi$.
		\end{itemize}
		\item A $k$-penalized path in the new environment has a passage time greater than or equal to the passage time of $\gamma$ in the new environment. The proof comes from the fact that a path with a reasonable passage time in the new environment is very constrained in $B_s$. The edges that do not belong to $\gamma$, the shortcut for $\gamma$ or the unique pattern in $B_s$ have a prohibitive passage time in the case \ref{c: Case III.} and an infinite passage time in the case \ref{c: Case I.}. Based on this observation and on the fact that $\gamma$ has a lower passage time in the new environment than in the initial environment, we simply prove that a $k$-penalized path can not save more time than $\gamma$ during the modification. 
		\item With the same ideas as above, we also prove that a $k$-geodesic in the new environment is also a $k$-geodesic in the initial environment.
	\end{enumerate}
	Indeed, by (i), (ii) and (iv) we get that $\gamma$ is a $k$-geodesic with a finite passage time in the new environment. Using the same arguments as before, we get that $\gamma$ crosses at least $q$ successful $k$-boxes in the new environment. Hence, the event $\cG$ occurs in the new environment. Furthermore, we also get that every $k$-geodesic crossing at least $q$ successful $k$-boxes in the new environment crosses at least $q$ successful $k$-boxes in the initial environment. Adding (iii) and (v), we get that the set of potential selected $k$-geodesics in the new environment is contained in the set of potential selected $k$ geodesic in the initial environment, and then $\gamma$ remains the first geodesic (and thus the selected one) among the geodesics of this set.
	
	\paragraph*{Advantages of a strategy using penalized geodesics.}
	
	In \cite{Jacq} the proof does not rely on penalized geodesics. Using penalized geodesics has two main advantages:
	\begin{itemize}
		\item In \cite{Jacq}, proving the result for all geodesics (and not only for one selected geodesic) requires further technicalities (see the use of concentric annuli in Section 2.1 in \cite{Jacq}). Here, it comes for free from the fact that the existence of one $k$-geodesic having a shortcut implies that $t(0,x)<t_P(0,x)$ and thus that every geodesic from $0$ to $x$ takes the pattern.
		\item It allows us to remove Assumption \eqref{eq: hypothèse de moment du premier article.}. Indeed, in item 3 in the paragraph on the modification and stability above, we need to have the same sequence of successful boxes crossed by $\gamma$ in the two environments. 
		Assume that we do not use penalized geodesics and, to make things easier, assume (only in this item) that we are in the case where there is a unique geodesic between any couple of vertices. Then the modification consists in replacing a subpath (denoted by $\OG$) of the geodesic from $0$ to $x$ (denoted by $\gamma$) by a path (denoted by $\pi$) with a shorter passage time which takes the pattern. It implies that in the new environment, $\OG$ does not belong to the new geodesic (which is the concatenation of the part of $\gamma$ from $0$ to $\pi$, then $\pi$ and then the part of $\gamma$ from $\pi$ to $x$). It can create a problem of stability if a box of the sequence of successful boxes crossed by $\gamma$ was crossed by $\OG$: the sequence of successful boxes crossed by the geodesic from $0$ to $x$ would not be the same in the two environments. To avoid this problem when we do not use penalized geodesics, we use the Cox-Durett shape theorem (Theorem 2.16 in \cite{50years}) in order to control the length of geodesics excursions from a box (see for example the proof of Lemma 2.1 in \cite{Jacq}). This is why we need Assumption \eqref{eq: hypothèse de moment du premier article.} in the strategy developed in \cite{Jacq}.
		
		Hence, by making a modification which guarantees that the penalized geodesic is the same in the initial environment as in the modified one, we avoid this problem without using the Cox-Durett shape theorem and thus without requiring Assumption \eqref{eq: hypothèse de moment du premier article.}. 
	\end{itemize}
		
	\subsection{Organization of the proof of Theorem \ref{Théorème à démontrer.}}\label{Sous-section organisation de la preuve.}
	
	Recall that, in this article, we assume \eqref{eq: Hypothèses sur la loi.}.
	One can check, using a standard re-normalization argument, that Theorem \ref{Théorème à démontrer.} is a simple consequence of the following proposition (see for example the proof of Theorem 2.3 in \cite{AndjelVares}).
	
	\begin{prop}\label{Gros théorème à démontrer, un seul motif.}
		Let $\mathfrak{P}=(\Lambda,u^\Lambda,v^\Lambda,\cA^\Lambda)$ be a valid pattern. Assume \eqref{eq: Hypothèses sur la loi.} and that $\L$ is useful. Then there exist $C>0$ and $D>0$ such that for all $n \ge 0$, for all $x$ such that $\|x\|_1=n$, 
		\begin{equation}
			\P \left((0,x) \in \mathfrak{C} \text{ and } \mbox{$\exists$ a geodesic $\gamma$ from $0$ to $x$ such that } N^\mathfrak{P}(\gamma)=0 \right) \le D \mathrm{e}^{-C n}. \label{Résultat gros héorème à démontrer, un seul motif.}
		\end{equation}
	\end{prop}
	
	Thus, the aim is now to prove Proposition \ref{Gros théorème à démontrer, un seul motif.}. Recall that there are two cases to be considered differently:
	\begin{enumerate}[label=(INF)]
		\item\label{c: Case I.} $\L(\infty)>0$,
	\end{enumerate}
	\begin{enumerate}[label=(FU)]
		\item\label{c: Case III.} $\L(\infty)=0$ and the support of $\L$ is unbounded.
	\end{enumerate}
	
	The proof of Proposition \ref{Gros théorème à démontrer, un seul motif.} is the aim of Section \ref{Section cas général.}. This section is divided in two parts. 
	
	Section \ref{Sous-section settings for the proof dans les cas a et b.} is devoted to patterns and typical boxes. We replace the original pattern by a larger pattern containing the original one and which satisfies several assumptions. Some of the assumptions are simply convenient: they simplify some parts of the proof. The assumption on the boundary is more crucial as explained in item (ii) in the paragraph "Some more details" in Section \ref{Sous-section sketch of the proof.}. In the case \ref{c: Case III.}, the requirement on the passage times on the boundary depends on the size of the boxes we consider in the proof. But the size of the boxes depends on the notion of typical boxes which in turn depends on parts of the definition of the pattern we consider. The definitions are thus intertwined. This is why we first start defining the new pattern in Section \ref{Sous-section hypothèses sur les motifs.} (postponing the boundary conditions in the case \ref{c: Case III.}), we then define and study typical boxes in Section \ref{Sous-section boîtes typiques.} and we finally choose the boundary of the pattern in the case \ref{c: Case III.} in Section \ref{Sous-section bordure des motifs.}.
	We then introduce the notions of $k$-penalized paths, shortcuts and successful boxes in Section \ref{Sous-section $k$-penalized paths.}.
	
	The second part of Section \ref{Section cas général.} is divided in four parts. In Section \ref{Sous-section réduction dans la preuve.}, the proof of Proposition \ref{Gros théorème à démontrer, un seul motif.} is reduced to the proof of a key lemma : Lemma \ref{Lemme pour les inégalités avec les indicatrices dans la modification.}. In this lemma, we introduce some sets of edges which correspond to the edges whose passage times have to be modified. The exact definitions of these sets are postponed to Section \ref{Sous-section modification}. It corresponds to the modification we want to make. The more difficult part in the proof of Lemma \ref{Lemme pour les inégalités avec les indicatrices dans la modification.} is item (iii), which is the key to get \eqref{eq: nouveau sketch of proof 2.}. To get this item and also item (i), we state and prove several properties which are consequences of the modification in Section \ref{Section conséquences de la modification.} before using them to conclude in Section \ref{Sous-section fin de la preuve du lemme.}.
	
	\subsection{Some tools and notations}\label{Sous-section Some tools and notations.}
	
	In this subsection, we recall some results and fix some notations. First, we denote by $\N$ the set of all non-negative integers, by $\N^*$ the set $\N \setminus \{0\}$, and by $\R_+$ the set of all $x \in \R$ such that $x \ge 0$. 
	
	For a self-avoiding\footnote{The definition can be extended to not necessarily self-avoiding paths by saying that a vertex $x$ is visited by $\pi$ before $y$ if there exists $i_0 \in \{0,\dots,k\}$ such that $x_{i_0}=x$ and for all $j \in \{0,\dots,k\}$, $x_j=y$ implies that $j>i_0$.} path $\pi=(x_0,...,x_k)$ going from $x_0$ to $x_k$, we say that $x_i$ is visited by $\pi$ before $x_j$ if $i<j$; we say that an edge $\{x_i,x_{i+1}\}$ is visited before an edge $\{x_j,x_{j+1}\}$ if $i<j$. A subpath of $\pi$ going from $x_i$ to $x_j$ (where $i,j \in \{0,\dots,k\}$ and $i<j$) is the path $(x_i,\dots,x_j)$ and is denoted by $\pi_{x_i,x_j}$.
	
	We say that a path, whose endpoints are denoted by $u$ and $v$, is oriented if this path has exactly $\|u-v\|_1$ edges. In other words, its number of edges is minimal among those of all the other paths linking $u$ and $v$.
	
	For a set $B$ of vertices, we denote by $\partial B$ its boundary, this is the set of vertices of $B$ which can be linked by an edge to a vertex which is not in $B$. We denote by $B^c$ the set of all vertices which does not belong to $B$. 
	When we define a set of vertices of $\Z^d$, sometimes we also want to say that an edge belongs to this set. So we make an abuse of notation by saying that an edge $e=\{u,v\}$ belongs to a set of vertices if $u$ and $v$ are in this set. Since now a subset $B$ of $\Z^d$ can be seen as a set of vertices or as a set of edges, we denote by $|B|_v$ the number of vertices of $B$ and by $|B|_e$ its number of edges. 
	
	Then, for all $c \in \Z^d$ and $r \in \R_+$, we denote 
	\begin{align*}
		B_\infty (c,r) & = \{u \in \Z^d \, : \, \|u-c\|_\infty \le r \}, \\
		B_1 (c,r) & = \{u \in \Z^d \, : \, \|u-c\|_1 \le r \},
	\end{align*}
	and for $n \in \N^*$, we denote by $\Gamma_n$ the boundary of $B_1(0,n)$, i.e.\ \begin{equation}
		\Gamma_n=\{u \in \Z^d \, : \, \|u\|_1=n \}. \label{définition Gamman}
	\end{equation}

	\paragraph*{Constants related to the distribution.}
	
	One can check that Lemma 5.5 in \cite{VdBK} can be adapted for a useful distribution $\L$ such that $\L(\infty) > 0$ and $\L([0,\infty))>p_c$. Thus there exist $\delta = \delta(\L) > 0$ and $D_0 = D_0(\L)$ fixed for the remaining of the article such that for all $u, \, v \in \Z^d$,
	\begin{equation}
		\P (\text{there exists a path $\pi$ from $u$ to $v$ such that } T(\pi) \le (\r+\delta) \|u-v\|_1) \le \mathrm{e}^{-D_0 \|u-v\|_1}. \label{Définition de delta}
	\end{equation}
	Furthermore, when $\r=0$,
	\begin{itemize}
		\item even if it means reducing $\delta$, in the cases \ref{c: Case I.} and \ref{c: Case III.}, we assume that $\delta>0$ is such that  $\L([0,\delta])<p_c$,
		\item even if it means reducing $\delta$, in the case \ref{c: Case I.}, we can fix $\nu_0$ such that
			\begin{equation}
				\nu_0 > \delta \text{ and } \L((\delta,\nu_0))>0. \label{On fixe nu0.}
		\end{equation}
		Note that it is possible since in the case \ref{c: Case I.}, we have $\L(0)+\L(\infty)<1$. Indeed, it comes from the fact that $\L$ is useful and that $\L([0,\infty)) > p_c$.
	\end{itemize}
	Then, still in the case where $\r=0$, we fix 
	\begin{equation}
		\beta>0, \, \beta'>0 \text{ and } \rho>0 \label{eq: On fixe les betas et rho.}
	\end{equation} such that \eqref{eq: Equation du lemme qui dit qu'il y a assez d'arêtes non nulles.} below holds with $\tau=\delta$. The existence of such constants is guaranteed by Lemma \ref{l: Lemme qui dit qu'il y a assez d'arêtes non nulles.} below whose proof is given in Appendix \ref{Annexe sur le lemme qui découle de Kesten Saint-Flour (5.8).}.

	\begin{lemma}\label{l: Lemme qui dit qu'il y a assez d'arêtes non nulles.}
		Assume that $\L$ is useful and that $\r=0$. Let $\tau>0$ such that $\L([0,\tau]) < p_c$. Then there exists $\beta>0$, $\beta'>0$ and $\rho>0$ such that for all $v,w \in \Z^d$,
		\begin{equation}
			\begin{split}
				\P(\text{there exists a self-avoiding path from $v$ to $w$ taking at most $\rho \|v-w\|_1$ edges $e$ such that} \\ T(e) > \tau) \le \beta' \mathrm{e}^{-\beta \|v-w\|_1}.
				\end{split}\label{eq: Equation du lemme qui dit qu'il y a assez d'arêtes non nulles.}
		\end{equation}
	\end{lemma}
	
	\section{Proof of Proposition \ref{Gros théorème à démontrer, un seul motif.}}\label{Section cas général.}
	
	Let $\mathfrak{P}=(\Lambda,u^\Lambda,v^\Lambda,\cA^\Lambda)$ be a valid pattern. We assume \eqref{eq: Hypothèses sur la loi.} and that $\L$ is useful. Thus one of the cases \ref{c: Case I.} or \ref{c: Case III.} stated in Section \ref{Sous-section organisation de la preuve.} is realized. The proofs in these two cases are almost the same. However, throughout this section, it will sometimes be necessary to distinguish the cases.
	
	\subsection{Settings for the proof}\label{Sous-section settings for the proof dans les cas a et b.}
	
	\subsubsection{Assumptions on the patterns}\label{Sous-section hypothèses sur les motifs.}
	
	We begin by making some assumptions on $\mathfrak{P}$ for the remaining of the proof. At first sight, these assumptions can be seen as a restriction but Lemma \ref{Lemme pour dire qu'on ne perd pas de généralité avec les restrictions sur le motifs.} guarantees that we can make them with no loss of generality.
	
	\paragraph*{In the case \ref{c: Case I.}.}
	
	\begin{definition}[Boundary condition]\label{Définition boundary condition.}
		For every $s \in \Z^d$ and $\lll \ge 3$, define the set $\cS_{s,\lll}$ as the set of edges belonging to the path going from $s-(\lll-1)\epsilon_1+(\lll-1)\epsilon_2$ to $s+(\lll-1)\epsilon_1+(\lll-1)\epsilon_2$ in the shortest way by $2(\lll-1)$ steps in the direction $\epsilon_1$.
		
		Then, we say that $B_\infty(s,\lll)$ satisfies the boundary condition in the environment $T$ if for all edges $e$ belonging to $B_\infty(s,\lll)$ but not to $B_\infty(s,\lll-3)$,
		\begin{itemize}
			\item either $e$ belongs to $\cS_{s,\lll} \cup (s+\Z \epsilon_1)$ and $T(e)$ is finite,
			\item or $T(e)$ is infinite.
		\end{itemize}
	\end{definition}

	\begin{figure}
		\begin{center}
			\begin{tikzpicture}[scale=0.33]
				\draw[gray,fill=Gray] (0,0) rectangle (10,10);
				\draw[gray,fill=white] (1.5,1.5) rectangle (8.5,8.5);
				\draw[Red,line width=2pt] (0.5,9.5) -- (9.5,9.5);
				\draw[LimeGreen,line width=2pt] (0,5) -- (1.5,5);
				\draw[LimeGreen,line width=2pt] (8.5,5) -- (10,5);
				\draw (8.5,1.5) node[above left,scale=0.8] {$B_\infty(s,\lll-3)$};
				\draw (10,0) node[above left,scale=0.8] {$B_\infty(s,\lll)$};
				\draw (5,5) node {$\bullet$};
				\draw (5,5) node[below right] {$s$};	
			\end{tikzpicture}
		\caption{Example of a ball satisfying the boundary condition. The set $\cS_{s,\lll}$ is represented in red. The edges of the gray area have an infinite passage time and the edges in green and red have a finite passage time.}\label{f: Boundary condition.}
		\end{center}
	\end{figure}
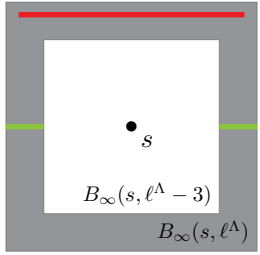

	\begin{remk}
		Let $s \in \Z^d$ and $\lll \ge 3$. If $B_\infty(s,\lll)$ satisfies the boundary condition in the environment $T$, then there is no path from $\partial B_\infty(s,\lll)$ to $\partial B_\infty(s,\lll)$ with finite passage time which takes an edge of $\cS_{s,\lll}$ (see Figure \ref{f: Boundary condition.} for a representation of the boundary condition in two dimensions).
	\end{remk}

	
	Let us consider the following assumptions:
	\begin{enumerate}[label=(AI-\arabic*)]
		\item\label{OP1} there exists an integer $\lll \ge 3$, fixed for the remaining of the proof, such that\footnote{We make a very slight abuse of notation: we also consider patterns where $0$ is in the center of $\Lambda$.} $\Lambda=B_\infty(0,\lll)$,
		\item\label{OP2} $\ulm = - \lll \epsilon_1$ and $\vlm = \lll \epsilon_1$,
		\item\label{OP7} there exist a constant $T^\Lambda>0$ and a path $\pi_\infty$ from $\ulm$ to $\vlm$ entirely contained in $\Lambda$ such that, when $\cA^\Lambda$ occurs,
		\begin{equation}
			T(\pi_\infty) < T^\Lambda. \label{On fixe TLambda dans le cas infini.}
		\end{equation}
		\item\label{OP3} if $\cA^\Lambda$ occurs, $\Lambda$ satisfies the boundary condition
	\end{enumerate}
	
	\paragraph*{In the case \ref{c: Case III.}.}
	Let us consider the following assumptions:
	\begin{enumerate}[label=(AF-\arabic*)]
		\item\label{OP4} there exists an integer $\lll>0$, fixed for the remaining of the proof, such that $\Lambda=B_\infty(0,\lll)$,
		\item\label{OP8} $\ulm = - \lll \epsilon_1$ and $\vlm = \lll \epsilon_1$,
		\item\label{OP5} when $\cA^\Lambda$ occurs, there exists a constant $M^\Lambda$ such that for every edge $e$ belonging to $\Lambda$ but not to $\partial \Lambda$, $T(e) \le M^\Lambda$,
		\item\label{OP6} for all $M > 0$, the event $\cA^\Lambda \cap \{\forall e \in \partial \Lambda, \, T(e)>M\}$ has a positive probability. 
	\end{enumerate}

	\begin{remk}\label{rem: Remarque sur les hypothèses concernant le motif.}
		The aim of wanting the pattern to satisfy the condition \ref{OP6} above is to be able to choose the passage times of the edges on its boundary once we have fixed some constants. Thus, from Section \ref{Sous-section bordure des motifs.} onwards the condition \ref{OP6} is replaced by the condition (AF-4') stated at this point.
	\end{remk}

	\paragraph*{In the two cases.}
	
	\begin{lemma}\label{Lemme pour dire qu'on ne perd pas de généralité avec les restrictions sur le motifs.}
		Let $\mathfrak{P}_0=(\Lambda_0,\ulm_0,\vlm_0,\cA^\Lambda_0)$ be a valid pattern. There exists a pattern $\mathfrak{P}=(\Lambda,\ulm,\vlm,\cA^\Lambda)$ such that:
		\begin{itemize}
			\item $\Lambda_0 \subset \Lambda$,
			\item $\P\left(\cA^\Lambda\right)$ is positive,
			\item on $\cA^\Lambda$, any path from $\ulm$ to $\vlm$ optimal for the passage time among the paths entirely inside $\Lambda$ contains a subpath from $\ulm_0$ to $\vlm_0$ entirely inside $\Lambda_0$,
			\item $\cA^\Lambda \subset \cA^\Lambda_0$,
			\item in the case \ref{c: Case I.}, $\mathfrak{P}$ satisfies the conditions \ref{OP1}, \ref{OP2}, \ref{OP7} and \ref{OP3}, and in the case \ref{c: Case III.}, $\mathfrak{P}$ satisfies the conditions \ref{OP4}, \ref{OP8}, \ref{OP5} and \ref{OP6}.
		\end{itemize}
	\end{lemma}

	Consider a valid pattern $\mathfrak{P}_0$ and a pattern $\mathfrak{P}$ satisfying the conditions of Lemma \ref{Lemme pour dire qu'on ne perd pas de généralité avec les restrictions sur le motifs.} above. Then, by this lemma, for every path $\pi$, if a vertex $x$ satisfies the condition $(\pi;\mathfrak{P})$, $x$ satisfies the condition $(\pi;\mathfrak{P}_0)$. Thus we get $N^{\mathfrak{P}_0} (\pi) \ge N^\mathfrak{P} (\pi)$ and to prove Proposition \ref{Gros théorème à démontrer, un seul motif.} for the pattern $\mathfrak{P}_0$, it is sufficient to prove it for the pattern $\mathfrak{P}$. That is why from now on, we can assume that the pattern $\mathfrak{P}$ introduced at the beginning of Section \ref{Section cas général.} satisfies the conditions \ref{OP1}, \ref{OP2}, \ref{OP7}, \ref{OP3} in the case \ref{c: Case I.} and the conditions \ref{OP4}, \ref{OP8}, \ref{OP5} and \ref{OP6} in the case \ref{c: Case III.}. The proof of Lemma \ref{Lemme pour dire qu'on ne perd pas de généralité avec les restrictions sur le motifs.} is postponed in Appendix \ref{Annexe sur les surmotifs.}.
	
	For the remaining of the proof, fix $\lll$ given by \ref{OP1} and \ref{OP4}. In the case \ref{c: Case I.}, fix $T^\Lambda$ given by \ref{OP7} and in the case \ref{c: Case III.}, fix 
	\begin{equation}
		M^\Lambda \text{ satisfying \ref{OP5} and } T^\Lambda > |\Lambda|_e M^\Lambda.\label{On fixe lLambda et TLambda dans le cas fini non borné.}
	\end{equation}
	
	\subsubsection{Typical boxes}\label{Sous-section boîtes typiques.}
	
	Recall that $\delta$ is fixed at \eqref{Définition de delta} and that, when $\r=0$, $\nu_0$ is fixed at \eqref{On fixe nu0.}.

	\paragraph*{Technical lemma.}
	We state in this paragraph the lemma used to create a "forbidden zone" in the case \ref{c: Case III.} (see item (ii) of the seventh paragraph of Section \ref{Sous-section sketch of the proof.}).
	
	Fix $r_P$ an integer such that
		\begin{equation}
			r_P > \max \left( \lll+2, \frac{2 (|B_\infty(0,\lll+1)|_e (\r+1)+T^\Lambda)}{\delta}, \frac{|B_\infty(0,\lll+1)|_e (\r+1)+T^\Lambda+1}{2 \r} \right). \label{On fixe rP.}
	\end{equation}
	
	\begin{lemma}\label{Lemme pour les clusters dans les typical boxes dans le cas finin non borné.}
		In the case \ref{c: Case III.}, we can define an event $\cT$, whose probability is positive, only depending on the edges of $B_\infty(0,r_P)$ and such that for all $x,y \in \partial B_\infty(0,r_P)$, for every self-avoiding path $\pi$ going from $x$ to $y$ and using only edges in $\partial B_\infty(0,r_P)$ and every path $\tilde{\pi}$ going from $x$ to $y$ using edges of $B_\infty(0,r_P)$ and at least one edge which is not in $\partial B_\infty(0,r_P)$, we have \[T(\pi) < T(\tilde{\pi}).\]
	\end{lemma}
	
	\begin{proof}
		Fix \[\nu > |\partial B_\infty(0,r_P)|_e (\r + 1).\] 
		Define the event $\cT$ as the event on which for all edge $e \in B_\infty(0,r_P)$, 
		\begin{itemize}
			\item $T(e) < \r + 1$ if $e \in \partial B_\infty(0,r_P)$,
			\item $T(e) > \nu$ else.
		\end{itemize}
		Since the support of $F$ is unbounded, the event $\cT$ has a positive probability.
		Assume that $\cT$ occurs.
		Then, let $x,y \in \partial B_\infty(0,r_P)$. Let $\pi$ be a self-avoiding path going from $x$ to $y$ and using only edges in $\partial B_\infty(0,r_P)$ and $\tilde{\pi}$ be a path going from $x$ to $y$ using edges of $B_\infty(0,r_P)$ and at least one edge $e'$ which is not in $\partial B_\infty(0,r_P)$.
		We get \[T(\pi) \le |\partial B_\infty(0,r_P)|_e (\r + 1),\] and 
		\[T(\tilde{\pi}) \ge T(e') \ge \nu > |\partial B_\infty(0,r_P)|_e (\r + 1).\]
		Hence, $T(\pi) < T(\tilde{\pi})$.
	\end{proof}
	
	\paragraph*{Boxes.}
	Recall that, when $\r=0$, $\rho$ is fixed at \eqref{eq: On fixe les betas et rho.}.
	Fix
	\begin{align}
		& 	r_1=1, \, r_2 > 4 d r_P \text{ and} \nonumber \\
		\text{when } \r>0, & \, r_3 > 2d(r_2+1) \label{On fixe r1 r2 et r3.} \\ 
		\text{when } \r=0, & \,r_3 > \max \left( 2d(r_2+1), \frac{2 \nu_0}{\rho \delta} \nonumber \right).
	\end{align}
	Note that, in particular, we have the following inequalities:
	\begin{itemize}
		\item $r_2 > 4 r_1 > 2 r_1$ since $r_1=1$ and $r_P \ge 1$,
		\item $r_3 > 2 r_2$ and $r_3 > r_2 + 1$.
	\end{itemize} 
	Then, for all $N \ge 1$, we define 
	\[\Bts=\{ v \in \Z^d \, : \, (s-r_3)N \le z < (s+r_3)N\}.\]
	and for $i \in \{1,2\}$, we define
	\[B_{i,s,N}=\{ v \in \Z^d \, : \, (s-r_i)N \le z \le (s+r_i)N\}.\]
	We use the word "box" to talk about $\Bts$.
	For $i \in \{1,2,3\}$, $\partial B_{i,s,N}$ is the set of vertices of $B_{i,s,N}$ having an adjacent vertex not contained in $B_{i,s,N}$. 

	\begin{definition}[Directed path and its selected straight segment]\label{d: central straight segment et chemin pi flèche.}
		Let $u \in \partial \Bds$ and $v \in \partial \Bus$. 
		\begin{itemize}
			\item We fix in an arbitrary way $\fpm(u,v)$ an oriented\footnote{Recall that, as it is defined in Section \ref{Sous-section Some tools and notations.}, a path, whose endpoints are denoted by $u$ and $v$, is oriented if this path has exactly $\|u-v\|_1$ edges.} path from $u$ to $v$ with a subpath $\fpm[u,v]$ between $\partial \Bds$ and $v$ using only edges in the same direction. We say that $\fpm(u,v)$ is the directed path between $u$ and $v$. The subpath $\fpm[u,v]$ is called the straight segment between $u$ and $v$ and its length is greater than or equal to $(r_2-r_1)N$.
			\item We define the selected straight segment between $u$ and $v$ as the set of vertices $c$ belonging to the straight segment between $u$ and $v$ and such that the distance for the norm $\|.\|_1$ between $c$ and $(\Bds)^c$ is at least $\displaystyle \frac{(r_2-r_1)N}{2} + d r_P$.
		\end{itemize}
	\end{definition}

	\begin{remk}\label{Remk: Selected straight segment non vide.}
		The selected straight segment is not the empty set since $v$ belongs to it. Indeed, the distance between $v$ and $(\Bds)^c$ is equal to $(r_2-r_1)N$ and $\displaystyle (r_2-r_1)N > \frac{(r_2-r_1)N}{2} + d r_P$ since $r_2 > 2 r_1$ and $r_2 > 4dr_P$ by \eqref{On fixe r1 r2 et r3.}.
	\end{remk}

	\begin{figure}
		\begin{center}
			\begin{tikzpicture}[scale=0.2]
				\draw (0,0) rectangle (16,16);
				\draw (6,6) rectangle (10,10);
				\draw (0,14) node[left] {$u$};
				\draw (10,7) node[right] {$v$};
				\draw (0,14) node[color=OrangeRed] {$\bullet$};
				\draw (10,7) node[color=OrangeRed] {$\bullet$};
				\draw (15,16) node[above] {$\Bds$};
				\draw (9,10) node[above] {$\Bus$};
				\draw[line width=2pt,color=OrangeRed] (0,14) -- (0,7) -- (10,7); 
				\draw (8,-3) node {(a)};
				\begin{scope}[xshift=25cm]
					\draw (0,0) rectangle (16,16);
					\draw (6,6) rectangle (10,10);
					\draw (0,14) node[left] {$u$};
					\draw (10,7) node[right] {$v$};
					\draw (0,14) node {$\bullet$};
					\draw (10,7) node[color=OrangeRed] {$\bullet$};
					\draw (15,16) node[above] {$\Bds$};
					\draw (9,10) node[above] {$\Bus$};
					\draw[line width=1.5pt] (0,14) -- (0,7) -- (10,7); 
					\draw[line width=2pt,color=OrangeRed] (0,7.01) -- (10,7.01); 
					\draw (8,-3) node {(b)};
				\end{scope}
				\begin{scope}[xshift=50cm]
					\draw (0,0) rectangle (16,16);
					\draw (6,6) rectangle (10,10);
					\draw (0,14) node[left] {$u$};
					\draw (10,7) node[right] {$v$};
					\draw (0,14) node {$\bullet$};
					\draw (10,7) node[color=OrangeRed] {$\bullet$};
					\draw (15,16) node[above] {$\Bds$};
					\draw (9,10) node[above] {$\Bus$};
					\draw[line width=1.5pt] (0,14) -- (0,7) -- (10,7); 
					\draw[line width=2pt,color=OrangeRed] (3.5,7.01) -- (10,7.01); 
					\draw[<->] (0,6) -- (3.5,6);
					\draw (2,5.5) node[below,fill=white,scale=0.7] {$\frac{(r_2-r_1)N}{2}+dr_P$};
					\draw (8,-3) node {(c)};
				\end{scope}
				\draw[line width=1.4pt,color=gray,dashed] (20.5,-2) -- (20.5,18);
				\draw[line width=1.4pt,color=gray,dashed] (45.5,-2) -- (45.5,18);
			\end{tikzpicture}
			\caption{Representation of the objects defined in Definition \ref{d: central straight segment et chemin pi flèche.} in two dimensions. (a) The directed path $\fpm(u,v)$ is in red. (b) The straight segment $\fpm[u,v]$ is in red. (c) The selected straight segment between $u$ and $v$ is in red.}\label{f: Selected straight segment.}
		\end{center}
	\end{figure}
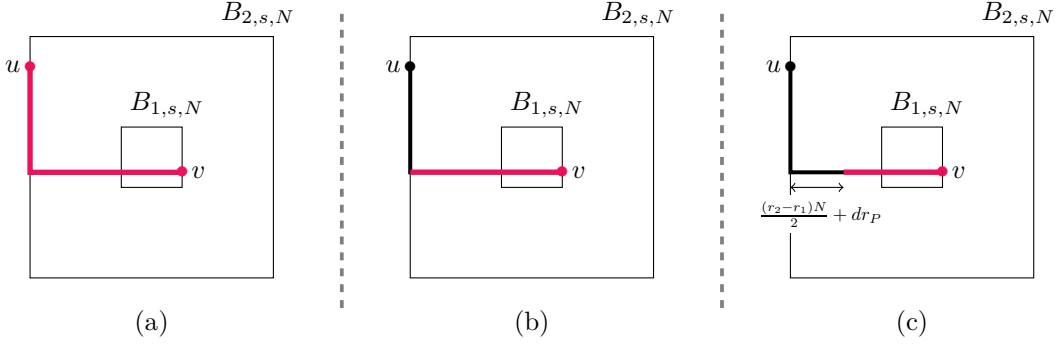
	
	\paragraph*{Typical boxes in the case \ref{c: Case I.}.}
	In these cases, a box $\Bts$ is typical if it verifies the following properties:
	\begin{enumerate}[label=(\roman*)]
		\item if $\r=0$, every path $\pi$ entirely contained in $\Bts$ from $u_\pi$ to $v_\pi$ with $\|u_\pi-v_\pi\|_1 \ge N$ has at least $\rho \|u_\pi-v_\pi\|_1$ edges whose passage time is greater than $\delta$,
		\item every path $\pi$ entirely contained in $\Bts$ from $u_\pi$ to $v_\pi$ with $\|u_\pi-v_\pi\|_1 \ge N$ has a passage time verifying: 
		\begin{equation}
			T(\pi) \ge (\r+\delta) \|u_\pi-v_\pi\|_1, \label{Définition chemins pas anormalement courts cas infini.}
		\end{equation}
		\item for all vertices $u \in \partial \Bds$ and $v \in \partial \Bus$, there exists a vertex $c$ belonging to the selected straight segment between $u$ and $v$ such that for every edge $e \in B_\infty(c,r_P)$, $T(e) = \infty$.
	\end{enumerate}

	\paragraph*{Typical boxes in the case \ref{c: Case III.}.}
	Fix $\cT$ the event given by Lemma \ref{Lemme pour les clusters dans les typical boxes dans le cas finin non borné.}. 
	We define a sequence $(\nu_1(N))_{N \in \N^*}$ such that:
	\begin{itemize}
		\item for all $N \in \N^*$, $\nu_1(N) > T^\Lambda$ if $\r>0$ and $\nu_1(N) > \max(\nu_0,T^\Lambda)$ if $\r=0$, 
		\item we have 
		\begin{equation}
			\lim\limits_{N \to \infty} \P \left( \sum_{e \in \Btz} T(e) \ge \nu_1(N) \right) = 0. \label{Troisième condition sur les suites nu1 de N et nu2 de N.}
		\end{equation}
	\end{itemize}
	Note that by \eqref{On fixe lLambda et TLambda dans le cas fini non borné.} and by the first item above, for all $N \in \N^*$, when $\cA^\Lambda$ occurs, $\nu_1(N)$ is strictly greater than the passage time of every edge belonging to $\Lambda$ but not to $\partial \Lambda$.
	Note also that $\L((\nu_1(N),\infty)) > 0$ for all $N \in \N^*$ since the support of $\L$ is unbounded. 
	
	In this case, a box $\Bts$ is typical if it verifies the following properties:
	\begin{enumerate}[label=(\roman*)]
		\item if $\r=0$, every path $\pi$ entirely contained in $\Bts$ from $u_\pi$ to $v_\pi$ with $\|u_\pi-v_\pi\|_1 \ge N$ has at least $\rho \|u_\pi-v_\pi\|_1$ edges whose passage time is greater than $\delta$,
		\item every path $\pi$ entirely contained in $\Bts$ from $u_\pi$ to $v_\pi$ with $\|u_\pi-v_\pi\|_1 \ge N$ has a passage time verifying: 
		\begin{equation}
			T(\pi) \ge (\r+\delta) \|u_\pi-v_\pi\|_1, \tag{\ref{Définition chemins pas anormalement courts cas infini.}}
		\end{equation}
		\item for all vertices $u \in \partial \Bds$ and $v \in \partial \Bus$, there exists a vertex $c$ belonging to the selected straight segment between $u$ and $v$ such that $\theta_{c} T \in \cT$,
		\item $\displaystyle \sum_{e \in \Bts} T(e) < \nu_1 (N)$.
	\end{enumerate}
	
	\paragraph*{Properties of typical boxes in both cases.}
	
	\begin{lemma}\label{Lemme propriétés boîtes typiques.}
		We have these two properties about typical boxes.
		\begin{enumerate}
			\item Let $s \in \Z^d$ and $N \in \N^*$. The typical box property only depends on the passage times of the edges in $\Bts$. 
			\item We have \[\lim\limits_{N \to \infty} \P \left( \Btz \text{ is a typical box} \right) = 1.\]
		\end{enumerate}
	\end{lemma}

	\begin{proof}
		\,
		\begin{enumerate}
			\item Properties (i) and (ii) in the two cases and property (iv) in the case \ref{c: Case III.} only depend on the edges of $\Bts$. Then, for all vertices $u \in \partial \Bds$ and $v \in \partial \Bus$, every vertex $c$ belonging to the selected straight segment between $u$ and $v$ has a distance with $(\Bds)^c$ greater than or equal to $\displaystyle \frac{(r_2-r_1)N}{2}+dr_P$. Hence, property (iii) only depends on the edges of $\Bds$ and $\Bds \subset \Bts$.
			\item First, let us prove that, in the two cases, the probability that $(i)$ is satisfied by $B_{3,0,N}$ goes to $1$. For this item, assume that $\r=0$. Let $\Pi$ denote the set of self-avoiding paths entirely contained in $B_{3,0,N}$. For a path $\pi$ going from a vertex $u_\pi$ to a vertex $v_\pi$, we say that $\pi$ satisfies the property $\cP_\delta$ if $\pi$ takes at least $\rho \|u_\pi-v_\pi\|_1$ edges $e$ such that $T(e)>\delta$. Then, using Lemma \ref{l: Lemme qui dit qu'il y a assez d'arêtes non nulles.},
			\begin{align*}
				& \P( B_{3,0,N} \mbox{ does not satisfy } (i) ) \\ \le & \sum_{\substack{u_\pi,v_\pi \in B_{3,0,N} \\ \|u_\pi-v_\pi\|_1 \ge N}} \P \left( \text{$\cP_\delta$ is not satisfied by a path of $\Pi$ whose endpoints are $u_\pi$ and $v_\pi$} \right) \\
				\le & \sum_{\substack{u_\pi,v_\pi \in B_{3,0,N} \\ \|u_\pi-v_\pi\|_1 \ge N}} \P \left( \text{$\cP_\delta$ is not satisfied by a path whose endpoints are $u_\pi$ and $v_\pi$} \right) \\
				\le & |B_{3,0,N}|_v^2 \beta' \mathrm{e}^{- \beta N} \xrightarrow[N \to \infty]{} 0, 
			\end{align*}
			since $|B_{3,0,N}|_v$ is bounded by a polynomial in $N$. 
			
			Now, for the remaining of this proof, $\r$ can be positive. Using \eqref{Définition de delta} and a similar computation as above, we get that \[\P( B_{3,0,N} \mbox{ does not satisfy } (ii) ) \xrightarrow[N \to \infty]{} 0.\]
			
			Recall Definition \ref{d: central straight segment et chemin pi flèche.}. To prove that the probability that $(iii)$ is satisfied by $\Btz$ goes to $1$, we begin by associating in a deterministic way to each couple of vertices $(u,v) \in \partial \Bdz \times \partial \Buz$ a set of vertices, denoted by $V(\fpm(u,v))$ such that:
			\begin{itemize}
				\item every vertex of $V(\fpm(u,v))$ belongs to the selected straight segment between $u$ and $v$,
				\item for all $z_1,z_2 \in V(\fpm(u,v))$, $B_\infty(z_1,r_P) \cap B_\infty(z_2,r_P) = \emptyset$,
				\item there can be no other set satisfying the two conditions above containing strictly more vertices than $V(\fpm(u,v))$.
			\end{itemize}
			Note that there exists a constant $K_1$ only depending on $r_1$, $r_2$ and $r_P$ such that $|V(\fpm(u,v))| \ge K_1N$. In the case \ref{c: Case I.}, we denote by $\cT_\infty$ the event on which for all $e \in B_\infty(0,r_P)$, $T(e) = \infty$. We use the notation $\overline{\cT}$ to designate the event $\cT$ in the case \ref{c: Case III.} and to designate $\cT_\infty$ in the case \ref{c: Case I.}, which allows us to conclude this part of proof in the two cases. 
			We have
			\begin{align*}
				\P( B_{3,0,N} \mbox{ does not satisfy } (iii) ) \le & \sum_{\substack{u \in \partial \Bdz \\ v \in \partial \Buz}} \P \left(\forall c \in \fpm(u,v), \, \theta_c \overline{\cT} \text{ does not occur} \right) \\
				\le & \sum_{\substack{u \in \partial \Bdz \\ v \in \partial \Buz}} \P \left(\forall c \in V(\fpm(u,v)), \, \theta_c \overline{\cT} \text{ does not occur} \right).
			\end{align*}
		Fix $u \in \partial \Bdz$ and $v \in \partial \Buz$. Since for all $z_1,z_2 \in V(\fpm(u,v))$, $B_\infty(z_1,r_P) \cap B_\infty(z_2,r_P) = \emptyset$, the family of events $\left( \left\{ \theta_c \overline{\cT} \text{ does not occur} \right\} \right)_{c \in V(\fpm(u,v))}$ are independent and thus 
		\[\P \left(\forall c \in V(\fpm(u,v)), \, \theta_c \overline{\cT} \text{ does not occur} \right) \le \left(1-\P(\overline{\cT})\right)^{|V(\fpm(u,v))|} \le \left(1-\P(\overline{\cT})\right)^{K_1N}.\]
		Since $\P(\overline{\cT})>0$, we get the existence of a constant $K_2$ not depending on $u$, $v$ and $N$ such that:
		\[\P \left(\forall c \in V(\fpm(u,v)), \, \theta_c \overline{\cT} \text{ does not occur} \right) \le \mathrm{e}^{-K_2N}.\]
		Hence, 
		\[\P( B_{3,0,N} \mbox{ does not satisfy } (iii) ) \le |\partial \Bdz|_v |\partial \Buz|_v \mathrm{e}^{-K_2N} \xrightarrow[N \to \infty]{} 0,\]
		since $|\partial \Bdz|_v |\partial \Buz|_v$ is bounded by a polynomial in $N$.
		
		Finally, in the case \ref{c: Case III.}, we get that the probability that $(iv)$ is satisfied by $\Btz$ goes to $1$ by \eqref{Troisième condition sur les suites nu1 de N et nu2 de N.}.
		\end{enumerate}
	\end{proof}

	\paragraph*{Crossing a box.}
	A self-avoiding path crosses a box $\Bts$ if it visits one vertex of $\partial \Bts$, then one of $\Bus$ and then another one of $\partial \Bts$. 
	
	The following lemma is a consequence of Lemma 5.2 in \cite{VdBK} which applies using Lemma \ref{Lemme propriétés boîtes typiques.}.

	\begin{lemma}\label{Lemme nombre linéaire de bonne boîtes rencontrées.}
		For any $N$ sufficiently large, we can take $D_1 > 0$ and $\alpha > 0$ such that for all $n \ge 1$,
		\begin{equation}
			\begin{split}
				\P (\exists \text{$z \in \Gamma_n$ such that $(0,z) \in \mathfrak{C}$ and $\exists$ a path from $0$ to $z$ that crosses} \\ \text{at most $\lfloor \alpha n \rfloor$ typical boxes}) \le \mathrm{e}^{-D_1 n}.\label{eq: Lemme nombre linéaire de bonnes boîtes rencontrées.}
			\end{split}			
		\end{equation}
	\end{lemma}
	Using Lemma \ref{Lemme nombre linéaire de bonne boîtes rencontrées.} above, we fix 
		\begin{equation}
			\text{$N \ge 1$ large enough and $D_1 > 0$, $\alpha > 0$ such that \eqref{eq: Lemme nombre linéaire de bonnes boîtes rencontrées.} holds.} \label{eq: On fixe N.}
		\end{equation}
	For the remaining of the proof, since $N$ is fixed, we write $\nu_1$ instead of $\nu_1(N)$ in the case \ref{c: Case III.}.
	We fix $\delta'>0$ such that
		\begin{equation}
			\delta' < \min \left( \frac{\delta}{2}, \frac{1}{N} \right). \label{eq: On fixe delta prime.}
	\end{equation}
	Note that, in particular, since $N \ge 1$, we have $\delta' < 1$.
	
	\subsubsection{Boundaries of the patterns in the case \ref{c: Case III.}}\label{Sous-section bordure des motifs.}

	In the case \ref{c: Case III.}, we fix $\nu_2>\nu_1$ such that 
	\begin{equation}
		\cA^\Lambda \cap \{\forall e \in \partial \Lambda, \, T(e) \in (\nu_1,\nu_2) \} \text{ has a positive probability. }\label{eq: On fixe nu2.}
	\end{equation} 
	It is possible since by the condition \ref{OP6} in Section \ref{Sous-section hypothèses sur les motifs.}, for all $M>$, $\cA^\Lambda \cap \{\forall e \in \partial \Lambda, \, T(e) > M \}$ has a positive probability.
	\begin{enumerate}[label=(AF-4')]
		\item For the remaining of the proof, we now replace $\cA^\Lambda$ by 
		\[\cA^\Lambda \cap \{\forall e \in \partial \Lambda, \, T(e) \in (\nu_1,\nu_2) \}.\]
	\end{enumerate}
	As announced in Remark \ref{rem: Remarque sur les hypothèses concernant le motif.}, from now on, the event $\cA^\Lambda$ of the pattern $\mathfrak{P}$ has been modified. The assumption \ref{OP6} is not satisfied by this new event but it satisfies the assumption (AF-4') above. 
	
	\begin{remk}\label{Remarque pas de motif dans une boîte typique.}
		In the case \ref{c: Case III.}, since there is no edge whose passage time is greater than $\nu_1$ in a typical box, there can be no pattern in a typical box.
	\end{remk}
		
	\subsubsection{$k$-penalized paths, shortcuts and successful boxes}\label{Sous-section $k$-penalized paths.}
	
	Recall that $r_3$ is fixed at \eqref{On fixe r1 r2 et r3.} and that $N$ is fixed at \eqref{eq: On fixe N.}. We partition $\Z^d$ with boxes $\Bts$ in $\Kpar=\Kpar(r_3)=|B_{3,0,1}|_v$ ways as follows. For each $z \in B_{3,0,1}$, the partition associated with $z$ is 
	\[\left\{ \Bts, \, s-z \in 2r_3\Z^d \right\}.\]
	For convenience, we index these different partitions from $1$ to $\Kpar$. For $k \in \{1,\dots,\Kpar\}$, the boxes belonging to the $k$-th partition are called $k$-boxes.
	
	Recall that we say that a self-avoiding path $\pi$ takes the pattern if there exists $z \in \Z^d$ satisfying the condition $(\pi;\mathfrak{P})$. For $k \in \{1,\dots,\Kpar\}$, we say that a self-avoiding path $\pi$ takes a pattern entirely contained in a $k$-box if there exists $z \in \Z^d$ such that:
	\begin{itemize}
		\item $z$ satisfies the condition $(\pi,\mathfrak{P})$,
		\item there exists a $k$-box containing $B_\infty(z,\lll)$.
	\end{itemize}
	
	\paragraph*{$k$-penalized paths.}
			For $k \in \{1,\dots,\Kpar\}$, a $k$-penalized path is a self-avoiding path which takes no pattern entirely contained in a $k$-box.


	\paragraph*{Penalized passage time.}

	For all $x \in \Z^d$, 
	and for $k \in \{1,\dots,\Kpar\}$, we define
	\[t_k(0,x) = \inf \{T(\pi) \, : \, \pi \text{ is a $k$-penalized path from $0$ to $x$}\},\] with the convention $\inf \emptyset = \infty$. 
	
	\paragraph*{$k$-geodesics.}
	For all $x \in \Z^d$,
	for all $k \in \{1,\dots,\Kpar\}$, a $k$-geodesic from $0$ to $x$ is a $k$-penalized path $\gamma$ from $0$ to $x$ such that $T(\gamma)=t_k(0,x)$. 
	
	\paragraph*{Shortcuts.}
	
	For all boxes $\Bts$, we say that a path $\pi$ has a shortcut in $\Bts$ if $\pi$ crosses $\Bts$ and if there exist two vertices $u$ and $v$ of $\pi$ and a path $\pi'$ going from $u$ to $v$ such that:
	\begin{itemize}
		\item $\pi'$ is entirely contained in $\Bts$,
		\item $\pi_{u,v}$ and $\pi'$ have only $u$ and $v$ as vertices in common,
		\item $\pi'$ takes a pattern entirely contained in $\Bts$,
		\item $T(\pi') < T(\pi_{u,v} \cap \Bts)$.
	\end{itemize}

	\paragraph*{Successful boxes.}
	Let $\Bts$ be a box and $\pi$ a self-avoiding path. 
	We say that $\Bts$ is successful \textbf{for the path $\pi$} if the following two conditions hold:
	\begin{itemize}
		\item $\pi$ crosses $\Bts$,
		\item $\Bts$ is a typical box or $\pi$ has a shortcut in $\Bts$.
	\end{itemize} 
	
	\paragraph*{$S^k$-sequences.}
	For every $k \in \{1,\dots,\Kpar\}$, for every $k$-geodesic $\gamma$ between two vertices, the $S^k$-sequence of $\gamma$ is the sequence of different $k$-boxes successful for $\gamma$ by order of first visit by $\gamma$. Note that the boxes of this $S^k$-sequence are pairwise disjoint by the definition of $k$-boxes.
	
	\subsection{Proof}\label{Sous-section preuve dans les cas a et b.}
	
	\subsubsection{Reduction}\label{Sous-section réduction dans la preuve.}
	
	We begin the proof with some definitions.
	Recall that $\alpha$ is fixed at \eqref{eq: On fixe N.} and that $\Kpar$ is fixed at the beginning of Section \ref{Sous-section $k$-penalized paths.}.
	For all $n \ge 1$, write
	\[Q_n = \left\lfloor \frac{\alpha n}{\Kpar} \right\rfloor.\]
	Fix
	\begin{equation}
	n \ge 1 \text{ and } x \in \Gamma_n \label{eq: On fixe n et x.}
	\end{equation}
	From now on, when we talk about a path, a geodesic or a $k$-geodesic without specifying its extremities, we mean that it is from $0$ to $x$. 
	For all $k \in \{1,\dots,\Kpar\}$, we define
	\begin{equation*}
		\begin{split}
			\good = \{\text{$(0,x) \in \mathfrak{C}$ and there exists a $k$-geodesic whose passage time is finite having} \\ \text{at least $Q_n$ boxes in its $S^k$-sequence}\}.
		\end{split}
	\end{equation*}
	
	\paragraph*{Selected $k$-geodesic and $S^k$-variables.}
	For all $k \in \{1,\dots,\Kpar\}$, on the event $\good$, 
	\begin{itemize}
		\item we define the selected $k$-geodesic as the first $k$-geodesic in the lexicographical order\footnote{The lexicographical order is based on the directions of the consecutive edges of the geodesics.} among those having at least $Q_n$ boxes in their $S^k$-sequences,
		\item for all $j \in \{1,\dots,Q_n\}$, we define the random variable $\cS^k_j$ as the vertex $s$ such that $\Bts$ is the $j$-th box of the $S^k$-sequence of the selected $k$-geodesic. 
	\end{itemize}  
	
	\paragraph*{Events $\cM^k$.}
	For all $k \in \{1,\dots,\Kpar\}$ and all $j \in \{1,\dots,Q_n\}$, we define
	\[\cM^k(j) = \good \cap \{\text{the selected $k$-geodesic does not have a shortcut in any of the first $j$ boxes of its $S^k$-sequence}\}.\]
	To make the end of this proof easier to read, we define the events
	\begin{align*}
		A & = \{\text{$(0,x) \in \mathfrak{C}$ and there exists a geodesic from $0$ to $x$ which does not take the pattern}\}, \\
		B & = \{\text{every path from $0$ to $\Gamma_n$ crosses at least $\lfloor \alpha n \rfloor + 1$ typical boxes}\},
	\end{align*}
	Note that $A$ is the event considered in Proposition \ref{Gros théorème à démontrer, un seul motif.} and $B$ the complementary event to the one considered in Lemma \ref{Lemme nombre linéaire de bonne boîtes rencontrées.}. The proof of Proposition \ref{Gros théorème à démontrer, un seul motif.} is based on the following two lemmas.
	
	\begin{lemma}\label{l: Claim de la réduction.}
		We have $\displaystyle A \cap B \subset \bigcup_{k=1}^{\Kpar} \cM^k(Q_n)$.
	\end{lemma}

	\begin{lemma}\label{Lemme majoration de la probabilité de l'événement par lambda puissance Q.}
		There exists $\lambda \in (0,1)$ which does not depend on $x$ and $n$ such that for all $k \in \{1,\dots,\Kpar\}$, \[\P \left(\cM^k(Q_n) \right) \le \lambda^{Q_n}. \]
	\end{lemma}

	\begin{proof}[Proof of Proposition \ref{Gros théorème à démontrer, un seul motif.} using Lemma \ref{l: Claim de la réduction.} and \ref{Lemme majoration de la probabilité de l'événement par lambda puissance Q.}]
		Recall that $N$ is fixed at \eqref{eq: On fixe N.} and that $n$ and $x$ are fixed at \eqref{eq: On fixe n et x.} but that $D_1$ and $\lambda$ do not depend on $x$ and $n$. We have
		\begin{align*}
			\P(A) & \le \P(A \cap B) + \P(B^c) \\
			& \le \sum_{k=1}^\Kpar \P \left( \cM^k(Q_n) \right) + \P(B^c) \text{ by Lemma \ref{l: Claim de la réduction.},} \\
			& \le \Kpar \lambda^{Q_n} + \P(B^c) \text{ by Lemma \ref{Lemme majoration de la probabilité de l'événement par lambda puissance Q.},} \\
			& \le \Kpar \lambda^{Q_n} + \mathrm{e}^{-D_1n} \text{ by Lemma \ref{Lemme nombre linéaire de bonne boîtes rencontrées.}.}
		\end{align*}
		As $D_1>0$ and $\lambda \in (0,1)$, and as this inequality holds for any $n \ge 1$ and any $x \in \Gamma_n$, we get the existence of two constants $C>0$ and $D>0$ such that for all $n$, for all $x \in \Gamma_n$, 
		\[\P(A) \le D \exp(-Cn).\]
	\end{proof}
	
	\begin{proof}[Proof of Lemma \ref{l: Claim de la réduction.}]
		Assume that $A$ occurs. Then there exists a self-avoiding path $\gamma$ from $0$ to $x$ such that:
		\begin{itemize}
			\item $\gamma$ does not take the pattern,
			\item $T(\gamma)=t(0,x)<\infty$.
		\end{itemize}
		For all $k \in \{1,\dots,\Kpar\}$, we get that:
		\begin{itemize}
			\item $\gamma$ is a $k$-penalized path,
			\item $T(\gamma)=t(0,x)=t_k(0,x)<\infty$.
		\end{itemize}
		Thus, for all $k \in \{1,\dots,\Kpar\}$, $\gamma$ is a $k$-geodesic from $0$ to $x$ and no $k$-geodesic has a shortcut in any box. Assume that $B$ also occurs. Then $\gamma$ crosses at least $\left\lfloor \alpha n \right\rfloor + 1$ typical boxes. Hence, there exists $k \in \{1,\dots,\Kpar\}$ such that $\gamma$ crosses at least $\displaystyle Q_n = \left\lfloor \frac{\alpha n}{\Kpar} \right\rfloor$ typical boxes. Since every typical $k$-box crossed by $\gamma$ is a successful box for $\gamma$, $\gamma$ is a $k$-geodesic having at least $Q_n$ boxes in its $S^k$-sequence. Hence the event $\cG^k$ occurs and the selected $k$-geodesic (which is not necessarily $\gamma$) does not have a shortcut in any of the first $Q_n$ boxes of its $S^k$-sequence since it does not have a shortcut in any box. So the event $\cM^k(Q_n)$ occurs.
	\end{proof}
	
	Now, for the remaining of the proof, the aim is to prove Lemma \ref{Lemme majoration de la probabilité de l'événement par lambda puissance Q.}.

	\paragraph*{Modification argument.}
	
	We introduce an independent copy $T'$ of the environment $T$, the two being defined on the same probability space. It is thus convenient to refer to the considered environment when dealing with the objects defined above. To this aim, we shall use the notation $\{T \in \cM^k(j)\}$ to denote that the event $\cM^k(j)$ holds with respect to the environment $T$. In other words, $\cM^k(j)$ is now seen as a subset of $[0,\infty]^\cE$ where $\cE$ is the set of all the edges. Similarly, we denote by $\cS^k_j(T')$ the random variables defined above but in the environment $T'$. 
	
	Fix $k \in \{1,\dots,K\}$ and $\l \in \{1,\dots,Q_n\}$. On $\{T \in \cM^k(\l)\}$, the event $\cG^k$ occurs and $B_{3,S^k_\l(T),N}$ is the $\l$-th box of the $S^k$-sequence of the selected $k$-geodesic. From this new environment, we associate a set of edges $E^*_\text{modif}(T)$ which is contained in $B_{3,S^k_\l(T),N}$. It corresponds to the edges for which we want to modify the time.	We get a new environment $\Tu$ defined for all edges $e$ by:
	\[\Tu(e) = \left\{
	\begin{array}{ll}
		T(e) & \mbox{if } e \notin E^*_\text{modif}(T), \\
		T'(e) & \mbox{else.}
	\end{array}
	\right.\]
	For $y$ and $z$ in $\Z^d$, we denote by $t^*(y,z)$ the geodesic time between $y$ and $z$ in the environment $\Tu$. Note that $T$ and $\Tu$ do not have the same distribution as the set $\Emod(T)$ depends on $T$.
	
	The proof of Lemma \ref{Lemme majoration de la probabilité de l'événement par lambda puissance Q.} relies on the following lemma whose proof is given in the next subsection. Recall that, in the case \ref{c: Case III.}, $\nu_2$ is fixed at \eqref{eq: On fixe nu2.} and that, if $\r=0$ in the cases= \ref{c: Case I.}, $\nu_0$ is fixed at \eqref{On fixe nu0.}.
	
	\begin{lemma}\label{Lemme pour les inégalités avec les indicatrices dans la modification.}
		There exists $\eta=\eta(N) > 0$ such that for all $\l$ in $\{1,\dots,Q_n\}$, for all $k \in \{1,\dots,\Kpar\}$, there exist measurable functions $\Em$, $\Emid$, $\Ep$, $\Einf$ all from $[0,\infty]^\cE$ to $\cP(\cE)$ and a measurable function $\cC \, : \, [0,\infty]^\cE \mapsto \Z^d$ such that:
		\begin{enumerate}[label=(\roman*)]
			\item on the event $\{T \in \cM^k(\l)\}$, $\Em(T)$, $\Emid(T)$, $\Ep(T)$, $\Einf(T)$ and $B_\infty(\cC(T),\lll)$ are pairwise disjoint and are contained in $B_{3,S^k_\l(T),N}$,
			\item on the event $\{T \in \cM^k(\l) \}$, we have $\P \left( T' \in \cB^*(T) | T \right) \ge \eta$, where $\{T' \in \cB^*(T)\}$ is a shorthand for the event on which:
			\begin{itemize}
				\item $\forall e \in \Em(T), \, T'(e) \le \r+\delta'$,
				\item $\forall e \in \Emid(T), \, T'(e) \in ( \r + \delta , \nu_0 )$,
				\item $\forall e \in \Ep(T), \, T'(e) > \nu_2$,
				\item $\forall e \in \Einf(T), \, T'(e) = \infty$,
				\item $\theta_{ \cC(T)} T' \in \cA^\Lambda$.
			\end{itemize}
			\item $\{T \in \cM^k(\l) \} \cap \{T' \in \cB^*(T)\} \subset \{\Tu \in \cM^k(\l-1) \setminus \cM^k(\l)\}$ and $S^k_\l(\Tu)=S^k_\l(T)$.
		\end{enumerate}
	\end{lemma}

	\begin{remk}
			Several of the functions of the previous lemma can be equal to the empty set depending on the cases. Thus this does not prevent us from having $\P \left( T' \in \cB^*(T) | T \right) \ge \eta$. In particular, for every environment $T$:
				\begin{itemize}
					\item in the case \ref{c: Case I.}, $\Ep(T)=\emptyset$,
					\item in the case \ref{c: Case III.}, $\Einf(T)=\emptyset$,
					\item in the two cases, if $\r>0$, $\Emid(T)=\emptyset$.
				\end{itemize}
	\end{remk}

	\begin{proof}[Proof of Lemma \ref{Lemme majoration de la probabilité de l'événement par lambda puissance Q.} using Lemma \ref{Lemme pour les inégalités avec les indicatrices dans la modification.}]
		Let $\l \in \{1,\dots,Q_n\}$ and $k \in \{1,\dots,\Kpar\}$. For every $s \in \Z^d
		$ and $\cE^*$ subset of edges of $\Bts$, let us consider the environment $\Tu_{s,\cE^*}$ defined for all edges $e$ by:
		\[\Tu_{s,\cE^*}(e) = \left\{
		\begin{array}{ll}
			T'(e) & \mbox{if } e \in \cE^*, \\
			T(e) & \mbox{else.}
		\end{array}
		\right.\]
		We define $E^*_\text{modif}(T)= \Em(T) \cup \Emid(T) \cup \Ep(T) \cup \Einf(T) \cup B_\infty(\cC(T),\lll)$. Thus, for every $s$ and $\cE^*$, $\Tu_{s,\cE^*}$ and $T$ have the same distribution and on the event $\{T \in \cM^k(\l)\} \cap \{S^k_\l(T)=s\} \cap \{E^*_\text{modif}(T) = \cE^*\}$, $\Tu = \Tu_{s,\cE^*}$.
		So, using this environment and writing with indicator functions the result of Lemma \ref{Lemme pour les inégalités avec les indicatrices dans la modification.}, we get: 
		\begin{equation}
			\1_{\{T \in \cM^k(\l) \}} \1_{\{S^k_\l(T)=s\}} \1_{\{E^*_\text{modif}(T) = \cE^*\}} \1_{\{T' \in \cB^*(T)\}} \le \1_{\{\Tu_{s,\cE^*} \in \cM^k(\l-1) \setminus \cM^k(\l)\}} \1_{\{S^k_\l(\Tu_{s,\cE^*})=s\}}, \label{eq: Inégalités des indicatrices.}
		\end{equation}
		We compute the expectation on both sides. For the left side, we have 
		\begin{align*}
			& \E \left[ \1_{\{T \in \cM^k(\l) \}} \1_{\{S^k_\l(T)=s\}} \1_{\{E^*_\text{modif}(T) = \cE^*\}} \1_{\{T' \in \cB^*(T)\}} \right] \\ 
			= & \, \E \left[\1_{\{T \in \cM^k(\l) \}} \1_{\{S^k_\l(T)=s\}} \1_{\{E^*_\text{modif}(T) = \cE^*\}} \E \left[ \left. \1_{\{T' \in \cB^*(T)\}} \right| T \right] \right]. 
		\end{align*}
		Since on the event $\{T \in \cM^k(\l) \} \cap \{S^k_\l(T)=s\}$, we have $\P \left( T' \in \cB^*(T) | T \right) \ge \eta$, the left side is bounded from below by $\eta \P (T \in \cM^k(\l), \, S^k_\l(T)=s, \, E^*_\text{modif}(T) = \cE^*)$. Since $\Tu_{s,\cE^*}$ and $T$ have the same distribution, using \eqref{eq: Inégalités des indicatrices.}, we get:
		\[\eta \P (T \in \cM^k(\l), \, S^k_\l(T)=s, \, E^*_\text{modif}(T) = \cE^*) \le \P \left( T \in \cM^k(\l-1) \setminus \cM^k(\l), \, S^k_\l(T)=s \right).\] 
		Then, by writing $K'$ the number of subsets of edges of $\Btz$ and by summing on all subsets $\cE^*$ of edges of $\Bts$, we get for all $s \in \Z^d$,
		\[\frac{\eta}{K'} \P (T \in \cM^k(\l), \, S^k_\l(T)=s) \le \P \left( T \in \cM^k(\l-1) \setminus \cM^k(\l), \, S^k_\l(T)=s \right).\] 
		Finally, by summing\footnote{Note that here, we must have the event $\{S^k_\l(\Tu_{s,\cE^*})=s\}$ on the right side of the inequality \eqref{eq: Inégalités des indicatrices.} to sum on all $s \in \Z^d$. } on all $s \in \Z^d$, we get
		\[\frac{\eta}{K'}\P(T \in \cM^k(\l)) \le \P(T \in \cM^k(\l-1) \setminus \cM^k(\l)).\]
		Now, since $\cM^k(\l) \subset \cM^k(\l-1)$,
		
		\[ \P(T \in \cM^k(\l-1) \setminus \cM^k(\l))= \P \left( T \in \cM^k(\l-1)\right) - \P \left( T \in \cM^k(\l) \right).\]
		Thus, \[\P(T \in \cM^k(\l)) \le \lambda \P(T \in \cM^k(\l-1) ),\] where $\displaystyle \lambda=\frac{1}{1+\frac{\eta}{K'}} \in (0,1)$ does not depend on $x$ and $n$. 
		Hence, using $\P(T \in \cM^k(0)) =1$, we get by induction \[\P(T \in \cM^k(Q_n)) \le \lambda^Q_n.\]
	\end{proof}
	
	\subsubsection{Proof of Lemma \ref{Lemme pour les inégalités avec les indicatrices dans la modification.}: modification}\label{Sous-section modification}
	
	Let $\l \in \{1,\dots,Q_n\}$, $k \in \{1,\dots,\Kpar\}$ and $s \in \Z^d$ such that $\Bts$ is a $k$-box. Assume that the event $\{T \in \cM^k(\l)\} \cap \{S^k_\l(T)=s\}$ occurs. Note that $(0,x) \in \mathfrak{C}$. We denote by $\gamma$ the selected $k$-geodesic. We know that:
	\begin{enumerate}[label=(H\arabic*)]
		\item\label{H1} $\gamma$ has at least $Q_n$ boxes in its $S^k$-sequence,
		\item\label{H2} $\gamma$ does not have a shortcut in any of the first $\l$ boxes of its $S^k$-sequence,
		\item\label{H3} $\Bts$ is the $\l$-th box of the $S^k$-sequence of $\gamma$ and is a typical box.
	\end{enumerate}
	
	\paragraph*{Construction of the forbidden zone and definition of $\cC(T)$.}
	
	Let $u_2$ (resp. $u_1$) be the entry point of $\gamma$ in $\Bds$ (resp. $\Bus$).
	Since $\Bts$ is a typical box, we can define:
	\begin{itemize}
		\item $\pi$ as the path $\fpm(u_2,u_1)$ defined in Definition \ref{d: central straight segment et chemin pi flèche.},
		\item $\cC(T)$ the first vertex belonging to the selected straight segment between $u_2$ and $u_1$ and satisfying property (iii) of a typical box. Such a vertex exists by Remark \ref{Remk: Selected straight segment non vide.} and since $\Bts$ is a typical box in the environment $T$. Recall that in the case \ref{c: Case I.}, for every edge $e \in B_\infty(\cC(T),r_P)$, $T(e) = \infty$ and in the case \ref{c: Case III.}, $\theta_{-\cC(T)} T \in \cT$. 
	\end{itemize}
	We use the expression "forbidden zone" to refer to $B_\infty(\cC(T),r_P) \setminus \partial B_\infty(\cC(T),r_P)$. This is the place where we want to place the pattern taken by the shortcut in the modified environment.
	
	\paragraph*{Properties of the forbidden zone.}
	 
	\begin{lemma}\label{l: Properties of the forbidden zone.}
		\begin{enumerate}
			\item The path $\gamma$ does not visit any vertex of the forbidden zone.
			\item The ball $B_\infty(\cC(T),\lll)$ is contained in the forbidden zone.
			\item The forbidden zone is contained in $\Bds$ and for every $c$ in the forbidden zone and every $z \in \partial \Bds$, \[\|z-c\|_1 \ge \frac{(r_2-r_1)N}{2}.\]
		\end{enumerate}
	\end{lemma}

	\begin{proof}
		\begin{enumerate}
			\item In the case \ref{c: Case I.}, every edge of the forbidden zone has an infinite passage time although $\gamma$ has a finite passage time. In the case \ref{c: Case III.}, it follows from Lemma \ref{Lemme pour les clusters dans les typical boxes dans le cas finin non borné.} and the fact that $\gamma$ crosses this box and is a geodesic in the environment $T$.
			\item It comes from the inequality $r_P > \lll + 1$ by \eqref{On fixe rP.}.
			\item By Definition \ref{d: central straight segment et chemin pi flèche.}, $\cC(T)$ belongs to $\Bds$ and the distance between $\cC(T)$ and $(\Bds)^c$ is at least $\displaystyle \|z-c\|_1 \ge \frac{(r_2-r_1)N}{2}$. We get the result using that, for every $c$ in the forbidden zone, $\|c-\cC(T)\|_1 \le d r_P$.
		\end{enumerate}
	\end{proof}

	\paragraph*{Construction of the shortcut $\pi'$.}
	
	\begin{figure}
		\begin{center}
			\begin{tikzpicture}[scale=0.7]
				\clip (-1,10) rectangle (17,22.5);
				\draw (0,0) rectangle (20,20);
				\draw (8.5,8.5) rectangle (11.5,11.5);
				\draw (13,20) node[above] {$\Bds$};
				\draw (8.5,11.5) node[above] {$\Bus$};
				\draw[line width=2pt] (7,20) -- (11,20) -- (11,11.5); 
				\draw[pattern= north west lines, pattern color=gray!70] (10,13) rectangle (12,15);
				\draw[Peach,fill=Peach] (10.75,13.75) rectangle (11.25,14.25);
				\draw[line width=2pt,Green] (11,18) -- (11,14.35) -- (10.65,14.35) -- (10.65,14) -- (10.8,14) -- (10.8,14.15) -- (11,14.15) -- (11,14) -- (11.35,14) -- (11.35,13.65) -- (11,13.65) -- (11,12);
				\draw[line width=1.5pt,Gray] (6,24) .. controls (10,23.5) and (7,21) .. (7,20);
				\draw[line width=1.5pt,Gray] (7,20) .. controls (7,19) and (9,17.5) .. (11,18);
				\draw[line width=1.5pt,Gray] (11,18) .. controls (16,19) and (14,12) .. (11,12);
				\draw[line width=1.5pt,Gray] (11,12) .. controls (10.5,12) and (10.5,11.5) .. (11,11.5);
				\draw[line width=1.5pt,Gray] (11,11.5) .. controls (17,11.5) and (18,8) .. (17,5);
				\draw (7,20) node[above left] {$u_2$};
				\draw (11,11.5) node[below] {$u_1$};
				\draw (7,20) node {$\bullet$};
				\draw (11,11.5) node {$\bullet$};
				\draw (11,18) node[below right] {$u_\pi$};
				\draw (11,12) node[above left] {$v_\pi$};
				\draw (11,18) node {$\bullet$};
				\draw (11,12) node {$\bullet$};
				\draw (14,15) node[right,Gray] {$\gamma$};
				\draw (11,16) node[Green,right] {$\pi'$};
			\end{tikzpicture}
			\caption{Example of construction of the shortcut $\pi'$ in two dimensions. The shortcut $\pi'$ is represented in green, the path $\gamma$ in gray, the ball $B_\infty(\cC(T),\lll)$ in orange and the forbidden zone by the hatched area}\label{f: Construction du shortcut pi'.}
		\end{center}
	\end{figure}
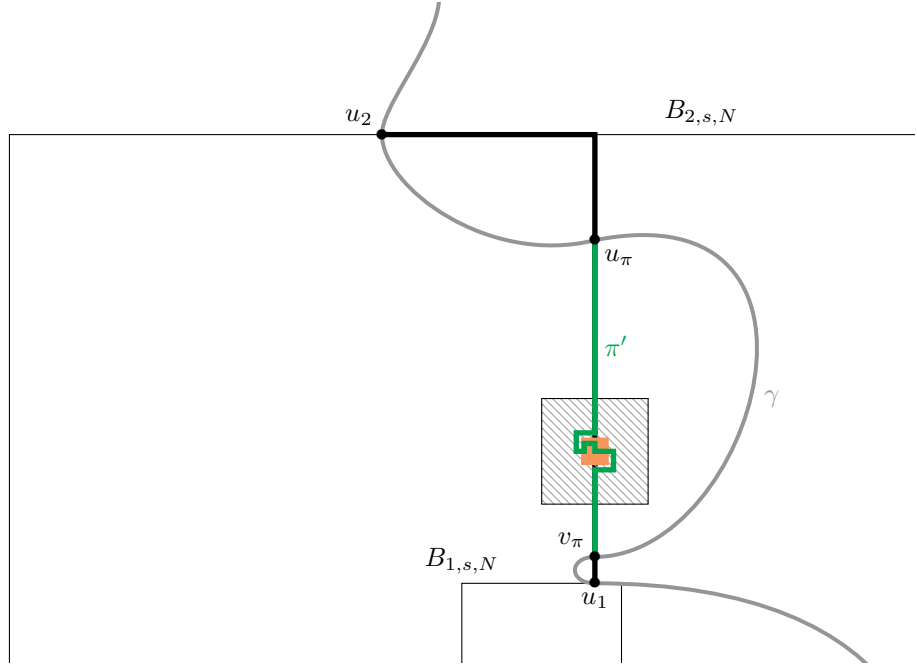
	
	Let $u_\pi$ be the last vertex of $\pi$ belonging to $\gamma$ before $\pi$ visits the forbidden zone and $v_\pi$ be the first vertex of $\pi$ belonging to $\gamma$ after the forbidden zone. 
	One can check that we can build a path, denoted by $\pi'$ for the remaining of the proof, such that:
	\begin{itemize}
		\item $\pi'$ is a self-avoiding path from $u_\pi$ to $v_\pi$,
		\item $\pi'$ is the concatenation of the subpath of $\pi$ between $u_\pi$ and the forbidden zone, then of a path entirely contained in the forbidden zone and then of the subpath of $\pi$ between the forbidden zone and $v_\pi$,
		\item $\pi'$ visits $B_\infty(\cC(T),\lll)$ for the first time in $\theta_{-\cC(T)} \ulm$ and for the last time in $\theta_{-\cC(T)} \vlm$, and between these two vertices, $\pi'$ is entirely contained in $B_\infty(\cC(T),\lll)$. Furthermore, in the case \ref{c: Case I.}, between these two vertices, $\pi'$ is equal to $\theta_{\cC(T)} \pi_\infty$ where $\pi_\infty$ is defined in Assumption \ref{OP7}. Note that, if $T' \in \cB^*(T)$ (where $\cB^*(T)$ is defined in Lemma \ref{Lemme pour les inégalités avec les indicatrices dans la modification.}), then $\theta_{\cC(T)} \pi_\infty$ has finite passage time in the environment $\Tu$,
		\item we have an upper bound for the number of edges in $\pi'$:
		\begin{equation}
			|\pi'|_e \le \|u_\pi-v_\pi\|_1 + |B_\infty(0,\lll+1)|_e.\label{eq: taille du chemin pi prime.}
		\end{equation}
	\end{itemize}
	Note that, by the second item above and the definition of $u_\pi$ and $v_\pi$, $\pi'$ only has $u_\pi$ and $v_\pi$ in common with $\gamma$. See Figure \ref{f: Construction du shortcut pi'.} for an example of construction of the shortcut $\pi'$.

	\paragraph{Beginning of the modification.}
	
	There are two cases for the beginning of the modification depending on whether $\r=0$ or $\r>0$ and on the number of edges in $\gamma_{u_\pi,v_\pi}$, denoted by $|\gamma_{u_\pi,v_\pi}|_e$. We have to distinguish two cases because we must be able to have a lower bound on the passage time of $\gamma_{u_\pi,v_\pi}$. To this aim, if $\gamma_{u_\pi,v_\pi}$ takes enough edges, we can use the second property of a typical box and if it is not the case, we can have a lower bound using the number of edges of $\gamma_{u_\pi,v_\pi}$ and $\r$ if $\r>0$. If $\r=0$ and if we can not use the second property of a typical box, then we use the modification to increase the passage times of $\gamma_{u_\pi,v_\pi}$. We describe the modification in each case. See Figure \ref{f: Modification.} for a representation of the objects involved in the modification.
	
	\subparagraph*{Case A: assume $\r=0$ and $|\gamma_{u_\pi,v_\pi}|_e < N$.}
	The beginning of the modification is the following.
	\begin{itemize}
		\item The edges of $\Emid(T)$ are the edges $e$ belonging to $\gamma_{u_\pi,v_\pi}$ and such that $T(e) < \r + \delta$.
		\item Recall that $\cC(T)$ is defined at the beginning of Section \ref{Sous-section modification}. 
		\item The edges of $\Em(T)$ are the edges $e$ of $\Bts$ satisfying the following two conditions:
		\begin{itemize}
			\item $e$ belongs to ($\gamma \setminus \gamma_{u_\pi,v_\pi}$ or to $\pi'$) but not to $B_\infty(\cC(T),\lll)$,
			\item $T(e) \ge \r + \delta'$.
		\end{itemize}
	\end{itemize}

	\subparagraph*{Case B: assume $\r>0$ or $|\gamma_{u_\pi,v_\pi}|_e \ge N$.}
	The beginning of the modification is the following.
	\begin{itemize}
		\item $\Emid(T) = \emptyset$.
		\item Recall that $\cC(T)$ is defined at the beginning of Section \ref{Sous-section modification}. 
		\item The edges of $\Em(T)$ are the edges $e$ of $\Bts$ satisfying the following two conditions:
		\begin{itemize}
			\item $e$ belongs to $\pi'$ but not to $B_\infty(\cC(T),\lll)$,
			\item $T(e) \ge \r + \delta'$.
		\end{itemize}
	\end{itemize}

	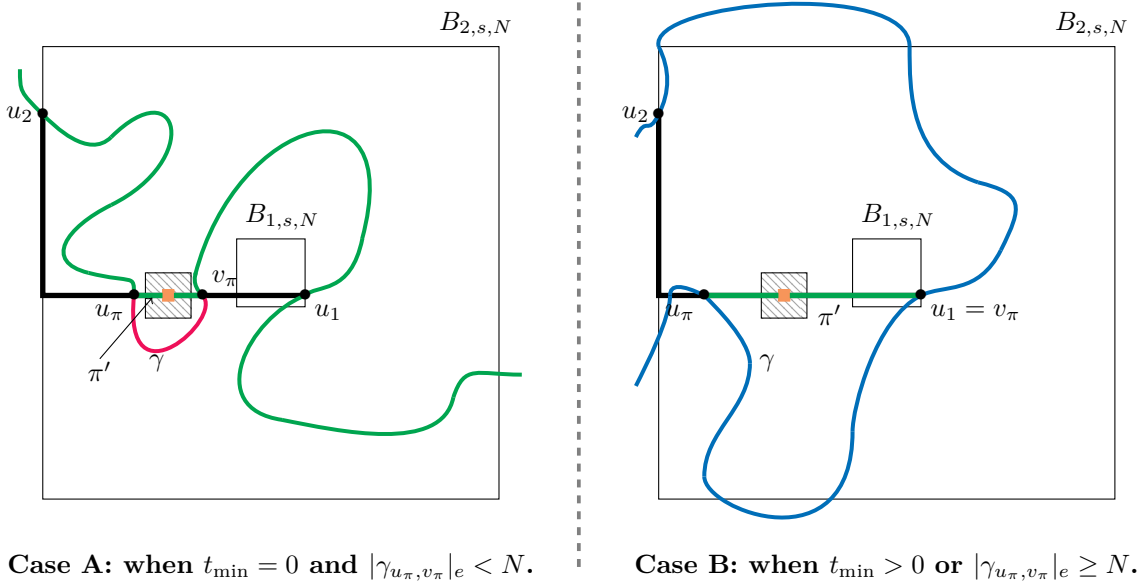
\begin{figure}[!t]
		\begin{center}
			\begin{tikzpicture}[scale=0.3]
				\draw (0,0) rectangle (20,20);
				\draw (8.5,8.5) rectangle (11.5,11.5);
				\draw (19,20) node[above] {$\Bds$};
				\draw (10.5,11.5) node[above] {$\Bus$};
				\draw[line width=2pt] (0,17) -- (0,9) -- (11.5,9); 
				\draw[pattern= north west lines, pattern color=gray!70] (4.5,8) rectangle (6.5,10);
				\draw[Peach,fill=Peach] (5.25,8.75) rectangle (5.75,9.25);
				\draw[line width=1.5pt,Green] (-1,19) .. controls (-1,18) and (-1,18) .. (0,17);
				\draw[line width=1.5pt,Green] (0,17) .. controls (0.5,16.5) and (2,15) .. (3,16);
				\draw[line width=1.5pt,Green] (3,16) .. controls (6,19) and (7,15) .. (3,13);
				\draw[line width=1.5pt,Green] (3,13) .. controls (1,12) and (0,10.5) .. (2,10);
				\draw[line width=1.5pt,Green] (2,10) .. controls (4,9.5) and (4,10) .. (4,9);
				\draw[line width=1.5pt,OrangeRed] (4,9) .. controls (3.5,4) and (8,8) .. (7,9);
				\draw[line width=1.5pt,Green] (7,9) .. controls (6.5,9.5) and (6.7,9.9) .. (6.7,10);
				\draw[line width=1.5pt,Green] (6.68,9.9) .. controls (8.5,17) and (16,19) .. (14,11.5);
				\draw[line width=1.5pt,Green] (14.01,11.51) .. controls (13.5,9) and (12,9.5) .. (11.5,9);
				\draw[line width=1.5pt,Green] (11.5,9) .. controls (9,8) and (7,4) .. (10.5,3.5);
				\draw[line width=1.5pt,Green] (10.49,3.5) .. controls (13,3) and (17.5,2) .. (18,4.5);
				\draw[line width=1.5pt,Green] (18,4.5) .. controls (18.5,6) and (19,5.5) .. (21,5.5);
				\draw[line width=2pt,Green] (4,9) -- (5.25,9);
				\draw[line width=2pt,Green] (5.75,9) -- (7,9);
				\draw[->] (2.5,6.2) -- (4.8,8.9);
				\draw (0,17) node[left] {$u_2$};
				\draw (11.5,9) node[below right] {$u_1$};
				\draw (0,17) node {$\bullet$};
				\draw (11.5,9) node {$\bullet$};
				\draw (4,9) node[below left] {$u_\pi$};
				\draw (7,9) node[above right] {$v_\pi$};
				\draw (4,9) node {$\bullet$};
				\draw (7,9) node {$\bullet$};
				\draw (5,6) node {$\gamma$};
				\draw (2.5,6.5) node[below] {$\pi'$};
				\draw (10,-3) node {\textbf{Case A: when $\r=0$ and $|\gamma_{u_\pi,v_\pi}|_e<N$.}};
				
				\begin{scope}[xshift=27cm]
					\draw (0,0) rectangle (20,20);
					\draw (8.5,8.5) rectangle (11.5,11.5);
					\draw (19,20) node[above] {$\Bds$};
					\draw (10.5,11.5) node[above] {$\Bus$};
					\draw[line width=2pt] (0,17) -- (0,9) -- (11.5,9); 
					\draw[pattern= north west lines, pattern color=gray!70] (4.5,8) rectangle (6.5,10);
					\draw[Peach,fill=Peach] (5.25,8.75) rectangle (5.75,9.25);
					\draw[line width=1.5pt,color=NavyBlue] (-1,16) .. controls (-0.5,17) and (-0.5,16) .. (0,17);
					\draw[line width=1.5pt,color=NavyBlue] (0,17) .. controls (0.5,18) and (0.5,19) .. (0,20);
					\draw[line width=1.5pt,color=NavyBlue] (0,20) .. controls (-0.5,22) and (11,23) .. (11,20);
					\draw[line width=1.5pt,color=NavyBlue] (11,20) .. controls (11,18) and (11,15) .. (14,14);
					\draw[line width=1.5pt,color=NavyBlue] (14,14) .. controls (16,13.5) and (16,13) .. (15,11);
					\draw[line width=1.5pt,color=NavyBlue] (15,11) .. controls (14,9) and (12.5,9.5) .. (11.5,9);
					\draw[line width=1.5pt,color=NavyBlue] (11.5,9) .. controls (9.5,8) and (8.5,4) .. (8.5,3);
					\draw[line width=1.5pt,color=NavyBlue] (8.5,3) .. controls (8.5,-3) and (2,-0.5) .. (2,1);
					\draw[line width=1.5pt,color=NavyBlue] (2,1) .. controls (2,3) and (4,4) .. (4,6);
					\draw[line width=1.5pt,color=NavyBlue] (4,6) .. controls (4,7) and (2.5,8.5) .. (2,9);
					\draw[line width=1.5pt,color=NavyBlue] (2,9) .. controls (1,9.5) and (0.5,9.5) .. (0.5,9);
					\draw[line width=1.5pt,color=NavyBlue] (0.5,9) .. controls (0.5,8) and (-0.5,6) .. (-1,5);
					\draw[line width=2pt,Green] (2,9) -- (5.25,9);
					\draw[line width=2pt,Green] (5.75,9) -- (11.5,9);
					\draw (0,17) node[left] {$u_2$};
					\draw (11.5,9) node[below right] {$u_1=v_\pi$};
					\draw (0,17) node {$\bullet$};
					\draw (11.5,9) node {$\bullet$};
					\draw (2,9) node[below left] {$u_\pi$};
					\draw (2,9) node {$\bullet$};
					\draw (4,6) node[right] {$\gamma$};
					\draw (7.5,9) node[below] {$\pi'$};
					\draw (10,-3) node {\textbf{Case B: when $\r>0$ or $|\gamma_{u_\pi,v_\pi}|_e \ge N$.}};
				\end{scope}
			
				\draw[line width=1.4pt,color=gray,dashed] (23.5,-3) -- (23.5,22);
			\end{tikzpicture}
			\caption{The modification. In the cases A and B, the pattern centered in $\cC(T)$ is represented in orange and the forbidden zone by the hatched area.
			In the case A, $\gamma$ is the path composed by the edges in green up to $u_\pi$, then the edges in red and then the edges in green from $v_\pi$. In the case B, $\gamma$ is the path composed by the edges in blue. 
			In the cases A and B, every edge of $\Em(T)$ belongs to the green part of the figure. When the modification is successful, the passage times of the edges in green which are greater than or equal to $\r+\delta'$ are replaced by passage times smaller than $\r+\delta'$. In the case A, every edge of $\Emid(T)$ belongs to the red part. When the modification is successful, the passage times of the edges in red which are smaller than $\r+\delta$ are replaced by passage times belonging to $(\r+\delta,\nu_0)$. In the case B, the passage times of the edges in blue are not modified. The boundary of $\Bts$ is not represented even if $\Bds$ is included in $\Bts$. The edges of $\Einf(T)$ or $\Ep(T)$ (depending on the case \ref{c: Case I.} or \ref{c: Case III.}) are all edges which are not in green, red, orange and blue.}\label{f: Modification.}
		\end{center}
	\end{figure}

	\paragraph{End of the modification.}\label{Paragraphe fin de la modification.}
	
	Up to now, we have defined $\Em(T)$, $\Emid(T)$ and $\cC(T)$. Note that $\Em(T)$ and $\Emid(T)$ disjoint sets included in $\Bts \cap (\gamma \cup \pi')$. 
	It remains to define $\Ep(T)$ and $\Einf(T)$. 
	There are two cases depending on whether there can be edges with infinite passage times or not.
	
	\subparagraph*{In the case \ref{c: Case I.}.}
	In this case, $\Ep(T)=\emptyset$ and the edges of $\Einf(T)$ are the edges of $\Bts$ which does not belong to $B_\infty(\cC(T),\lll)$, to $\pi'$ or to $\gamma$. 
	
	\subparagraph*{In the case \ref{c: Case III.}.}
	In this case, $\Einf(T) = \emptyset$ and the edges of $\Ep(T)$ are the edges of $\Bts$ which does not belong to $B_\infty(\cC(T),\lll)$, to $\pi'$ or to $\gamma$.
	
	\subsubsection{Proof of Lemma \ref{Lemme pour les inégalités avec les indicatrices dans la modification.}: consequences of the modification}\label{Section conséquences de la modification.}
	
	Assume for the remaining of the proof that the event 
	\[\{T \in \cM^k(\l)\} \cap \{S^k_\l(T)=s\} \cap \{T' \in \cB^*(T)\} \text{ occurs,}\]
	where $\cB^*(T)$ is defined in (ii) of Lemma \ref{Lemme pour les inégalités avec les indicatrices dans la modification.}. 
	We now state some consequences of the modification useful for the following. 
	\begin{enumerate}[label=(\alph*)]
		\item\label{Cofj en premier} If $|\gamma_{u_\pi,v_\pi}|_e < N$, every edge in $\gamma_{u_\pi,v_\pi}$ belongs to $\Bts$. Thus, every edge belonging to $\Emid(T)$ belongs to $\Bts$.
		\item\label{Cofc} $B_\infty(\cC(T),\lll)$ is entirely contained in $\Bts$.
		\item\label{Cofa} We have $\Tu(\gamma) \le T(\gamma)$. Furthermore, for all vertices $u$ and $v$ in $\partial \Bts$ visited by $\gamma$, $\Tu(\gamma_{u,v}) \le T(\gamma_{u,v})$.
		\item\label{Cofb} $\Tu(\pi')<\Tu(\gamma_{u_\pi,v_\pi} \cap \Bts)$.
		\item\label{Cofde} There is only one pattern entirely contained in $\Bts$ in the environment $\Tu$, which is the one centered in $\cC(T)$. 
	\end{enumerate}
	\begin{enumerate}[label=(\alph*),resume]
		\item\label{Cofhg} Let $\pi_0$ be a self-avoiding path from $\partial \Bts$ to $\partial \Bts$ entirely contained in $\Bts$ such that in the environment $\Tu$:
		\begin{itemize}
			\item in the case \ref{c: Case I.}, it has a finite passage time,
			\item in the case \ref{c: Case III.}, it does not take any edge whose passage time is greater than $\nu_1$,
			\item in the cases \ref{c: Case I.} and \ref{c: Case III.}, it does not take any pattern contained in $\Bts$.
		\end{itemize}
		Then every edge of $\Bts$ belonging to $\pi_0$ belongs to $\gamma$. 
		\item\label{Cofi} Let $\tpi$ be a $k$-penalized path from $0$ to $x$ in the environment $\Tu$ with $\Tu(\tpi)$ finite. Then \[T(\tpi) - \Tu(\tpi) \le T(\gamma) - \Tu(\gamma).\] In other words, no $k$-penalized path can save more time than $\gamma$ during the modification.
	\end{enumerate}

	\begin{proof}[Proof of \ref{Cofj en premier}]
		Since $\pi$ is included in $\Bds$, $u_\pi \in \Bds$. Thus 
		\begin{equation}
			\|u_\pi-sN\|_1 \le r_2N. \label{eq: preuve de (a) 1.}
		\end{equation}
		Assume that $|\gamma_{u_\pi,v_\pi}|_e < N$ and let $z$ be a vertex visited by $\gamma_{u_\pi,v_\pi}$. We have
		\begin{equation}
			\|z-u_\pi\|_1 \le N. \label{eq: preuve de (a) 2.}
		\end{equation}
		Combining \eqref{eq: preuve de (a) 1.} and \eqref{eq: preuve de (a) 2.} gives
		\[\|z-sN\|_1 \le (r_2+1)N < r_3N,\]
		since $r_3 > r_2 + 1$ by \eqref{On fixe r1 r2 et r3.}. 
		Hence $z$ belongs to $\Bts$. 
		For the second part of the property, in the case B of the modification, there is no edge in $\Emid(T)$ and in the case A of the modification all the edges of $\Emid(T)$ belong to $\gamma_{u_\pi,v_\pi}$.
	\end{proof}

	\begin{proof}[Proof of \ref{Cofc}]
		By Lemma \ref{l: Properties of the forbidden zone.}, $B_\infty(\cC(T),\lll)$ is contained in the forbidden zone and the forbidden zone is contained in $\Bds$. Since $r_3>r_2$, $\Bds$ is contained in $\Bts$.
	\end{proof}

	\begin{proof}[Proof of \ref{Cofa}]
		In the case $B$ of the modification, the passage time of every edge of $\gamma$ in $\Tu$ is equal to its passage time in $T$. Thus, property \ref{Cofa} holds. 
		Now, assume the case $A$ of the modification and recall that in this case $\r=0$. The only edges of $\gamma$ whose passage times in $\Tu$ are strictly greater than their passage times in $T$ are those in $\gamma_{u_\pi,v_\pi}$ and all of these edges are contained in $\Bts$ by property \ref{Cofj en premier} above. 
		Hence, to prove property \ref{Cofa}, it is sufficient to prove that $\Tu(\gamma_{u,v}) \le T(\gamma_{u,v})$ when $u$ is the last vertex in $\partial \Bts$ visited by $\gamma$ before it visits $u_\pi$ and $v$ is the first vertex in $\partial \Bts$ visited by $\gamma$ after it visits $v_\pi$. So, let $u$ and $v$ be these vertices. First, since $u_\pi$ belongs to $\Bds$ and $u$ to $\partial \Bts$, we have \[\|u-u_\pi\|_1 \ge (r_3-r_2)N \ge N,\] since $r_3 >
		 r_2+1$ by \eqref{On fixe r1 r2 et r3.}. Hence, by the first property of a typical box, $\gamma_{u,u_\pi}$ takes at least $\rho (r_3-r_2) N$ edges whose passage times is greater than $\delta$. Since $\delta > \delta'$ by \eqref{eq: On fixe delta prime.}, all of these edges belongs to $\Em(T)$ and there are no edges of $\gamma_{u,u_\pi}$ whose passage times have been increased. Thus
	 	\begin{equation}
	 		\Tu(\gamma_{u,u_\pi})-T(\gamma_{u,u_\pi}) \le - \rho(r_3-r_2)N(\delta-\delta'). \label{eq: propriété (b) 1.}
	 	\end{equation}
 		The same arguments give
 		\begin{equation}
 			\Tu(\gamma_{v_\pi,v})-T(\gamma_{v_\pi,v}) \le - \rho(r_3-r_2)N(\delta-\delta'). \label{eq: propriété (b) 2.}
 		\end{equation}
 		Furthermore, since the only edges of $\gamma_{u_\pi,v_\pi}$ whose passage times have been modified belong to $\Emid(T)$,
 		\begin{equation}
 			\Tu(\gamma_{u_\pi,v_\pi})-T(\gamma_{u_\pi,v_\pi}) \le N \nu_0. \label{eq: propriété (b) 3.}
 		\end{equation}
 		Thus, we get
 		\begin{align*}
 			\Tu(\gamma_{u,v})-T(\gamma_{u,v}) &= \Tu(\gamma_{u,u_\pi})-T(\gamma_{u,u_\pi}) + \Tu(\gamma_{u_\pi,v_\pi})-T(\gamma_{u_\pi,v_\pi}) + \Tu(\gamma_{v_\pi,v})-T(\gamma_{v_\pi,v}) \\
 			& \le N (\nu_0 - 2\rho (\delta-\delta') (r_3 - r_2)) \quad \text{ by \eqref{eq: propriété (b) 1.}, \eqref{eq: propriété (b) 2.} and \eqref{eq: propriété (b) 3.},} \\
 			& \le 0,
 		\end{align*}
 		since $r_3 > 2 r_2$, $\delta > 2 \delta'$, $r_3$ is large enough compared to $\nu_0$ by \eqref{On fixe r1 r2 et r3.} and $N \ge 1$ by \eqref{eq: On fixe N.}.
	\end{proof}

	\begin{figure}
		\begin{center}
			\begin{tikzpicture}[scale=1]
				\clip (-1.5,6.5) rectangle (10.5,12);
				\draw (0,0) rectangle (20,20);
				\draw (8.5,8.5) rectangle (11.5,11.5);
				\draw (0,11) node[left] {$\Bds$};
				\draw (8.5,11) node[left] {$\Bus$};
				\draw[pattern= north west lines, pattern color=gray!70] (4.5,8) rectangle (6.5,10);
				\draw[Peach,fill=Peach] (5.25,8.75) rectangle (5.75,9.25);
				\draw[line width=1.5pt,NavyBlue] (-1,6) .. controls (-2,6.5) and (4,7) .. (4,9);
				\draw[line width=1.5pt,NavyBlue] (4,9) .. controls (3.5,12) and (7,11) .. (7,9);
				\draw[line width=1.5pt,NavyBlue] (7,9) .. controls (7.3,7) and (12,5) .. (14,6);
				\draw (5.5,11) node {$\gamma$};
				\draw[line width=2pt,Green] (4,9) -- (5.25,9);
				\draw[line width=2pt,Green] (5.75,9) -- (7,9);
				\draw (4.8,9) node[below] {$\pi'$};
				\draw[<->] (6.6,8) -- (6.6,8.9);
				\draw (6.55,8.5) node[right] {$r_P$};
				\draw (0,17) node[left] {$u_2$};
				\draw (11.5,9) node[below right] {$u_1$};
				\draw (0,17) node {$\bullet$};
				\draw (11.5,9) node {$\bullet$};
				\draw (4,9) node[below left] {$u_\pi$};
				\draw (7,9) node[above right] {$v_\pi$};
				\draw (4,9) node {$\bullet$};
				\draw (7,9) node {$\bullet$};
			\end{tikzpicture}
			\caption{A picture to illustrate \eqref{eq: Preuve de Cofb figure.}. The legend is the same as in Figure \ref{f: Modification.}. In this example in two dimensions, when $\r>0$ and $|\gamma_{u_\pi,v_\pi}|_e < N$ (a special case of case B of the modification), $\gamma_{u_\pi,v_\pi}$ has to take $\|u_\pi-v_\pi\|_1$ edges in the direction $\epsilon_1$ and also at least $2r_P$ edges in the direction $\epsilon_2$ to avoid the forbidden zone.}\label{f: Preuve de (d).}
		\end{center}
	\end{figure}
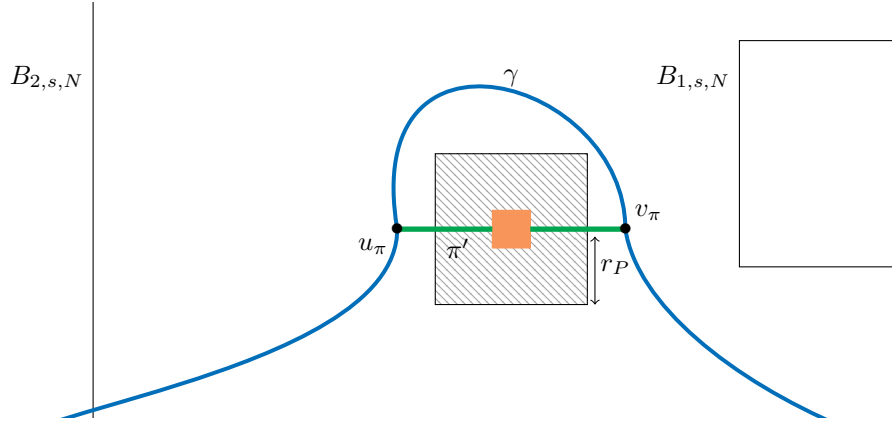

	\begin{proof}[Proof of \ref{Cofb}]
		First, in all cases, by \eqref{eq: taille du chemin pi prime.}, and by \eqref{On fixe TLambda dans le cas infini.} and \eqref{On fixe lLambda et TLambda dans le cas fini non borné.}, 
		\begin{equation}
			\Tu(\pi') \le (\|u_\pi-v_\pi\|_1 + |B_\infty(0,\lll+1)|_e) (\r + \delta') + T^\Lambda. \label{eq: Cofb2.}
		\end{equation}
		To conclude this proof, we distinguish three cases.
		
		\textbf{If $|\gamma_{u_\pi,v_\pi}|_e \ge N$ (case B of the modification).}
		First, let us prove that
		\begin{equation}
			\Tu(\gamma_{u_\pi,v_\pi}\cap\Bts) \ge \|u_\pi-v_\pi\|_1 (\r + \delta).\label{eq: Cofb3.}
		\end{equation}
		Since this is the case $B$ of the modification, the edges of $\gamma_{u_\pi,v_\pi}$ have not been modified. Thus, if $\gamma_{u_\pi,v_\pi}$ is entirely contained in $\Bts$ the second property of the typical boxes gives \eqref{eq: Cofb3.}. If $\gamma_{u_\pi,v_\pi}$ is not entirely contained in $\Bts$, let $u'$ be the first vertex of $\partial \Bts$ visited by $\gamma_{u_\pi,v_\pi}$. The vertex $u_\pi$ is in $\Bds$, so $\|u'-u_\pi\|_1 \ge (r_3-r_2) N \ge N$ since $r_3 > r_2$ by \eqref{On fixe r1 r2 et r3.}. Using that $\|u_\pi-v_\pi\|_1 \le d(r_2+r_1) N$, the fact that $r_3-r_2 > d(r_2+r_1)$ by \eqref{On fixe r1 r2 et r3.} and the second property of a typical box gives:
		\[\Tu(\gamma_{u_\pi,v_\pi}\cap\Bts) \ge \underbrace{\|u'-u_\pi\|_1}_{\ge (r_3-r_2)N} (\r + \delta) \ge (r_3-r_2)N(\r+\delta) \ge \|u_\pi-v_\pi\|_1 (\r + \delta).\]
		This concludes the proof of \eqref{eq: Cofb3.}.
		Then, combining \eqref{eq: Cofb2.} and \eqref{eq: Cofb3.}, and using that $\|u_\pi-v_\pi\|_1 \ge r_P$, we have
		\begin{align*}
			\Tu(\gamma_{u_\pi,v_\pi}\cap\Bts) - \Tu(\pi') & \ge \|u_\pi-v_\pi\|_1 (\delta - \delta') - |B_\infty(0,\lll+1)|_e (\r + \delta') - T^\Lambda \\
			& \ge r_P (\delta - \delta') - |B_\infty(0,\lll+1)|_e (\r + \delta') - T^\Lambda > 0,
		\end{align*} by \eqref{On fixe rP.} and since $\delta > 2 \delta'$. 
		
		\textbf{If $\r=0$ and $|\gamma_{u_\pi,v_\pi}|_e<N$ (case A of the modification).}
		In the environment $\Tu$, all the edges $e$ belonging to $\gamma_{u_\pi,v_\pi}$ have a time greater than $\r + \delta$ and the property \ref{Cofj en premier}, $\gamma_{u_\pi,v_\pi} \cap \Bts = \gamma_{u_\pi,v_\pi}$. Hence
		\begin{equation}
			\Tu(\gamma_{u_\pi,v_\pi}\cap\Bts) \ge \|u_\pi-v_\pi\|_1 (\r + \delta).\label{eq: Cofb1.}
		\end{equation}
		We conclude the proof of this case as the previous one combining \eqref{eq: Cofb2.} and \eqref{eq: Cofb1.}.
	
		\textbf{If $\r>0$ and $|\gamma_{u_\pi,v_\pi}|_e < N$ (case B of the modification).}
		In this case, since $\cC(T)$ belongs to the selected straight segment between $u_2$ and $u_1$ and since $v_\pi$ is visited by $\pi$ after the forbidden zone, $v_\pi$ belongs to the selected straight segment. Thus, the distance between $v_\pi$ and $(\Bds)^c$ is greater than or equal to $\displaystyle \frac{(r_2-r_1)N}{2}$. Then, since $|\gamma_{u_\pi,v_\pi}|_e < N$, we have $\|u_\pi-v_\pi\|_1 < N$. Since $r_2 > 4 r_1$ by \eqref{On fixe r1 r2 et r3.}, $u_\pi$ does not belong to $\partial \Bds \cap \pi$ and thus $\pi_{u_\pi,v_\pi}$ takes edges in only one direction: the direction of the selected straight segment. Denote this direction by $\epsilon_i$. Then $\gamma_{u_\pi,v_\pi}$ has to take $\|u_\pi-v_\pi\|_1$ edges in the direction $\epsilon_i$ but it can not take edges of the forbidden zone (see Figure \ref{f: Preuve de (d).}). Hence 
		\begin{equation}
			|\gamma_{u_\pi,v_\pi}|_e \ge \|u_\pi-v_\pi\|_1 + 2 r_P.\label{eq: Preuve de Cofb figure.}
		\end{equation}
		By the property \ref{Cofj en premier}, we have $|\gamma_{u_\pi,v_\pi}|_e = |\gamma_{u_\pi,v_\pi} \cap \Bts|_e$ and thus 
		\begin{equation}
			\Tu(\gamma_{u_\pi,v_\pi}\cap\Bts) \ge (\|u_\pi-v_\pi\|_1 + 2 r_P) \r. \label{eq: Cofb4.}
		\end{equation}
		Combining \eqref{eq: Cofb2.} and \eqref{eq: Cofb4.} gives
		\begin{align*}
			\Tu(\gamma_{u_\pi,v_\pi}\cap\Bts) - \Tu(\pi') & \ge 2 r_P \r - \|u_\pi-v_\pi\|_1 \delta'  - |B_\infty(0,\lll+1)|_e (\r+\delta') - T^\Lambda \\
			& \ge 2 r_P \r - N \delta'  - |B_\infty(0,\lll+1)|_e (\r+\delta') - T^\Lambda  > 0,
		\end{align*}
		since $N \delta' \le 1$ by \eqref{eq: On fixe delta prime.} and since $r_P$ is large enough by \eqref{On fixe rP.}.
	\end{proof}

	\begin{proof}[Proof of \ref{Cofde}]
		\textbf{In the case \ref{c: Case I.}.}
		Recall that the pattern satisfies the boundary condition (see Definition \ref{Définition boundary condition.}). Thus, if there is a pattern entirely contained in $\Bts$ centered in a vertex $z$, it implies that there exists a path of length $2(\lll-1)$ such that: 
		\begin{itemize}
			\item it goes from $z-(\lll-1)\epsilon_1+(\lll-1)\epsilon_2$ to $z+(\lll-1)\epsilon_1+(\lll-1)\epsilon_2$ only using edges in the direction $\epsilon_1$,
			\item its passage time is finite,
			\item there exists no path with finite passage time from $\partial \Bts$ to one of its vertices.
		\end{itemize}
		  
		In the environment $\Tu$, the only edges with finite passage times are edges belonging to $\gamma$, $\pi'$, $\cS_{\cC(T),\lll}$ and $B_\infty(\cC(T),\lll-3)$. For every vertex $z$ belonging to $\gamma$ and $\pi'$, there exists a path from $\partial \Bts$ to $z$ with finite passage time. Furthermore, there is no path of length $2(\lll-1)$ with finite passage time using only edges in the direction $\epsilon_1$ having at least one edge in $B_\infty(\cC(T),\lll-3)$ and which does not visit any vertex of $\pi'$. Thus, the only path of length $2(\lll-1)$ satisfying the three conditions above is the one from $\cC(T)-(\lll-1)\epsilon_1+(\lll-1)\epsilon_2$ to $\cC(T)+(\lll-1)\epsilon_1+(\lll-1)\epsilon_2$ and the only pattern entirely contained in $\Bts$ is the one centered in $\cC(T)$. 
		
		\textbf{In the case \ref{c: Case III.}.}
		If, in the environment $\Tu$, there is a pattern entirely contained in $\Bts$ centered in a vertex $z$, then for every edge $e \in \partial B_\infty(z,\lll)$, $T(e) \in (\nu_1,\nu_2)$ by the assumption (AF-4') in Section \ref{Sous-section bordure des motifs.}. 
		Since $\Bts$ is a typical box in the environment $T$, for every edge $e \in \Bts$, $T(e) \le \nu_1$. The only edges $e$ such that $\Tu(e) > T(e)$ are:
		\begin{itemize}
			\item the edges of $\Emid(T)$ when $\r=0$ but for every edge $e \in \Emid(T)$, $\Tu(e) < \nu_0 < \nu_1$ since $\nu_1>\nu_0$, 
			\item the edges of $\Ep(T)$ but for every edge $e \in \Ep(T)$, $\Tu(e) > \nu_2$,
			\item some edges in $B_\infty(\cC(T),\lll)$.
		\end{itemize}
		Thus, if an edge $e \in \Bts$ is such that $\Tu(e) \in (\nu_1,\nu_2)$, this edge belongs to $B_\infty(\cC(T),\lll)$. Thus there is only one pattern entirely contained in $\Bts$ which is the one centered in $\cC(T)$. 
	\end{proof}

	\begin{proof}[Proof of \ref{Cofhg}]
		Let $\pi_0$ be a path from $\partial \Bts$ to $\partial \Bts$ entirely contained in $\Bts$ such that in the environment $\Tu$, it has a finite passage time in the case \ref{c: Case I.} and it does not take any edge whose passage time is greater than $\nu_1$ in the case \ref{c: Case III.}. Then the only edges of $\Bts$ that $\pi_0$ can take are edges of $\gamma$, $\pi'$ and some edges of $B_\infty(\cC(T),\lll)$. Furthermore, since by Lemma \ref{l: Properties of the forbidden zone.}, $\gamma$ does not take any edge of the forbidden zone, if $\pi_0$ links two vertices of $\gamma$ without taking edges of $\pi'$, then $\pi_0$ is exactly $\gamma$ between these two vertices.
		
		Now, assume that $\pi_0$ does not take any pattern entirely contained in $\Bts$.
		Since $\pi_0$ can not take edges of $B_\infty(\cC(T),\lll)$ without taking the pattern centered in $\cC(T)$, it remains to prove that $\pi_0$ does not take any edge of $\pi'$. But since $\pi'$ is a self-avoiding path entirely contained in $\Bds$ which takes the pattern centered in $\cC(T)$, and which has only two vertices in common with $\gamma$, if $\pi_0$ takes an edge of $\pi'$, $\pi_0$ takes the pattern centered in $\cC(T)$, which is impossible.		 
	\end{proof}

	\begin{proof}[Proof of \ref{Cofi}]
		Let $\tpi$ be a $k$-penalized path from $0$ to $x$ in the environment $\Tu$ with finite passage time in the environment $\Tu$. There are three cases. 
		
		\textbf{First case.} If $\tpi$ does not take edges of $\Bts$, $\Tu(\tpi)=T(\tpi)$ since the only edges whose passage time have been modified are edges of $\Bts$. Property \ref{Cofi} follows from \ref{Cofa}. 
		
		\textbf{Second case.} Assume that we are in the case \ref{c: Case III.} and that $\tpi$ takes an edge $e' \in \Bts$ such that $\Tu(e') \ge \nu_1$. Then, 
		\[\Tu(\tpi) = \sum_{e \in \tpi} \Tu(e) = \sum_{e \in \tpi \cap \Bts} \Tu(e) + \sum_{e \in \tpi \cap \Bts^c} \Tu(e).\]Since $e' \in \tpi \cap \Bts$ and since $\Bts$ is a typical box, we have using the fourth property of a typical box \[\sum_{e \in \tpi \cap \Bts} \Tu(e) \ge \nu_1 > \sum_{e \in \Bts} T(e) \ge \sum_{e \in \tpi \cap \Bts} T(e).\]
		Furthermore, the passage times of the edges outside $\Bts$ have not been modified. Hence, \[\sum_{e \in \tpi \cap \Bts^c} \Tu(e) = \sum_{e \in \tpi \cap \Bts^c} T(e).\]
		Thus, 
		\begin{equation}
			\Tu(\tpi) = \sum_{e \in \tpi \cap \Bts} \Tu(e) + \sum_{e \in \tpi \cap \Bts^c} \Tu(e) > \sum_{e \in \tpi \cap \Bts} T(e) + \sum_{e \in \tpi \cap \Bts^c} T(e) = T(\tpi). \label{eq: 1 preuve Cofi.}
		\end{equation}
		In this case, property \ref{Cofi} follows from \eqref{eq: 1 preuve Cofi.} and from property \ref{Cofa}.
		
		\textbf{Third case.}
		Now assume that $\tpi$ takes at least one edge in $\Bts$ and that in the case \ref{c: Case III.}, $\tpi$ does not take any edge in $\Bts$ having a passage time greater than or equal to $\nu_1$. Since $\tpi$ has a finite passage time in the environment $\Tu$, in the case \ref{c: Case I.}, $\tpi$ does not take any edge $e'$ in $\Bts$ such that $\Tu(e') = \infty$. Since $\tpi$ is a $k$-penalized path, it does not take any pattern entirely contained in $\Bts$. Hence, using \ref{Cofhg}, the only edges in $\Bts$ that $\tpi$ can take are edges of $\gamma$. So, let $u_1,v_1,\dots,u_\kappa,v_\kappa$ be the successive entry and exit points of $\tpi$ in $\Bts$, we get for all $i \in \{1,\dots,\kappa\}$, $\tpi_{u_i,v_i} = \gamma_{u_i,v_i}$.
		Furthermore, we also get that $T(\tpi)$ is finite. Indeed, since the only edges whose passage time have been modified are the edges in $\Bts$ and since $\Tu(\tpi)$ is finite, the only edges with infinite passage time in the environment $T$ that $\tpi$ can take are edges in $\Bts$. But these edges being edges of $\gamma$ which has a finite passage time in the environment $T$, $T(\tpi)$ is finite. 
		
		Thus, using again that the only edges whose passage time have been modified are the edges of $\Bts$, we have 
		\begin{equation}
			T(\tpi) - \Tu(\tpi) = \sum_{i \in \{1,\dots,\kappa\}} T(\tpi_{u_i,v_i}) - \Tu(\tpi_{u_i,v_i}) = \sum_{i \in \{1,\dots,\kappa\}} T(\gamma_{u_i,v_i}) - \Tu(\gamma_{u_i,v_i}) .\label{Equation i 1.}
		\end{equation}
		Now, using \ref{Cofa}, 
		\begin{equation}
			\sum_{i \in \{1,\dots,\kappa\}} T(\gamma_{u_i,v_i}) - \Tu(\gamma_{u_i,v_i}) \le T(\gamma) - \Tu(\gamma).\label{Equation i 2.}
		\end{equation}
		Thus, combining \eqref{Equation i 1.} and \eqref{Equation i 2.}, we also get in this last case that
		\[T(\tpi) - \Tu(\tpi) \le T(\gamma) - \Tu(\gamma).\]
	\end{proof}
	
	\subsubsection{End of the proof of Lemma \ref{Lemme pour les inégalités avec les indicatrices dans la modification.}}\label{Sous-section fin de la preuve du lemme.}
	
	We prove Lemma \ref{Lemme pour les inégalités avec les indicatrices dans la modification.} with the sets $\Em(T)$, $\Emid(T)$, $\Ep(T)$ and $\Einf(T)$ and the vertex $\cC(T)$ defined in Section \ref{Sous-section modification}.
	Let us first prove item (i) of this lemma. $\Emid(T)$ is contained in $\Bts$ by property \ref{Cofj en premier} in Section \ref{Section conséquences de la modification.}. Using property \ref{Cofc} of Section \ref{Section conséquences de la modification.}, we get that $B_\infty(\cC(T),\lll)$ is contained in $\Bts$. $\Em(T)$, $\Einf(T)$ and $\Ep(T)$ are contained in $\Bts$ by their definitions. To get that these sets are pairwise disjoint in both cases, we only have to prove that:
	\begin{itemize}
		\item $\gamma_{u_\pi,v_\pi}$ does not visit any edge of $B_\infty(\cC(T),\lll)$. This comes from the fact that $B_\infty(\cC(T),\lll) \subset B_\infty(\cC(T),r_P)$, which comes from the fact that $r_P > \lll+1$ by \eqref{On fixe rP.}. By property (iii) of a typical box, in the case \ref{c: Case I.}, in the environment $T$, for every edge $e \in B_\infty(\cC(T),r_P)$, $T(e) = \infty$ but $T(\gamma_{u_\pi,v_\pi})<\infty$. Hence, $\gamma_{u_\pi,v_\pi}$ does not take any edge in $B_\infty(\cC(T),r_P)$. In the case \ref{c: Case III.}, by property (iii) of a typical box, the event $\theta_{\cC(T)} \cT$ holds with respect to the environment $T$. Recall Remark \ref{Remarque pas de motif dans une boîte typique.}: there is no pattern in $\Bts$. Since $\cT$ satisfies the second condition in Lemma \ref{Lemme pour les clusters dans les typical boxes dans le cas finin non borné.} and since $\gamma_{u_\pi,v_\pi}$ is a $k$-geodesic, $\gamma_{u_\pi,v_\pi}$ does not visit any edge of the forbidden zone. So $\gamma_{u_\pi,v_\pi}$ does not visit any edge of $B_\infty(\cC(T),\lll)$.
		\item $\gamma_{u_\pi,v_\pi}$ and $\pi'$ do not have any edge in common by the definition of $u_\pi$, $v_\pi$ and $\pi'$.
	\end{itemize}
	To get item (ii), fix $\eta = \P (T \in \cA^\Lambda) \tilde{p}^{|\Bts|}$, where, in the case \ref{c: Case I.} and if $\r=0$,
	\[\tilde{p} = \min(\L([\r,\r+\delta')),\L((\r+\delta,\nu_0)),\L(\infty)),\]
	in the case \ref{c: Case I.} and if $\r>0$,
	\[\tilde{p} = \min(\L([\r,\r+\delta']),\L(\infty)),\]
	in the case \ref{c: Case III.} and if $\r=0$,
	\[\tilde{p} = \min(\L([\r,\r+\delta')),\L((\r+\delta,\nu_0)),\L((\nu_2,\infty))),\]
	and in the case \ref{c: Case III.} and if $\r>0$,
	\[\tilde{p} = \min(\L([\r,\r+\delta')),\L((\nu_2,\infty))),\]
	Thus, $\eta$ only depends on $\L$, the pattern and $N$ and we have that
	\[\P\left(T' \in \cB^*(T) | T\right) \ge \P (T \in \cA^\Lambda) \tilde{p}^{|\Bts|} = \eta.\]
	Now, let us prove item (iii) of Lemma \ref{Lemme pour les inégalités avec les indicatrices dans la modification.}. Let $\gu$ be the selected $k$-geodesic in the environment $\Tu$ if it exists. The aim is to prove the following properties in the environment $\Tu$:
	\begin{enumerate}[label=(C\arabic*)]
		\item\label{C1} $\gu$ exists, i.e.\ there exists a $k$-geodesic having at least $Q_n$ boxes in its $S^k$-sequence,
		\item\label{C2} $\gu$ does not have a shortcut in the first $\l-1$ boxes of its $S^k$-sequence,
		\item\label{C3} $\gu$ has a shortcut in the $\l$-th box of its $S^k$-sequence,
		\item\label{C4} $\Bts$ is the $\l$-th box of the $S^k$-sequence of $\gu$. 
	\end{enumerate}

	To get these four properties, we use the following ones:
	\begin{enumerate}[label=(P\arabic*)]
		\item\label{P1} a $k$-box different from $\Bts$ is a typical box in the environment $T$ if and only if it is a typical box in the environment $\Tu$,
		\item\label{P4} a path has a shortcut in a $k$-box different from $\Bts$ in the environment $T$ if and only if it has a shortcut in this box in the environment $\Tu$,  
		\item\label{P2} $\gamma$ has a shortcut in $\Bts$ in the environment $\Tu$ and then, $\Bts$ is successful in the environment $\Tu$ for $\gamma$,
		\item\label{P3} $\gu$ exists and $\gu=\gamma$.
	\end{enumerate}

	To conclude the proof, we have to prove \ref{P1}, \ref{P4}, \ref{P2} and \ref{P3}. Indeed, by \ref{P1} and \ref{P4}, a $k$-box different from $\Bts$ is successful for a path $\tpi$ in the environment $T$ if and only if it is successful for $\tpi$ in the environment $\Tu$. Thus, since $\Bts$ is a typical box for $\gamma$ in the environment $T$, using \ref{P2}, the successful boxes crossed by $\gamma$ in the environments $T$ et $\Tu$ are the same. Furthermore, if we have \ref{P3}, then $\gamma$ is a $k$-geodesic in the environment $\Tu$ and we can define its $S^k$-sequence in this environment. We get that the $S^k$-sequence of $\gamma$ is the same in the environments $T$ and $\Tu$. 
	Hence, using again \ref{P3}, we get \ref{C1}, we get \ref{C2} by \ref{H2} and we get \ref{C3} and \ref{C4} by \ref{H3} using again \ref{P2}.
	
	At this stage of the proof, we easily get \ref{P1}, \ref{P4} and \ref{P2} (which is the aim of the following paragraph) but the proof of \ref{P3} is a bit longer (this is the aim of Section \ref{Paragraphe P3.}).
	
	\paragraph{Proof of properties \ref{P1}, \ref{P4} and \ref{P2}.}
	
	We get \ref{P1} using that the fact that a box is typical only depends on the edges of the box (by Lemma \ref{Lemme propriétés boîtes typiques.}), that every $k$-box different from $\Bts$ does not have edges in common with $\Bts$ and that the only edges whose time has been modified are edges belonging to $\Bts$. 
	
	\ref{P4} uses the same arguments than above. The fact that a path has a shortcut in a $k$-box only depends on the edges of the box. 
	
	We get \ref{P2} by considering the path $\pi'$ defined at the beginning of Section \ref{Sous-section modification}. By construction $\pi'$ is entirely contained in $\Bts$, $\pi'$ and $\gamma$ only have $u_\pi$ and $v_\pi$ in common, $\pi'$ takes the pattern and $\Tu(\gamma_{u_\pi,v_\pi} \cap \Bts) > \Tu(\pi')$ by \ref{Cofb}.
	
	\paragraph{Proof of \ref{P3}: $\gamma$ is the selected $k$-geodesic in the environment $\Tu$.}\label{Paragraphe P3.}
	
	To prove this property, we prove the following ones in the last four lemmas of this section:
	\begin{itemize}
		\item $\gamma$ is a $k$-penalized path in the environment $\Tu$,
		\item every $k$-penalized path from $0$ to $x$ in the environment $\Tu$ has a passage time greater than or equal to the passage time of $\gamma$,
		\item if a path is a $k$-geodesic from $0$ to $x$ in the environment $\Tu$, it is also a $k$-geodesic in the environment $T$,
		\item if a $k$-geodesic from $0$ to $x$ in the environment $\Tu$ has at least $Q_n$ boxes in its $S^k$-sequence, it has also at least $Q_n$ boxes in its $S^k$-sequence in the environment $T$. 
	\end{itemize}
	
	We can conclude with these properties. Indeed, with the first two properties above, $\gamma$ is a $k$-geodesic in the environment $\Tu$. As a consequence of \ref{P1}, \ref{P4} and \ref{P2}, it has the same $S^k$-sequence in the environments $T$ and $\Tu$. Thus $\gamma$ has at least $Q_n$ boxes in its $S^k$-sequence and it can be the selected $k$-geodesic. By the last two properties above, we have that the set of the $k$-geodesics having at least $Q_n$ boxes in their $S^k$-sequences in the environment $\Tu$ is included in the set of the $k$-geodesics having at least $Q_n$ boxes in their $S^k$-sequences in the environment $T$. Since $\gamma$ is the first path in the lexicographical order among the paths of this last set, it is also the first path in the lexicographical order in the first set.  
	
	It remains to prove the four properties above. Before proving them, we begin by the following lemma. 
	
	\begin{lemma}\label{Lemme pour dire qu'un chemin $k$-pénalié dans Tu est quand pénalisé dans T.}
		A $k$-penalized path in the environment $\Tu$ with finite passage time in the environment $\Tu$ is also a $k$-penalized path in the environment $T$.
	\end{lemma}
	
	\begin{proof}
		Let $\tpi$ be a $k$-penalized path in the environment $\Tu$ with $\Tu(\tpi)$ finite. 
		Since the edges outside $\Bts$ have not been modified, $\tpi$ takes a pattern entirely contained in a $k$-box different from $\Bts$ in the environment $T$ if and only if it takes a pattern entirely contained in this box in the environment $\Tu$. It remains to prove that $\tpi$ does not take a pattern entirely contained in $\Bts$ in the environment $T$.
		
		In the case \ref{c: Case I.}, since the time of $\tpi$ is finite and since $\tpi$ is a $k$-penalized path in the environment $\Tu$, by property \ref{Cofhg}, the only edges of $\Bts$ that $\tpi$ can take are edges of $\gamma$. It implies that, in the environment $T$, if $\tpi$ takes a pattern entirely contained in $\Bts$, $\gamma$ also takes this pattern, which is impossible since $\gamma$ is a $k$-penalized path in the environment $T$.
		
		In the case \ref{c: Case III.}, it is impossible since $\Bts$ is a typical box in the environment $T$ and there is no pattern in a typical box by Remark \ref{Remarque pas de motif dans une boîte typique.}. 
	\end{proof}
	
	\begin{lemma}\label{Lemme pour dire que gamma est un chemin k-pénalisé dans l'environnement Tu.}
		$\gamma$ is a $k$-penalized path in the environment $\Tu$. 
	\end{lemma}

	\begin{proof}
		The fact that a path takes a pattern entirely contained in a $k$-box only depends on the passage times of the edges of this $k$-box. Since $\gamma$ is a $k$-penalized path in the environment $T$, it does not take a pattern entirely contained in a $k$-box in this environment. Since the edges of the $k$-boxes different from $\Bts$ have not been modified, $\gamma$ does not take a pattern entirely contained in a $k$-box different from $\Bts$ in the environment $\Tu$. 
		To conclude, it remains to prove that $\gamma$ does not take a pattern entirely contained in $\Bts$ in the environment $\Tu$. By \ref{Cofde}, there is only one pattern entirely contained in $\Bts$ which is the one centered in $\cC(T)$. By Lemma \ref{l: Properties of the forbidden zone.}, $\gamma$ does not take any edge of the forbidden zone and the pattern centered in $\cC(T)$ is entirely contained in the forbidden zone, which gives the result.
	\end{proof}

	\begin{lemma}\label{Lemme pour dire que gamma est une k-géodésique dans l'environnment modifié.}
		Every $k$-penalized path from $0$ to $x$ in the environment $\Tu$ has a passage time for $\Tu$ greater than or equal to the passage time of $\gamma$ for $\Tu$.
	\end{lemma}

	\begin{proof}
		Let $\ogu$ be a $k$-penalized path in the environment $\Tu$.
		By Lemma \ref{Lemme pour dire qu'un chemin $k$-pénalié dans Tu est quand pénalisé dans T.}, $\ogu$ is also a $k$-penalized path in the environment $T$. 
		Thus, since $\ogu$ is a $k$-penalized path in the environment $T$ and since $\gamma$ is a $k$-geodesic, we get $T(\gamma) \le T(\ogu)$. Hence, using \ref{Cofi} in Section \ref{Section conséquences de la modification.}, \[\Tu(\ogu) \ge \Tu(\gamma) + \underbrace{T(\ogu) - T(\gamma)}_{\ge 0} \ge \Tu(\gamma).\]
	\end{proof}

	\begin{lemma}
		If a path from $0$ to $x$ is a $k$-geodesic in the environment $\Tu$, it is also a $k$-geodesic in the environment $T$. 
	\end{lemma}

	\begin{proof}
		Let $\ogu$ be a $k$-geodesic in the environment $\Tu$ from $0$ to $x$. By Lemma \ref{Lemme pour dire que gamma est un chemin k-pénalisé dans l'environnement Tu.}, $\gamma$ is a $k$-penalized path in the environment $\Tu$. Thus, since $\Tu(\gamma)$ is finite, $\Tu(\ogu)$ is also finite and by Lemma \ref{Lemme pour dire qu'un chemin $k$-pénalié dans Tu est quand pénalisé dans T.}, $\ogu$ is a $k$-penalized path in the environment $T$. Moreover, using \ref{Cofi} in Section \ref{Section conséquences de la modification.}, we get 
		\[T(\ogu) \le T(\gamma) + \Tu(\ogu) - \Tu(\gamma).\] Since $\ogu$ is a $k$-geodesic, by  Lemma \ref{Lemme pour dire que gamma est un chemin k-pénalisé dans l'environnement Tu.} and Lemma \ref{Lemme pour dire que gamma est une k-géodésique dans l'environnment modifié.}, $\Tu(\ogu) = \Tu(\gamma)$. So $T(\ogu) \le T(\gamma)$ and $\ogu$ is a $k$-geodesic in the environment $T$.  
	\end{proof}

	\begin{lemma}
		If a $k$-geodesic from $0$ to $x$ in the environment $\Tu$ has at least $Q_n$ boxes in its $S^k$-sequence, it has also at least $Q_n$ boxes in its $S^k$-sequence in the environment $T$.
	\end{lemma}

	\begin{proof}
		Let $\ogu$ be a $k$-geodesic in the environment $\Tu$ from $0$ to $x$ having at least $Q_n$ boxes in its $S^k$-sequence. Using \ref{P4} and by the construction of the $S^k$-sequence, each box different from $\Bts$ belonging to the $S^k$-sequence of $\ogu$ in the environment $\Tu$ belongs to its $S^k$-sequence in the environment $T$.  If $\Bts$ belongs to its $S^k$-sequence in the environment $\Tu$, since $\Bts$ is a typical box in the environment $T$, then $\Bts$ belongs to its $S^k$-sequence in the environment $T$, which allows us to conclude. 
	\end{proof}

	\section{Extension of the van den Berg-Kesten comparison principle}\label{Section preuve VdB-K.}
	
	This section is dedicated to the proof of Theorem \ref{Théorème VdB-K.}.
	Let $\L$ and $\Lt$ be two distributions taking values in $[0,\infty]$ such that:
	\begin{enumerate}[label=($\mathcal{H}$\arabic*)]
		\item $\L$ is useful,
		\item $\L([0,\infty)) > p_c$ and $\Lt([0,\infty)) > p_c$,
		\item $\L \ne \Lt$,
		\item there exists a couple of random variables $(\tau, \tilde{\tau})$ on some probability space, with marginal distributions $\L$ and $\Lt$, respectively, and satisfying 
		\begin{equation}
			\E [\tilde{\tau} | \tau] \le \tau. \label{eq: preuve extension de VdB-K-1.}
		\end{equation}
	\end{enumerate}

	In what follows, $(\tau,\tilde{\tau})$ is a couple of random variables with marginal distributions $\L$ and $\Lt$, and satisfying \eqref{eq: preuve extension de VdB-K-1.}. Such a couple exists by \ref{enum: VdB-K 3.}. Note that by \eqref{eq: preuve extension de VdB-K-1.}, we have 
	\begin{equation}
		\{\tau<\infty\} \subset \{\tilde{\tau}<\infty\} \text{ a.s.} \label{eq: section 3 eq1.}
	\end{equation}
	Then, we consider a family $(T,\tilde{T})=\{(T(e),\tilde{T}(e)) \, : \, e \in \cE\}$ of i.i.d.\ random variables defined on the same probability space such that for all $e \in \cE$, $(T(e),\tilde{T}(e))$ has the same distribution as $(\tau,\tilde{\tau})$. 
	
	The proof of Theorem \ref{Théorème VdB-K.} is an application of Theorem \ref{Théorème général.}. We begin by defining a valid pattern in Section \ref{Sous-section VdB-K définition du motif valable.} and then, we apply Theorem \ref{Théorème général.} with this pattern in Section \ref{Sous-section vdbk fin de la preuve.}.
	 
	\subsection{Definition of the valid pattern}\label{Sous-section VdB-K définition du motif valable.}
	
	The fact that a pattern is valid or not depends on the distribution of the passage times of the environment. Here, we use Theorem \ref{Théorème général.} only in the environment $T$. Thus, when we define a pattern $\mathfrak{P}=(\Lambda,u^\Lambda,v^\Lambda,\cA^\Lambda)$ below, we consider that the event $\cA^\Lambda$ only depends on the family $(T(e))_{e \in \Lambda}$. 
	
	Now, for a valid pattern $\mathfrak{P}=(\Lambda,u^\Lambda,v^\Lambda,\cA^\Lambda)$, denote by $\Pi^\mathfrak{P}$ the set of all self-avoiding paths going from $\ulm$ to $\vlm$ and which are contained in $\Lambda$.
	Denote by $\cG$ the $\sigma$-field generated by the family $(T(e))_{e\in\cE}$. 
	Section \ref{Sous-section VdB-K définition du motif valable.} is devoted to the proof of the following lemma.
	
	\begin{lemma}\label{l: 3 preuve vdbk lemme fondamental des moitfs}
		There exist a valid pattern $\mathfrak{P}=(\Lambda,u^\Lambda,v^\Lambda,\cA^\Lambda)$ and a constant $\eta>0$ such that on the event $\cA^\Lambda$,
		\begin{equation}
			\E \left[ \min_{\pi \in \Pi^\mathfrak{P}} \tilde{T}(\pi) | \cG \right] < \min_{\pi \in \Pi^\mathfrak{P}} T(\pi) - \eta.\label{eq: 1 preuve vdbk inégalité avec les minima des espérances conditionnelles.}
		\end{equation}
	\end{lemma}

	To prove Lemma \ref{l: 3 preuve vdbk lemme fondamental des moitfs}, there are three different cases to be considered. Noting that, if $\P(\tilde{\tau}<\infty \text{ and } \tau=\infty)=0$, using \eqref{eq: section 3 eq1.}, we get \[\{\tau<\infty\} = \{\tilde{\tau}<\infty\} \text{ a.s.},\] these three cases can be written as follows:
	\begin{itemize}
		\item $\P(\tilde{\tau}<\infty \text{ and } \tau=\infty)>0$,
		\item $\{\tau<\infty\} = \{\tilde{\tau}<\infty\} \text{ a.s.}$ and $\P(\E[\tilde{\tau}|\tau]<\tau)>0$,
		\item $\{\tau<\infty\} = \{\tilde{\tau}<\infty\} \text{ a.s.}$ and $\P(\E[\tilde{\tau}|\tau]=\tau)=1$.
	\end{itemize}
	The most technical case is the third one.
	
	\subsubsection{First case: when $\P(\tilde{\tau}<\infty \text{ and } \tau=\infty)>0$}
	
	\begin{proof}[Proof of Lemma \ref{l: 3 preuve vdbk lemme fondamental des moitfs} in the first case]
		\,

		Assume that $\P(\tilde{\tau}<\infty \text{ and } \tau=\infty)>0$. 
		Let $\nu \in (0,\infty)$ such that 
		\begin{equation}
			\L([\nu,\infty))>0.\label{eq: premier cas vdbk}
		\end{equation} Such a constant exists since we have $\L(0)+\L(\infty)<1$ since $\L$ is useful by Assumption \ref{enum: VdB-K 4.} and $\L([0,\infty))>p_c$ by Assumption \ref{enum: VdB-K 1.}.
		Let $M \in [0,\infty)$ such that 
		\begin{equation}
			\P(\tilde{\tau}<M \text{ and } \tau=\infty)>0.\label{eq: 1 preuve vdbk cas 1.}
		\end{equation} Let $\eta_0>0$. Fix 
		\begin{equation}
			m > \frac{M+\eta_0}{2 \nu}.\label{eq: 2 preuve vdbk cas 1.}
		\end{equation}
		Then, define the pattern $\mathfrak{P}=(\Lambda,u^\Lambda,v^\Lambda,\cA^\Lambda)$ where $\Lambda = \{0,1\} \times \{0,\dots,m\} \times \prod_{j=3}^d \{0\}$, $\ulm=(0,\dots,0)$, $\vlm=(1,0,\dots,0)$ and $\cA^\Lambda$ is defined as follows. Denote by $\pi^{\text{fin}}$ the path going from $\ulm$ to $m\epsilon_2$ by $m$ steps in the direction $\epsilon_2$, then to $\epsilon_1+m\epsilon_2$ by one step in the direction $\epsilon_1$ and finally to $\vlm_2$ by $m$ steps in the direction $\epsilon_2$. The event $\cA^\Lambda$ is the event on which for every $e \in \pi^\text{fin}$, $\nu \le T(e) < \infty$ and for every edge $e \in \Lambda$ but not in $\pi^\text{fin}$, $T(e) = \infty$. This pattern is valid since the event $\cA^\Lambda$ has a positive probability by \eqref{eq: premier cas vdbk} and \eqref{eq: 1 preuve vdbk cas 1.}, and since the path $\pi^\text{fin}$ is a path between $\ulm$ and $\vlm$ with a finite passage time in the environment $T$ when the event $\cA^\Lambda$ occurs.
		
		Now, denote \[\beta = \P \left( \tilde{\tau} < M | \tau = \infty \right).\] By \eqref{eq: 1 preuve vdbk cas 1.}, $\beta>0$. Then, on the event $\cA^\Lambda$, we have
		\begin{equation}
			\E \left[ \min_{\pi \in \Pi^\mathfrak{P}} \tilde{T}(\pi) | \cG \right] \le T(\pi^\text{fin}) - \beta (T(\pi^\text{fin})-M).\label{eq: corrections vdbk équation avec le beta.}
		\end{equation}
		Indeed, denote by $e^{\infty}$ the edge $\{\ulm,\vlm\}$ and by $\pi^\infty$ the path going from $\ulm$ to $\vlm$ by taking only this edge. Then, on the event $\cA^\Lambda$, since $\cA^\Lambda \subset \{T(e^\infty) = \infty\}$,
		\begin{align}
			\E \left[ \min_{\pi \in \Pi^\mathfrak{P}} \tilde{T}(\pi) | \cG \right] & \le \E \left[ \min(\tilde{T}(\pi^\text{fin}), \, \tilde{T}(\pi^\infty)) | \cG \right] \nonumber \\
			& \le \E \left[M \1_{\{\tilde{T}(e^\infty)<M\}} + \tilde{T}(\pi^\text{fin}) \1_{\{\tilde{T}(e^\infty) \ge M\}} | \cG \right] \nonumber \\
			& \le M \beta + \sum_{e \in \pi^\text{fin}} \E \left[ \tilde{T}(e) \1_{\tilde{T}(e^\infty) \ge M} | \cG \right]. \label{eq: vdbk motif 1 correction 1.}
		\end{align}
		But, for every $e \in \pi^\text{fin}$,
		\begin{align}
			\E \left[ \tilde{T}(e) \1_{\tilde{T}(e^\infty) \ge M} | \cG \right] & = \E \left[ \tilde{T}(e) | T(e) \right] \P\left(\tilde{T}(e^\infty) \ge M | T(e^\infty)\right) \nonumber \\
			& \le \E \left[ \tilde{T}(e) | T(e) \right] (1-\beta) \text{ since $\cA^\Lambda \subset \{T(e^\infty) = \infty\}$,} \nonumber \\
			& \le T(e) (1-\beta),\label{eq: vdbk motif 1 correction 2.}
		\end{align}
		since $(T(e),\tilde{T}(e))$ has the same distribution as $(\tau,\tilde{\tau})$ which satisfies \eqref{eq: preuve extension de VdB-K-1.}.
		Thus, combining \eqref{eq: vdbk motif 1 correction 1.} and \eqref{eq: vdbk motif 1 correction 2.}, we get, on the event $\cA^\Lambda$, 
		\begin{align*}
			\E \left[ \min_{\pi \in \Pi^\mathfrak{P}} \tilde{T}(\pi) | \cG \right] & \le M \beta + (1- \beta)T(\pi^\text{fin}) = T(\pi^\text{fin}) - \beta (T(\pi^\text{fin})-M),
		\end{align*}
		and \eqref{eq: corrections vdbk équation avec le beta.} is proved.
		
		Now, by the definition of the pattern, on the one hand, $\displaystyle T(\pi^\text{fin}) = \min_{\pi \in \Pi^\mathfrak{P}} T(\pi)$ and on the other hand, $T(\pi^\text{fin}) \ge 2m\nu$. This gives, using \eqref{eq: 2 preuve vdbk cas 1.}, 
		\[T(\pi^\text{fin}) - M > \eta_0.\]
		Hence, 
		\[\E \left[ \min_{\pi \in \Pi^\mathfrak{P}} \tilde{T}(\pi) | \cG \right] < \min_{\pi \in \Pi^\mathfrak{P}} T(\pi) - \beta \eta_0,\]
		which allows us to conclude since $\beta \eta_0 > 0$.
	\end{proof}
	
	\subsubsection{Second case: when $\{\tau<\infty\} = \{\tilde{\tau}<\infty\} \text{ a.s.}$ and $\P(\E[\tilde{\tau}|\tau]<\tau)>0$}
	
	\begin{proof}[Proof of Lemma \ref{l: 3 preuve vdbk lemme fondamental des moitfs} in the second case]
		\,
		
		Assume that $\P(\E[\tilde{\tau}|\tau]<\tau)>0$ and that $\{\tau<\infty\} = \{\tilde{\tau}<\infty\}$ a.s.
		Then there exist $\eta>0$ and a Borel set $I \subset [0,\infty)$ such that $\P(\tau \in I) > 0$ and on the event $\{\tau \in I\}$, 
		\begin{equation}
			\E[\tilde{\tau}|\tau] < \tau - \eta. \label{eq: section 3 eq2.}
		\end{equation}
		Now, define the pattern $\mathfrak{P}=(\{u^\Lambda,v^\Lambda\},u^\Lambda,v^\Lambda,\cA^\Lambda)$ where $u^\Lambda=(0,\dots,0)$, $v^\Lambda=(1,0,\dots,0)$ and $\cA^\Lambda$ is the event on which the passage time of the only edge of the pattern, denoted by $e$, belongs to $I$. 
		Then, this pattern is valid since the event $\cA^\Lambda$ has a positive probability since $\P(\tau \in I)>0$ and since the passage time of the path $(\ulm,\vlm)$ in the environment $T$ is finite when $\cA^\Lambda$ occurs since $I \subset [0,\infty)$. Furthermore, on the event $\cA^\Lambda$,
		\begin{align*}
			\E \left[ \min_{\pi \in \Pi^\mathfrak{P}} \tilde{T}(\pi) | \cG \right] & = \E \left[\tilde{T}(e) | \cG\right] = \E \left[\tilde{T}(e) | T(e) \right] < T(e) - \eta \text{ by \eqref{eq: section 3 eq2.},} \\
			& = \min_{\pi \in \Pi^\mathfrak{P}} T(\pi) - \eta.
		\end{align*}
	\end{proof}
	
	\subsubsection{Third case: when $\{\tau<\infty\} = \{\tilde{\tau}<\infty\} \text{ a.s.}$ and $\P(\E[\tilde{\tau}|\tau]=\tau)=1$}
	
	Assume that 
	\begin{equation}
		\P(\E[\tilde{\tau}|\tau]=\tau)=1, \label{eq: Correction vdbk hypothèse dans la preuve du troisième motif.}
	\end{equation} and that
	\begin{equation}
		\{\tau<\infty\} = \{\tilde{\tau}<\infty\} \text{ a.s.} \label{eq: preuve extension de VdB-K-3.}
	\end{equation}
	
	\begin{lemma}\label{l: lemme préliminaire VdB-K pour avoir ensuite le lemme avec l'espérance conditionnelle.}
		In this case, there exist $\beta>0$ and $\delta>0$ such that
		\begin{equation}
			\P \biggl( \tau < \infty \text{ and } \P(\tilde{\tau} \le \tau - 2 \delta |\tau) \ge \beta \text{ and } \P(\tilde{\tau}\ge\tau|\tau) \ge \beta \biggr) > 0.
			\label{eq: section 3 eq3.}
		\end{equation}
	\end{lemma}

	\begin{proof}
		Denote $A=\{\tau < \infty \text{ and } \P(\tilde{\tau}<\tau|\tau)>0\}$ and $B=\{\P(\tilde{\tau}\ge\tau|\tau)>0\}$. To prove the lemma, it suffices to prove that
		\begin{equation}
			\P(A \cap B)>0.\label{eq: preuve VdBK lemme avec les événements A et B.}
		\end{equation}
		Let $K$ be a transition probability kernel such that for every measurable function $\phi \, : \, [0,\infty] \times [0,\infty] \to [0,\infty]$, 
		\begin{equation}
			\E \left[ \phi(\tau,\tilde{\tau}) | \tau \right] = \int_{[0,\infty]} \phi(\tau,\tilde{t}) K(\tau,d \tilde{t}).\label{eq: Cinlar 1.}
		\end{equation}
		The existence of such a $K$ is given by Theorem 2.19 in Chapter 4 in \cite{Cinlar}.
	
		First, let us prove that $\P(B) = 1$. We have
		\begin{align*}
			B^c & = \left\{ \E \left[ \1_{\tilde{\tau} < \tau} | \tau \right] = 1 \right\} = \left\{ K(\tau,[0,\tau)) = 1 \right\} \text{ using \eqref{eq: Cinlar 1.},} \\
			& = \left\{ K(\tau,[0,\tau)) = 1 \text{ and } \tau < \infty \right\} \text{ using \eqref{eq: preuve extension de VdB-K-3.}} \\
			& \subset \left\{ \int_{[0,\infty]} \tilde{t} K(\tau,d\tilde{t}) < \tau \right\} \\
			& = \left\{ \E \left[ \tilde{\tau} | \tau \right] < \tau \right\}.
		\end{align*}
		Now, since $\P ( \E \left[ \tilde{\tau} | \tau \right] < \tau ) = 0$ by \eqref{eq: Correction vdbk hypothèse dans la preuve du troisième motif.}, we get $\P(B^c) = 0$ and thus $\P(B) = 1$.
		
		So, to get \eqref{eq: preuve VdBK lemme avec les événements A et B.}, it remains to prove that $\P(A)>0$. To this aim, we shall prove that 
		\begin{equation}
			A^c = \{\tau = \tilde{\tau}\} \text{ a.s.,}\label{eq: Correction vdbk preuve du troisième motif 2.}
		\end{equation}
		which leads to $\P(A^c) < 1$ by \ref{enum: VdB-K 2.}, and thus $\P(A)>0$. To prove \eqref{eq: Correction vdbk preuve du troisième motif 2.}, observe that
		\begin{align*}
			A^c & = \{\tau = \infty \text{ or } \P(\tilde{\tau} < \tau | \tau) = 0\} \\
			& = \{\tau = \infty\} \cup \{\tau < \infty \text{ and } \P (\tilde{\tau} < \tau | \tau) = 0\}, \\
			& = \{\tau = \infty\} \cup \{\tau < \infty \text{ and } K(\tau,[\tau,\infty])=1\} \text{ using \eqref{eq: Cinlar 1.},} \\
			& = \{\tau = \infty\} \cup \{\tau < \infty \text{ and } K(\tau,\{\tau\})=1\},
		\end{align*}
		since \[\{\tau < \infty, \, K(\tau,[\tau,\infty]) = 1 \text{ and } K(\tau,(\tau,\infty])>0\} \subset \{\E[\tilde{\tau} | \tau] > \tau\}\]
		and $\P \left( \E[\tilde{\tau} | \tau] > \tau \right) = 0$ by \eqref{eq: Correction vdbk hypothèse dans la preuve du troisième motif.}.
		Thus, using again \eqref{eq: Cinlar 1.}, we get
		\begin{align*}
			A^c & = \{\tau = \infty\} \cup \{\tau < \infty \text{ and } \P(\tilde{\tau} = \tau | \tau) = 1\} \\
			& = \{\tau = \infty\} \cup \{\tau < \infty \text{ and } \tilde{\tau} = \tau\} \text{ a.s.}
		\end{align*}
		Now, since $\{\tau = \infty\} = \{\tau = \tilde{\tau} = \infty\}$ a.s.\ by \eqref{eq: preuve extension de VdB-K-3.}, we get \eqref{eq: Correction vdbk preuve du troisième motif 2.}.
		Hence, \eqref{eq: preuve VdBK lemme avec les événements A et B.} holds and the lemma is proved.
	\end{proof}

	\begin{lemma}\label{l: lemme VdB-K avec le min des espérances conditionnelles.}
		In this case, there exists a bounded Borel set $I \subset (0,\infty)$ and $\eta > 0$ such that 
		\begin{itemize}
			\item $\P(\tau \in I) > 0$,
			\item and \[\E[\min(\tilde{\tau}_1+\tilde{\tau}_2,\tilde{\tau}_3+\tilde{\tau}_4)|\cF] < \min(\tau_1+\tau_2,\tau_3+\tau_4) - \eta\] on the event $\cI = \{\tau_1,\dots,\tau_4 \in I\}$, where $(\tau_1,\tilde{\tau}_1),\dots,(\tau_4,\tilde{\tau}_4)$ are independent copies of $(\tau,\tilde{\tau})$ and $\cF = \sigma(\tau_1,\dots,\tau_4)$.
		\end{itemize}
	\end{lemma}

	\begin{proof}
		Let $(\tau_1,\tilde{\tau}_1),\dots,(\tau_4,\tilde{\tau}_4)$ be independent copies of $(\tau,\tilde{\tau})$. Denote $\cF = \sigma(\tau_1,\dots,\tau_4)$.
		Fix $\beta>0$ and $\delta>0$ given by Lemma \ref{l: lemme préliminaire VdB-K pour avoir ensuite le lemme avec l'espérance conditionnelle.}. 
		Using \eqref{eq: section 3 eq3.}, we can find a Borel set $I_0 \subset (0,\infty)$, fixed for the remaining of the proof, such that $\P(\tau \in I_0)>0$ and on the event $\{\tau \in I_0\}$, 
		\begin{equation}
			\tau < \infty \text{ and } \P(\tilde{\tau} \le \tau - 2 \delta |\tau) \ge \beta \text{ and } \P(\tilde{\tau}\ge\tau|\tau) \ge \beta.\label{eq: preuve vdbk à propos du beta.}
		\end{equation}
		Now, fix $\eta>0$ such that $\eta < 2 \delta \beta^4$ and then fix
		\begin{equation}
			0 < \delta_0 < \frac{2\delta \beta^4-\eta}{4}.\label{eq: preuve vdbk diamètre de I.}
		\end{equation}
		Note that this gives $\displaystyle \delta_0 < \frac{\delta}{2}$ as $\beta \le 1$.
		Let $y_0 \in I_0$ such that $\P(\tau \in I_0 \cap (y_0-\delta_0,y_0+\delta_0))>0$. Such a $y_0$ exists because $\P(\tau \in I_0)>0$. Set $I = I_0 \cap (y_0-\delta_0,y_0+\delta_0)$. 
		
		Then, we have \[\min(\tilde{\tau}_1+\tilde{\tau}_2,\tilde{\tau}_3+\tilde{\tau}_4) + 2 \delta \1_{\{\tilde{\tau}_3 \le \tilde{\tau}_1-\delta \text{ and } \tilde{\tau}_4 \le \tilde{\tau}_2 - \delta\}} \le \tilde{\tau}_1 + \tilde{\tau}_2.\]
		Thus, on the event $\cI=\{\tau_1,\dots,\tau_4 \in I\}$, since $(\tau_1,\tilde{\tau}_1)$ and $(\tau_2,\tilde{\tau}_2)$ satisfy \eqref{eq: preuve extension de VdB-K-1.},
		\begin{align*}
			\E[\min(\tilde{\tau}_1+\tilde{\tau}_2,\tilde{\tau}_3+\tilde{\tau}_4)|\cF] & \le \E[\tilde{\tau}_1+\tilde{\tau}_2|\cF] - 2\delta \P(\tilde{\tau}_3 \le \tilde{\tau}_1 - \delta | \cF) \P(\tilde{\tau}_4 \le \tilde{\tau}_2 - \delta | \cF) \\
			& \le \tau_1 + \tau_2 - 2\delta \P(\tilde{\tau}_3 \le \tilde{\tau}_1 - \delta | \cF) \P(\tilde{\tau}_4 \le \tilde{\tau}_2 - \delta | \cF) \\
			& \le 2 \sup I - 2 \delta \P (\tilde{\tau}_3 \le \tau_3 - 2 \delta \text{ and } \tilde{\tau}_1 \ge \tau_1 | \cF) \P (\tilde{\tau}_4 \le \tau_4 - 2 \delta \text{ and } \tilde{\tau}_2 \ge \tau_2 | \cF),
		\end{align*}	
		since the diameter of $I$ is lower than $\delta$ by \eqref{eq: preuve vdbk diamètre de I.}. Then, by \eqref{eq: preuve vdbk à propos du beta.}
		\begin{align*}
			\E[\min(\tilde{\tau}_1+\tilde{\tau}_2,\tilde{\tau}_3+\tilde{\tau}_4)|\cF] & \le 2 \sup I - 2 \delta \beta^4 \\
			& < 2 \inf I - \eta,
		\end{align*}
		since $2(\sup I - \inf I) \le 4 \delta_0 < 2 \delta \beta^4 - \eta$ by \eqref{eq: preuve vdbk diamètre de I.}.
		Hence, since on the event $\cI$,
		\[\min(\tau_1+\tau_2,\tau_3+\tau_4) \ge 2 \inf I,\]
		we get, on the event $\cI$, \[\E[\min(\tilde{\tau}_1+\tilde{\tau}_2,\tilde{\tau}_3+\tilde{\tau}_4)|\cF] < \min(\tau_1+\tau_2,\tau_3+\tau_4) - \eta.\]
	\end{proof}

	\begin{proof}[Proof of Lemma \ref{l: 3 preuve vdbk lemme fondamental des moitfs} in the third case.]
		Fix $I \subset (0,\infty)$ and $\eta>0$ given by Lemma \ref{l: lemme VdB-K avec le min des espérances conditionnelles.}. Define the pattern $\mathfrak{P}=(\Lambda,u^\Lambda,v^\Lambda,\cA^\Lambda)$ where $\Lambda$ is the set containing 4 vertices defined by $\Lambda = \{0,1\} \times \{0,1\} \times \prod_{j=3}^d \{0\}$, $\ulm_1=(0,\dots,0)$, $\vlm_1=(1,1,0,\dots,0)$ and $\cA^\Lambda$ is the event on which for every edge $e \in \Lambda$, $T(e) \in I$. This pattern is valid since the event $\cA^\Lambda$ has a positive probability since $\P(\tau \in I)>0$ by Lemma \ref{l: lemme VdB-K avec le min des espérances conditionnelles.}, and since on the event $\cA^\Lambda$, every edge of the pattern has a finite passage time.
		
		Now denote $e_1=\{\ulm,\ulm+\epsilon_1\}$, $e_2=\{\ulm+\epsilon_1,\vlm\}$, $e_3=\{\ulm,\ulm+\epsilon_2\}$ and $e_4=\{\ulm+\epsilon_2,\vlm\}$. Since there are only two paths in $\Pi^\mathfrak{P}$, the one taking $e_1$ and then $e_2$ and the one taking $e_3$ and then $e_4$, we get
		\begin{align*}
			\E \left[ \min_{\pi \in \Pi^\mathfrak{P}} \tilde{T}(\pi) | \cG \right] & = \E \left[ \min\left( \tilde{T}(e_1) + \tilde{T}(e_2), \tilde{T}(e_3) + \tilde{T}(e_4) \right)| \cG \right] \\
			& = \E \left[ \min\left( \tilde{T}(e_1) + \tilde{T}(e_2), \tilde{T}(e_3) + \tilde{T}(e_4) \right)| \sigma(T(e_1),T(e_2),T(e_3),T(e_4)) \right] \\
			& < \min(T(e_1)+T(e_2), T(e_3)+T(e_4)) - \eta \text{ by Lemma \ref{l: lemme VdB-K avec le min des espérances conditionnelles.},} \\
			& = \min_{\pi \in \Pi^\mathfrak{P}} T(\pi) - \eta.
		\end{align*}
	\end{proof}

	\subsection{End of the proof of Theorem \ref{Théorème VdB-K.}}\label{Sous-section vdbk fin de la preuve.}
	
	By \ref{enum: VdB-K 1.} and \eqref{eq: section 3 eq1.}, we can fix $M > 0$ such that 
	\begin{equation}
		\P (\tau \le M \text{ and } \tilde{\tau} \le M) > p_c.\label{eq: on fixe M dans la preuve de VDBK.}
	\end{equation}
	Let $\cC$ and $\tilde{\cC}$ be the clusters defined respectively as the clusters $\cC_M$ and $\tilde{\cC}_M$ in Section \ref{Sous-section résultat principal et applications.} for $M$ fixed above. Recall the definitions of $\phi$, $\tilde{\phi}$, $\mu$ and $\tilde{\mu}$ also given in Section \ref{Sous-section résultat principal et applications.}, and the convergence given at \eqref{eq: extension de VdB-K-3.}.

	Let $x \in \Z^d \setminus \{0\}$.
	For any $n \in \N$, we define the random path $\gamma_n$ as the first geodesic in the lexicographical order from $\phi(0)$ to $\phi(nx)$ in the environment $T$. As stated in Section \ref{Sous-section Settings.}, almost surely, there exists at least one geodesic from $\phi(0)$ to $\phi(nx)$. Recall that $\cG$ is the $\sigma$-field generated by the family $(T(e))_{e\in\cE}$. Note that $\gamma_n$ is $\cG$-measurable.
	
	\begin{lemma}\label{l: lemme intégrabilité des temps tilde pour le phi pas tilde.}
		We have $\E \left[ \tilde{t} (\phi(0),\phi(x)) \right] < \infty$. 
	\end{lemma}

	\begin{proof}
		Since $\gamma_1$ is $\cG$-measurable,
		\begin{align*}
			\E \left[ \tilde{t}(\phi(0),\phi(x)) | \cG \right] & \le \E \left[ \tilde{T} (\gamma_1) | \cG \right] \\
			& = \sum_{e \in \gamma_1} \E \left[ \tilde{T}(e) | T(e) \right] \\
			& \le \sum_{e \in \gamma_1} T(e) = T(\gamma_1) = t(\phi(0),\phi(x)).
		\end{align*}
		Note that the last inequality comes from the fact that for every $e \in \cE$, $\E \left[ \tilde{T}(e) | T(e) \right] \le T(e)$.
		Then, taking expectation, it gives 
		\[\E \left[ \tilde{t}(\phi(0),\phi(x)) \right] \le \E \left[ t(\phi(0),\phi(x)) \right].\]
		Since by Proposition 2 in \cite{CerfTheret}, $\E \left[ t(\phi(0),\phi(x)) \right] < \infty$, we get $\E \left[ \tilde{t}(\phi(0),\phi(x)) \right] < \infty$.
	\end{proof}
	
	With Lemma \ref{l: lemme intégrabilité des temps tilde pour le phi pas tilde.}, we can use the Subadditive Ergodic Theorem (see for example Theorem 2.2 in \cite{50years}), and thus we get the existence of a finite constant $\overline{\mu}(x)$ such that 
	\[\lim\limits_{n \to \infty} \frac{\tilde{t}(\phi(0),\phi(nx))}{n} = \overline{\mu}(x) \text{ a.s. and in $L^1$}.\]
	
	\begin{remk}
		There are two main differences between the definitions of $\mu(x)$ and $\tilde{\mu}(x)$: 
		\begin{itemize}
			\item $\mu(x)$ is defined with the passage times in the environment $T$ although $\tilde{\mu}(x)$ is defined with the passage times in the environment $\tilde{T}$,
			\item $\mu(x)$ is defined with geodesic times between vertices obtained with $\phi$ although $\tilde{\mu}(x)$ is defined with geodesic times between vertices obtained with $\tilde{\phi}$.
		\end{itemize}
		In order to compare $\mu(x)$ and $\tilde{\mu}(x)$, it is therefore natural to introduce $\overline{\mu}(x)$, an intermediate object which has one difference with $\mu(x)$ and one difference with $\tilde{\mu}(x)$ in its definition. 
	\end{remk}

	\begin{remk}
		The proof of Theorem \ref{Théorème VdB-K.} given in this section also holds in the case originally proven by van den Berg and Kesten in \cite{VdBK}, i.e.\ when we assume that $\tau$ has a finite first moment. However, it is simpler in this case since we do not need the clusters $\cC$ and $\tilde{\cC}$ to define $\mu(x)$ and $\tilde{\mu}(x)$. In this case, for every $y \in \Z^d$, we can take $\phi(y) = \tilde{\phi}(y)=y$ and $\overline{\mu}(x) = \tilde{\mu}(x)$, and thus in the sequel, we do not need Lemma \ref{l: vdbk lemme tilde plus petit que overline.} and the proof of Lemma \ref{l: vdbk lemme nombre de motifs rencontrés assez grand.} is much simpler. 
	\end{remk}

	The following lemma is based on elementary arguments of percolation.
	
	\begin{lemma}\label{l: vdbk lemme tilde plus petit que overline.}
		We have $\tilde{\mu}(x) \le \overline{\mu}(x)$.
	\end{lemma}

	\begin{proof}
		Let $\widehat{\cC}$ be the infinite cluster for the Bernoulli percolation $(\1_{\{T(e) \le M \text{ and } \tilde{T}(e) \le M\}}, \, e \in \cE)$ which exists and is unique a.s.\ by \eqref{eq: on fixe M dans la preuve de VDBK.}. Note that $\widehat{\cC}$ is included in $\cC$ and in $\tilde{\cC}$. For any $y \in \R^d$, we define $\widehat{\phi}(y)$ the random point of $\widehat{\cC}$ such that $\|y-\widehat{\phi}(y)\|_1$ is minimal, with a deterministic rule to break ties. For any $n \in \Z$, we have
		\begin{align}
			\tilde{t}(\tilde{\phi}(0),\tilde{\phi}(nx)) 
			& \le \tilde{t}(\phi(0),\widehat{\phi}(0)) + \tilde{t}(\widehat{\phi}(0),\tilde{\phi}(0)) + \tilde{t}(\phi(0),\phi(nx)) + \tilde{t}(\phi(nx),\widehat{\phi}(nx)) + \tilde{t}(\widehat{\phi}(nx),\tilde{\phi}(nx)) \nonumber \\
			& = A(0) + \tilde{t}(\phi(0),\phi(nx)) + A(nx), \label{eq: vdbk inégalité preuve tilde plus petit que overline.}
		\end{align}
		by writing for any $y \in \Z^d$, $A(y) = \tilde{t}(\phi(y),\widehat{\phi}(y)) + \tilde{t}(\widehat{\phi}(y),\tilde{\phi}(y))$.
		
		Now, for any $y \in \Z^d$, there exists a path between $\phi(y)$ and $\widehat{\phi}(y)$ contained in $\cC$. Thus, it only takes edges with finite passage times in the environment $T$, and thus with finite passage times in the environment $\tilde{T}$ by \eqref{eq: section 3 eq1.}. Furthermore, there exists a path between $\widehat{\phi}(y)$ and $\tilde{\phi}(y)$ contained in $\tilde{\cC}$, which gives that $\tilde{t}(\widehat{\phi}(y),\tilde{\phi}(y))$ is also finite. Hence $A(y)$ is a finite random variable. It gives that \[\frac{A(0)}{n}\] converges almost surely, and thus in probability, towards $0$. Hence \[\frac{A(nx)}{n}\] also converges towards $0$ in probability. We get that 
		\[\frac{A(0) + \tilde{t}(\phi(0),\phi(nx)) + A(nx)}{n}\]
		converges towards $\overline{\mu}(x)$ in probability. Using \eqref{eq: vdbk inégalité preuve tilde plus petit que overline.}, this gives $\tilde{\mu}(x) \le \overline{\mu}(x)$.
	\end{proof}

	Let $\eta>0$ and $\mathfrak{P}$ the pattern given by Lemma \ref{l: 3 preuve vdbk lemme fondamental des moitfs}. 		
	For any path $\pi$ we denote by $\cN^\mathfrak{P}(\pi)$ the maximum number of disjoint translations of the pattern $\mathfrak{P}$ crossed by $\pi$.

	\begin{lemma}\label{l: vdbk lemme nombre de motifs rencontrés assez grand.}
		There exists a constant $c>0$ such that for any $n$ sufficiently large,
		\[\E \left[ \cN^\mathfrak{P}(\gamma_n) \right] \ge cn.\]
	\end{lemma}

	\begin{proof}
		Recall the definition of $N^\mathfrak{P}(\pi)$ given for any path $\pi$ at \eqref{Compteur nombre de motifs empruntés introduction.}.
		Since the pattern $\mathfrak{P}$ is valid by Lemma \ref{l: 3 preuve vdbk lemme fondamental des moitfs}, $\L$ is useful by \ref{enum: VdB-K 4.} and $\L([0,\infty))>p_c$ by \ref{enum: VdB-K 1.}, we can use Theorem \ref{Théorème général.} and thus there exist $\alpha>0$, $\beta_1>0$ and $\beta_2>0$ such that for every $n \in \N$ and for every $\displaystyle y \in B_1 \left( 0, \frac{\|nx\|_1}{4} \right)$ and $\displaystyle z \in B_1 \left( nx, \frac{\|nx\|_1}{4} \right)$,
		\begin{equation}
			\P \left((y,z) \in \mathfrak{C} \text{ and } \exists \mbox{ a geodesic $\gamma$ from $y$ to $z$ such that } N^\mathfrak{P}(\gamma) < \alpha n \right) \le \beta_1 \mathrm{e}^{-\beta_2 n}. \label{eq: preuve vdbk nombre de motifs rencontrés assez grand.}
		\end{equation}
		Denote by $\cV_n$ the set $\displaystyle B_1 \left( 0, \frac{\|nx\|_1}{4} \right) \times B_1 \left( nx, \frac{\|nx\|_1}{4} \right)$. Then, for every $n \in \N$,
		\begin{align}
			\P \left( N^\mathfrak{P}(\gamma_n) < \alpha n \right) \le \P & \left( (\phi(0),\phi(nx)) \notin \cV_n \right) \nonumber \\ & + \sum_{(y,z) \in \cV_n} \P \left( (y,z) \in \mathfrak{C} \text{ and } \exists \mbox{ a geodesic $\gamma$ from $y$ to $z$ such that } N^\mathfrak{P}(\gamma) < \alpha n  \right). \label{eq: preuve vdbk nombre de motifs rencontrés assez grand 2.}
		\end{align}
		Since $x \ne 0$, by Theorem 8 in \cite{CerfTheret}, there exist $\beta_3>0$ and $\beta_4>0$ such that for every $n \in \N$, 
		\[\P \left( (\phi(0),\phi(nx)) \notin \cV_n \right) \le \beta_3 \mathrm{e}^{-\beta_4 n}.\]
		Hence, using \eqref{eq: preuve vdbk nombre de motifs rencontrés assez grand.} and since $|\cV_n|$ is bounded by a polynomial in $n$, there exist $\beta_5>0$ and $\beta_6>0$ such that for every $n \in \N$, 
		\[\P \left( N^\mathfrak{P}(\gamma_n) < \alpha n \right) \le \beta_5 \mathrm{e}^{-\beta_6 n}.\]
		Thus, there exists $c' > 0$ such that for any $n$ sufficiently large
		\[\E \left[ N^\mathfrak{P}(\gamma_n) \right] \ge c' n.\]
		We conclude the proof by observing that simple geometric considerations provide a constant $c''>0$ such that for all path $\pi$, 
		\[\cN^\mathfrak{P}(\pi) \ge c'' N^\mathfrak{P}(\pi).\]
	\end{proof}

	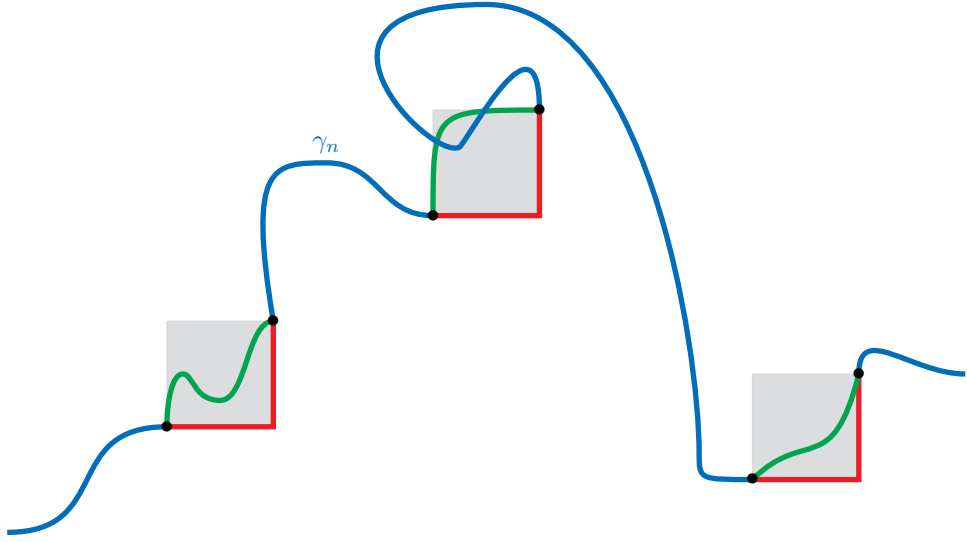
\begin{figure}
		\begin{center}
			\begin{tikzpicture}[scale=0.7]
				\draw[gray!30,fill=Gray!30] (5,5) rectangle (7,7);
				\draw[gray!30,fill=Gray!30] (10,9) rectangle (12,11);
				\draw[gray!30,fill=Gray!30] (16,4) rectangle (18,6);
				\draw[Red,line width=2pt] (5,5) -- (7,5) -- (7,7);
				\draw[Red,line width=2pt] (10,9) -- (12,9) -- (12,11);
				\draw[Red,line width=2pt] (16,4) -- (18,4) -- (18,6);
				\draw[line width=2pt,Green] (5,5) .. controls (5,5.5) and (5.1,6) .. (5.3,6) .. controls (5.5,6) and (5.5,5.5) .. (6,5.5) .. controls (6.5,5.5) and (6.5,7) .. (7,7);
				\draw[line width=2pt,Green] (10,9) .. controls (10,11) and (10,11) .. (12,11);
				\draw[line width=2pt,Green] (16,4) .. controls (17,5) and (17.5,4) .. (18,6);
				\draw[line width=2pt,NavyBlue] (2,3) .. controls (4,3) and (3,5) .. (5,5);
				\draw[line width=2pt,NavyBlue] (7,7) .. controls (6.5,10) and (7,10) .. (8,10) .. controls (9,10) and (9,9) .. (10,9);
				\draw[line width=2pt,NavyBlue] (12,11) .. controls (12,13) and (10.7,10.5) .. (10.5,10.3) .. controls (10,10) and (7,13) .. (11,13) .. controls (14,13) and (15,6.5)  .. (15,4.5) .. controls (15,4) and (15,4) .. (16,4);
				\draw[line width=2pt,NavyBlue] (18,6) .. controls (18,7) and (19,6) .. (20,6);
				\draw (5,5) node {$\bullet$};
				\draw (7,7) node {$\bullet$};
				\draw (10,9) node {$\bullet$};
				\draw (12,11) node {$\bullet$};
				\draw (16,4) node {$\bullet$};
				\draw (18,6) node {$\bullet$};
				\draw (8,10) node[NavyBlue,above] {$\gamma_n$};
			\end{tikzpicture}
			\caption{Illustration of the proof of Lemma \ref{l: vdbk lemme overline plus petit que sans rien.}. The translations of the pattern $\mathfrak{P}$ are in gray. The path $\gamma_n$ is the concatenation of the blue parts and of the red parts. In each pattern, the green part corresponds to the optimal path in the environment $\tilde{T}$ among the paths between the endpoints of the pattern and entirely contained in the pattern. The idea is to bound from above the geodesic time between $\phi(0)$ and $\phi(nx)$ in the environment $\tilde{T}$ by the passage time in the environment $\tilde{T}$ of the concatenation of the blue parts and the green parts.}\label{f: Preuve VDBK.}
		\end{center}
	\end{figure}

	\begin{lemma}\label{l: vdbk lemme overline plus petit que sans rien.}
		We have $\overline{\mu} (x) < \mu(x)$.
	\end{lemma}

	\begin{proof}
		For any $n \ge 1$, denote by $\cS^\mathfrak{P}(\gamma_n)$ the set, chosen according to a deterministic rule if there are several such sets, of $\cN^\mathfrak{P}(\gamma_n)$ disjoint translations of $\mathfrak{P}$ crossed by $\gamma_n$.
		Denote by $\cE^\mathfrak{P}(\gamma_n)$ the set of all edges of $\gamma_n$ which are not in a subpath of $\gamma_n$ between the endpoints of a pattern of $\cS^\mathfrak{P}(\gamma_n)$. 
		Recall that we denote by $\Pi^\mathfrak{P}$ the set of all self-avoiding paths going from $\ulm$ to $\vlm$ and which are contained in $\Lambda$.
		For a pattern $\mathfrak{P}' \in \cS^\mathfrak{P}(\gamma_n)$, we can associate a unique $s \in \Z^d$ such that $s$ satisfies the condition $(\gamma_n;\mathfrak{P})$ and such that $\mathfrak{P}'$ is located at $\theta_s \Lambda$. Then, we denote by $\Pi^{\mathfrak{P}'}$ the set of all self-avoiding paths $\pi$ such that $\theta_{-s} \pi \in \Pi^\mathfrak{P}$.
		Denote by $\Gamma^\mathfrak{P}(\gamma_n)$ the set of all paths from $\phi(0)$ to $\phi(nx)$ following $\gamma_n$ outside all the patterns of $\cS^\mathfrak{P}(\gamma_n)$ and following a path of $\Pi^{\mathfrak{P}'}$ for every pattern $\mathfrak{P}'$ of $\cS^\mathfrak{P}(\gamma_n)$. 
		With these definitions, we immediately get that 
		\begin{equation}
			\min_{\pi \in \Gamma^\mathfrak{P}(\gamma_n)} \tilde{T}(\pi) = \sum_{e \in \cE^\mathfrak{P}(\gamma_n)} \tilde{T}(e) + \sum_{\mathfrak{P}' \in \cS^\mathfrak{P}(\gamma_n)} \min_{\pi' \in \Pi^{\mathfrak{P}'}} \tilde{T}(\pi').\label{eq: vdbk fin de preuve.}
		\end{equation}
		
		Let $n$ be sufficiently large such that Lemma \ref{l: vdbk lemme nombre de motifs rencontrés assez grand.} holds and let $c>0$ be the constant given by this lemma.
		Recall that $\cG$ is the $\sigma$-field generated by the family $(T(e))_{e\in\cE}$. Then, $\gamma_n$, $\cS^\mathfrak{P}(\gamma_n)$ and $\cE^\mathfrak{P}(\gamma_n)$ are $\cG$-measurable and we get 
		\begin{align*}
			\E \left[ \tilde{t}(\phi(0),\phi(nx)) | \cG \right] & \le \E \left[ \min_{\pi \in \Gamma^\mathfrak{P}(\gamma_n)} \tilde{T}(\pi) | \cG \right] \text{ by the definition of the geodesic time in the environment $\tilde{T}$,} \\
			& = \sum_{e \in \cE^\mathfrak{P}(\gamma_n)} \E \left[ \tilde{T}(e) | T(e) \right] + \sum_{\mathfrak{P}' \in \cS^\mathfrak{P}(\gamma_n)} \E \left[ \min_{\pi' \in \Pi^{\mathfrak{P}'}} \tilde{T}(\pi') | \cG \right] \text{ by \eqref{eq: vdbk fin de preuve.},} \\
			& \le \sum_{e \in \cE^\mathfrak{P}(\gamma_n)} T(e) + \sum_{\mathfrak{P}' \in \cS^\mathfrak{P}(\gamma_n)} \left( \min_{\pi' \in \Pi^{\mathfrak{P}'}} T(\pi') - \eta \right).
		\end{align*}
		For the first sum, the last inequality comes from the fact that for every edge $e \in \cE$, $(T(e),\tilde{T}(e))$ has the same distribution as $(\tau,\tilde{\tau})$ and thus satisfies $E[\tilde{T}(e)|T(e)] \le T(e)$. For the second sum, it comes from Lemma \ref{l: 3 preuve vdbk lemme fondamental des moitfs}.
		
		Then, by the definitions of $\cE^\mathfrak{P}(\gamma_n)$ and $\cS^\mathfrak{P}(\gamma_n)$, \[\sum_{e \in \cE^\mathfrak{P}(\gamma_n)} T(e) + \sum_{\mathfrak{P}' \in \cS^\mathfrak{P}(\gamma_n)} \min_{\pi' \in \Pi^{\mathfrak{P}'}} T(\pi') = T(\gamma_n) = t(\phi(0),\phi(nx)).\]
		Furthermore, recall that $\cN^\mathfrak{P}(\gamma_n)$ is the number of elements of $\Pi^{\mathfrak{P}'}$. Thus, we get 
		\[\E \left[ \tilde{t}(\phi(0),\phi(nx)) | \cG \right] \le t(\phi(0),\phi(nx)) - \eta \cN^\mathfrak{P}(\gamma_n).\]
		Now, taking expectation and dividing by $n$ gives
		\[\frac{\E \left[ \tilde{t}(\phi(0),\phi(nx)) \right]}{n} \le \frac{\E \left[ t(\phi(0),\phi(nx)) \right]}{n} - \eta \frac{\E \left[ \cN^\mathfrak{P}(\gamma_n) \right]}{n} \le \frac{\E \left[ t(\phi(0),\phi(nx)) \right]}{n} - \eta c,\]
		by Lemma \ref{l: vdbk lemme nombre de motifs rencontrés assez grand.}. 
		We conclude the proof using that 
		\[\lim\limits_{n \to \infty} \frac{\E \left[ \tilde{t}(\phi(0),\phi(nx)) \right]}{n} = \overline{\mu}(x) \text{ and } \lim\limits_{n \to \infty} \frac{\E \left[ t(\phi(0),\phi(nx)) \right]}{n} = \mu(x).\]
	\end{proof}

	Now, we conclude the proof of Theorem \ref{Théorème VdB-K.} by combining Lemma \ref{l: vdbk lemme tilde plus petit que overline.} and Lemma \ref{l: vdbk lemme overline plus petit que sans rien.}.

	\appendix
	
	\section{Existence of geodesics}\label{Annexe sur l'existence des géodésiques avec des arêtes infinies.}
	
	\begin{prop}\label{prop: Proposition annexe existence des géodésiques.}
		Assume that $\L(0) < p_c$. With probability one, for all $x$, $y$ such that $(x,y) \in \mathfrak{C}$, there exists a geodesic between $x$ and $y$. 
	\end{prop}

	To prove the above proposition, we begin by the following lemma. 
	
	\begin{lemma}\label{l: Lemme annexe existence des géodésiques.}
		Assume that $\L(0) < p_c$. There exists $\beta>0$, $\beta' > 0$ and $\rho > 0$ such that for all $n \ge 1$,
		\[\P(\exists \text{ a self-avoiding path $\pi$ from $0$ which contains at least $n$ edges but has } T(\pi) < \rho n) \le \beta' \mathrm{e}^{-\beta n}.\]
	\end{lemma}

	\begin{proof}[Proof of Lemma \ref{l: Lemme annexe existence des géodésiques.}]
		For each environment $T$, we define a new environment $\tilde{T}$ defined for all edges $e$ by 
		\[\tilde{T}(e) = \left\{
		\begin{array}{ll}
			T(e) & \mbox{if } T(e) < \infty \\
			1 & \mbox{else.}
		\end{array}
		\right.\]
		Since $\L(0)<p_c$, we have $\P(\tilde{T}(e)=0) = \P(T(e) = 0) < p_c$. Thus, we can use Proposition (5.8) in \cite{SaintFlourKesten} and we get $\beta > 0$, $\beta' > 0$ and $\rho > 0$ such that for all $n \ge 1$, for all $z \in \Z^d$,
		\[\P(\exists \text{ a self-avoiding path $\pi$ from $z$ which contains at least $n$ edges but has } \tilde{T}(\pi) < \rho n) \le \beta' \mathrm{e}^{-\beta n}.\]
		Now, for every edge $e$, $\tilde{T}(e) \le T(e)$. Therefore, 
		\begin{equation*}
			\begin{split}
				\P(\exists \text{ a self-avoiding path $\pi$ from $0$ which contains at least $n$ edges but has } T(\pi) < \rho n) \\ 
				\le \P(\exists \text{ a self-avoiding path $\pi$ from $z$ which contains at least $n$ edges but has } \tilde{T}(\pi) < \rho n),
			\end{split}
		\end{equation*}
		which allows us to conclude.
	\end{proof}

	\begin{proof}[Proof of Proposition \ref{prop: Proposition annexe existence des géodésiques.}]
		It is sufficient to prove that for every $x$ and $y$, with probability one, there exists a geodesic between $x$ and $y$ if $(x,y) \in \mathfrak{C}$. Fix $x$ and $y$ in $\Z^d$. 
		Fix $\beta$, $\beta'$ and $\rho$ given by Lemma \ref{l: Lemme annexe existence des géodésiques.}. For every $n \ge 1$, denote by $A_n$ the event on which  every path $\pi$ from $x$ which contains at least $n$ edges has $T(\pi) \ge \rho n$.
		By the Borel-Cantelli Lemma and by Lemma \ref{l: Lemme annexe existence des géodésiques.}, with probability one, for all $n$ large enough, $A_n$ occurs. We work on this probability one event. Assume that $(x,y) \in \mathfrak{C}$.
		Let $\pi_{x,y}$ be a path between $x$ and $y$ such that $T(\pi_{x,y})<\infty$. Fix $n$ large enough such that $A_n$ occurs and 
		\begin{equation}
			n > \frac{T(\pi_{x,y})}{\rho}.\label{eq: Annexe existence des géodésiques preuve de la proposition.}
		\end{equation}
		Then, every path from $x$ to the boundary of $B_1(x,n)$ has a passage time greater than or equal to $\rho n$ since the event $A_n$ occurs, and thus a passage time strictly greater than $T(\pi_{x,y})$ by \eqref{eq: Annexe existence des géodésiques preuve de la proposition.}.
		Hence, the infimum in the definition of $t(x,y)$ is over the finite set of paths contained in $B_1(x,n)$, and there must be a geodesic between $x$ and $y$.
	\end{proof}
	
	\section{Edges with positive passage times taken by self-avoiding paths}\label{Annexe sur le lemme qui découle de Kesten Saint-Flour (5.8).}
	
	\begin{proof}[Proof of Lemma \ref{l: Lemme qui dit qu'il y a assez d'arêtes non nulles.}]
		Assume that $\L$ is useful and that $\r=0$. Let $\tau>0$ such that $\L([0,\tau])<p_c$. For each environment $T$, we define a new environment $\tilde{T}$ defined for all edges $e$ by 
		\[\tilde{T}(e) = \left\{
		\begin{array}{ll}
			0 & \mbox{if } T(e) \le \tau \\
			1 & \mbox{else.}
		\end{array}
		\right.\]
		We have $\P(\tilde{T}(e)=0) = \P(T(e) \le \tau) < p_c$.
		By Proposition (5.8) in \cite{SaintFlourKesten}, we get $\beta > 0$, $\beta' > 0$ and $\rho > 0$ such that for all $n \ge 1$,
		\[\P(\exists \text{ a self-avoiding path $\pi$ from $0$ which contains at least $n$ edges but has } \tilde{T}(\pi) < \rho n) \le \beta' \mathrm{e}^{-\beta n}.\]
		Thus, 
		\begin{equation*}
			\begin{split}
				\P(\exists \text{ a self-avoiding path $\pi$ from $0$ which contains at least $n$ edges but containes at most} \\ \rho n \text{ edges $e$ such that $T(e) > \tau$}) \le \beta' \mathrm{e}^{-\beta n},
			\end{split}
		\end{equation*}
		and we get \eqref{eq: Equation du lemme qui dit qu'il y a assez d'arêtes non nulles.} for all $v,w \in \Z^d$.
	\end{proof}
	
	\section{Overlapping patterns}\label{Annexe sur les surmotifs.}
	
		\begin{proof}[Proof of Lemma \ref{Lemme pour dire qu'on ne perd pas de généralité avec les restrictions sur le motifs.}]
		Let $\mathfrak{P}_0=(\Lambda_0,\ulm_0,\vlm_0,\cA^\Lambda_0)$ be a valid pattern. Denote by $L_1,\dots,L_d$ the integers such that $\displaystyle \Lambda_0=\prod_{i=1}^d \{0,\dots,L_i\}$. Fix 
		\begin{equation}
			\lll=\max(L_1,\dots,L_d)+4. \label{eq: hypothèse sur le motif 1.}
		\end{equation} 
		Let $M^\Lambda_0>0$ such that 
		\begin{equation}
			\P(\cA^\Lambda_0 \cap \{\forall e \in \Lambda_0, \, T(e) \le M^\Lambda_0 \text{ or } T(e)=\infty\})>0. \label{eq: lemme hypotèse sur le motif.}
		\end{equation}
		
		\textbf{In the case \ref{c: Case I.}.}
		Consider the pattern $\mathfrak{P}=(\Lambda,\ulm,\vlm,\cA^\Lambda)$ defined as follows:
		\begin{itemize}
			\item $\Lambda=B_\infty(0,\lll)$.
			\item $\ulm = - \lll \epsilon_1$ and $\vlm = \lll \epsilon_1$.
			\item Let $\pi_\infty$ be a path from $\ulm$ to $\vlm$ such that:
			\begin{itemize}
				\item $\pi_\infty$ is a self-avoiding path.
				\item In $\Lambda \setminus B_\infty(0,\lll-3)$, $\pi_\infty$ uses only 6 vertices, all in the set $\{k\epsilon_1, \, -\lll \le k \le \lll\}$,
				\item $\pi_\infty$ visits $\ulm_0$ and $\vlm_0$, and the portion of $\pi_\infty$ between these two vertices, denoted by $\pi_{\infty,0}$ is entirely contained in $\Lambda_0$. Furthermore, when $\cA^\Lambda_0$ occurs, $T(\pi_{\infty,0}) < \infty$. Note that this is possible since $\mathfrak{P}_0$ is valid.
				\item $\pi_\infty \setminus \pi_{\infty,0}$ does not take any edge of $\Lambda_0$.
			\end{itemize}
			Then, $\cA^\Lambda$ is the event such that:
			\begin{itemize}
				\item $\cA^\Lambda_0 \cap \{\forall e \in \Lambda_0, T(e) \le M^\Lambda_0 \text{ or } T(e)=\infty\}$ occurs,
				\item for all $e$ belonging to $\cS_{s,\lll}$, $T(e) < \infty$,
				\item for all $e$ belonging to $\pi_\infty \setminus \pi_{\infty,0}$, $T(e) \le M^\Lambda_0$,
				\item for all $e$ which does not belong to $\Lambda_0 \cup \pi_\infty$, $T(e) = \infty$. 
			\end{itemize}
		\end{itemize}
		We get that $\mathfrak{P}$ satisfies \ref{OP1}, \ref{OP2} and \ref{OP3}. Then, since $\pi_\infty$ takes only edges whose passage time is smaller than or equal to $M^\Lambda_0$, $\mathfrak{P}$ satisfies \ref{OP7} by taking $T^\Lambda > |B_\infty(0,\lll)|_e M^\Lambda_0$. 
		
		Furthermore, we have $\cA^\Lambda \subset \cA^\Lambda_0$ by the definition of $\cA^\Lambda$, $\Lambda_0 \subset \Lambda$ by \eqref{eq: hypothèse sur le motif 1.}, $\cA^\Lambda$ has a positive probability by \eqref{eq: lemme hypotèse sur le motif.} and when $\cA^\Lambda$ occurs, every path from $\ulm$ to $\vlm$ whose passage time is finite is equal to $\pi_\infty \setminus \pi_{\infty,0}$ outside $\Lambda_0$, visits $\ulm_0$ and $\vlm_0$ and is entirely contained in $\Lambda_0$. 
		
		\textbf{In the case \ref{c: Case III.}.} 
		Let $M^\Lambda> |B_\infty(0,\lll)|_e M^\Lambda_0 + 1$ such that 
		\begin{equation}
			\L((M^\Lambda-1,M^\Lambda))>0. \label{eq: lemme hypothèse sur le motif 2.}
		\end{equation}
		Consider the pattern $\mathfrak{P}=(\Lambda,u^\Lambda,v^\Lambda,\cA^\Lambda)$ defined as follows:
		\begin{itemize}
			\item $\Lambda=B_\infty(0,\lll)$.
			\item $\ulm = - \lll \epsilon_1$ and $\vlm = \lll \epsilon_1$.
			\item Let $\pi_f$ be a path from $\ulm$ to $\vlm$ such that:
			\begin{itemize}
				\item $\pi_f$ is a self-avoiding path.
				\item $\pi_f$ does not visit any vertex in $\partial \Lambda$ except $\ulm$ and $\vlm$.
				\item $\pi_f$ visits $\ulm_0$ and $\vlm_0$ and the portion of $\pi_f$ between these two vertices, denoted by $\pi_{f,0}$ is entirely contained in $\Lambda_0$.
				\item $\pi_f \setminus \pi_{f,0}$ does not take any edge of $\Lambda_0$.
			\end{itemize}
			Then $\cA^\Lambda$ is the event such that:
			\begin{itemize}
				\item $\cA^\Lambda_0 \cap \{\forall e \in \Lambda_0, \, T(e) \le M^\Lambda_0\}$ occurs,
				\item for all $e$ belonging to $\pi_f \setminus \pi_{f,0}$, we have $T(e) \le M^\Lambda_0$,
				\item for all $e$ which does not belong to $\partial \Lambda \cup \Lambda_0 \cup \pi_f$, $T(e) \in (M^\Lambda-1,M^\Lambda)$,
				\item for all $e \in \partial \Lambda$, $T(e) \ge M^\Lambda$.
			\end{itemize}
		\end{itemize}
		We get that $\mathfrak{P}$ satisfies \ref{OP4}, \ref{OP8}, \ref{OP5} and \ref{OP6}. Furthermore, 
		\begin{itemize}
			\item $\Lambda_0 \subset \Lambda$ by \eqref{eq: hypothèse sur le motif 1.}.
			\item $\P(\cA^\Lambda)$ is positive by \eqref{eq: lemme hypotèse sur le motif.} and \eqref{eq: lemme hypothèse sur le motif 2.}, and then $\mathfrak{P}$ is a valid pattern.
			\item On $\cA^\Lambda$, any path from $\ulm$ to $\vlm$ optimal for the passage time among the paths entirely inside $\Lambda$ contains a subpath from $\ulm_0$ to $\vlm_0$ entirely inside $\Lambda_0$. Indeed, let $\pi$ be a path from $\ulm$ to $\vlm$ which does not contain a subpath from $\ulm$ to $\vlm$ entirely inside $\Lambda_0$. Since $\pi_f$ is a self-avoiding path, it implies that $\pi$ takes an edge whose time is greater than $M^\Lambda - 1 > |\Lambda|_e M^\Lambda_0$. But we have $T(\pi_f) \le |\Lambda|_e M^\Lambda_0 < T(\pi)$ and thus $\pi$ is not an optimal path. Hence, for every optimal path $\pi$, if a vertex $x \in \Z^d$ satisfies the condition $(\pi;\mathfrak{P}_0)$, $x$ satisfies the condition $(\pi;\mathfrak{P})$. 
			\item $\cA^\Lambda \subset \cA^\Lambda_0$ by the definition of $\cA^\Lambda$.
		\end{itemize}
	\end{proof}


\begin{thebibliography}{1}
	
	\bibitem{AndjelVares}
	Enrique~D. Andjel and Maria~E. Vares.
	\newblock {First passage percolation and escape strategies}.
	\newblock {\em {Random Structures and Algorithms}}, 47(3):414--423, October
	2015.
	
	\bibitem{50years}
	A.~Auffinger, M.~Damron, and J.~Hanson.
	\newblock {\em 50 years of first-passage percolation}.
	\newblock American Mathematical Society, 2017.
	
	\bibitem{CerfTheret}
	Rapha{\"e}l Cerf and Marie Th{\'e}ret.
	\newblock {Weak shape theorem in first passage percolation with infinite
		passage times}.
	\newblock {\em Annales de l'Institut Henri Poincaré, Probabilités et
		Statistiques}, 52(3):1351 -- 1381, 2016.
	
	\bibitem{Cinlar}
	E.~{\c{C}}{\i}nlar.
	\newblock {\em Probability and Stochastics}.
	\newblock Graduate Texts in Mathematics. Springer New York, 2011.
	
	\bibitem{Grimmett}
	Geoffrey Grimmett.
	\newblock {\em Percolation}, volume 321 of {\em Grundlehren der Mathematischen
		Wissenschaften [Fundamental Principles of Mathematical Sciences]}.
	\newblock Springer-Verlag, Berlin, second edition, 1999.
	
	\bibitem{Jacq}
	Antonin Jacquet.
	\newblock Geodesics in first-passage percolation cross any pattern.
	\newblock arXiv 2204.02021, 2023.
	
	\bibitem{SaintFlourKesten}
	Harry Kesten.
	\newblock Aspects of first passage percolation.
	\newblock In P.~L. Hennequin, editor, {\em {\'E}cole d'{\'E}t{\'e} de
		Probabilit{\'e}s de Saint Flour XIV - 1984}, pages 125--264, Berlin,
	Heidelberg, 1986. Springer Berlin Heidelberg.
	
	\bibitem{MarchandStrictInqualities}
	R.~Marchand.
	\newblock {Strict inequalities for the time constant in first passage
		percolation}.
	\newblock {\em The Annals of Applied Probability}, 12(3):1001 -- 1038, 2002.
	
	\bibitem{VdBK}
	J.~van~den Berg and H.~Kesten.
	\newblock {Inequalities for the Time Constant in First-Passage Percolation}.
	\newblock {\em The Annals of Applied Probability}, 3(1):56 -- 80, 1993.
	
\end{thebibliography}
\end{document}